\documentclass[11pt,nosumlimits,fleqn,reqno]{amsart}

\usepackage{amsthm}
\usepackage{amssymb}
\usepackage{amsfonts}
\usepackage{mathtools}
\usepackage{thmtools}
\usepackage{stmaryrd}
\usepackage{mathrsfs}
\usepackage{fullpage}
\usepackage{bbm}
\usepackage{color}
\usepackage{enumitem}
\usepackage{xifthen}
\usepackage{cancel}
\usepackage{hyperref}
\hypersetup{
	colorlinks=true,
	allcolors=niceblue
}
\usepackage[english,noabbrev,capitalise]{cleveref}
\usepackage{tikz}

\renewcommand{\paragraph}[1]{
	
	\noindent{\mbox{\normalfont\normalsize\bfseries #1.}\quad}}
\renewcommand{\subparagraph}[1]{
	
	\noindent{\mbox{\normalfont\normalsize\itshape #1.}\quad}}

\definecolor{nicegreen}{RGB}{2,161,114}
\definecolor{nicered}{RGB}{213,91,0}
\definecolor{niceblue}{RGB}{0,110,180}
\definecolor{niceorange}{RGB}{255,180,0}
\definecolor{nicepink}{RGB}{204,115,168}
\definecolor{nicesky}{RGB}{85,178,235}

\usetikzlibrary{decorations.pathmorphing}

\usepackage[
defernumbers=true,
maxbibnames=6,
backend=biber,
style=alphabetic,
giveninits=true,
sortcites=true,
maxalphanames=4,
minalphanames=3,
doi=true,
isbn=false,
url=false,
safeinputenc,
natbib
]{biblatex}
\theoremstyle{definition}
\newtheorem{thm}{Theorem}[section]
\newtheorem{lmm}[thm]{Lemma}
\newtheorem{crl}[thm]{Corollary}
\newtheorem{prp}[thm]{Proposition}
\newtheorem{fct}[thm]{Fact}

\newtheorem{dfn}[thm]{Definition}

\Crefname{lmm}{Lemma}{Lemmas}
\Crefname{thm}{Theorem}{Theorems}
\Crefname{crl}{Corollary}{Corollaries}
\Crefname{dfn}{Definition}{Definitions}
\Crefname{qst}{Question}{Questions}

\newcommand{\ts}{\textstyle}

\let\phi\varphi
\let\epsilon\varepsilon
\let\emptyset\varnothing

\let\subset\subseteq
\let\supset\supseteq
\let\nsubset\nsubseteq

\let\bar\overline

\newcommand{\cal}{\mathcal}
\newcommand{\bb}{\mathbb} 
 
\renewcommand{\bf}{\mathbf}
\renewcommand{\sf}{\mathsf}

\renewcommand{\rm}{\mathrm}
\newcommand{\fr}{\mathfrak}

\newcommand{\fforall}{\,\forall}
\newcommand{\eexists}{\,\exists}
\newcommand{\emp}{\emptyset}		
\renewcommand{\Cap}{\bigcap}		
\renewcommand{\Cup}{\bigcup}		
	
\newcommand{\md}{\vDash}                    
\newcommand{\fc}{\Vdash}                    
\newcommand{\imp}{\Rightarrow}				
\newcommand{\bimp}{\Leftrightarrow}			

\newcommand{\leftrightdelim}[4][]{%
	\def\leftrightdelimparam{#1}%
	\def\leftrightdelimbig{big}%
	\def\leftrightdelimBig{Big}%
	\ifx\leftrightdelimparam\empty%
	\mathopen{}#2{#4}#3\mathclose{}%
	\else\ifx\leftrightdelimparam\leftrightdelimbig%
	\mathopen{}\big#2{#4}\big#3\mathclose{}%
	\else\ifx\leftrightdelimparam\leftrightdelimBig%
	\mathopen{}\Big#2{#4}\Big#3\mathclose{}%
	\else%
	\mathopen{}\left#2#4\right#3\mathclose{}%
	\fi\fi\fi%
}
\newcommand{\leftrightseparator}[1][]{%
	\def\leftrightseparatorparam{#1}%
	\def\leftrightseparatorbig{big}%
	\def\leftrightseparatorBig{Big}%
	\ifx\leftrightseparatorparam\empty%
	~|~%
	\else\ifx\leftrightseparatorparam\leftrightseparatorbig%
	~\big|~%
	\else\ifx\leftrightseparatorparam\leftrightseparatorBig%
	~\Big|~%
	\else%
	~\middle|~%
	\fi\fi\fi%
}

\newcommand{\rb}[2][]{\leftrightdelim[#1]{(}{)}{#2}}
\newcommand{\ac}[2][]{\leftrightdelim[#1]{\{}{\}}{#2}}
\newcommand{\sqb}[2][]{\leftrightdelim[#1]{[}{]}{#2}}
\newcommand{\card}[2][]{\leftrightdelim[#1]{\lvert}{\rvert}{#2}}

\newcommand{\ab}[2][]{\leftrightdelim[#1]{\langle}{\rangle}{#2}}
\newcommand{\dab}[2][]{\leftrightdelim[#1]{\langle}{\rangle}{\!\ab[#1]{#2}\!}}

\newcommand{\set}[3][]{\ac[#1]{#2\leftrightseparator[#1]#3}}

\newcommand{\abset}[3][]{\ab[#1]{#2\leftrightseparator[#1]#3}}

\newcommand{\iv}[3][]{\leftrightdelim[#1]{[}{)}{#2,\,#3}}
\newcommand{\opiv}[3][]{\leftrightdelim[#1]{(}{)}{#2,\,#3}}

\newcommand{\ap}[1]{\text{``}\,#1\,\text{''}}

\newcommand{\dom}{\operatorname{dom}}
\newcommand{\ran}{\operatorname{ran}}
\newcommand{\ot}{\operatorname{ot}}
\newcommand{\cf}{\operatorname{cf}}
\newcommand{\cov}{\operatorname{cov}}
\newcommand{\non}{\operatorname{non}}
\newcommand{\cof}{\operatorname{cof}}
\newcommand{\add}{\operatorname{add}}

\newcommand{\supp}{\operatorname{supp}}

\newcommand{\Ord}{\bf{Ord}}

\newcommand{\omcs}{{}^{\omega}2}
\newcommand{\fomcs}{{}^{<\omega}2}
\newcommand{\omom}{{}^{\omega}\omega}

\newcommand{\kacs}{{}^{\kappa}2}
\newcommand{\fkacs}{{}^{<\kappa}2}

\newcommand{\kaka}{{}^{\kappa}\kappa}

\newcommand{\sqkaka}{\sqb{\kappa}^\kappa}

\newcommand{\ins}{\in^*}

\newcommand{\eqs}{=^*}

\newcommand{\neqi}{\mathrel{\cancel{=^\infty}}}
\newcommand{\leqs}{\leq^*}

\newcommand{\nleqs}{\mathrel{\cancel{\leq^*}}}

\newcommand{\IPk}{\rm{IP}_\kappa}
\newcommand{\BPk}{\rm{BP}_\kappa}
\newcommand{\IPklim}{\rm{IP}_\kappa^\rm{lim}}
\newcommand{\BPklim}{\rm{BP}_\kappa^\rm{lim}}
\newcommand{\SkI}{\Sigma_\kappa\rb I}
\newcommand{\SkB}{\Sigma_\kappa\rb B}
\newcommand{\NWD}[1][]{
	\def\NWDparam{#1}%
	\ifx\NWDparam\empty%
	\cal{NWD}%
	\else%
	\cal{NWD}\rb[m]{#1}%
	\fi%
}
\newcommand{\Mk}[1][]{%
	\def\Mkparam{#1}%
	\ifx\Mkparam\empty%
	\cal{M}_\kappa%
	\else%
	\cal{M}_\kappa\rb[m]{#1}%
	\fi%
}
\newcommand{\CMk}{\cal{CM}_\kappa}
\newcommand{\SNk}{\cal{SN}_\kappa}
\newcommand{\NAk}{\cal{NA}_\kappa}
\newcommand{\MAk}{\cal{MA}_\kappa}
\newcommand{\Ek}{\cal{E}_\kappa}
\newcommand{\EAk}{\cal{EA}_\kappa}
\newcommand{\EkE}{\cal{E}_\kappa^\exists}
\newcommand{\EkA}{\cal{E}_\kappa^\forall}

\newcommand{\BEk}{\cal{BE}_\kappa}

\newcommand{\BEkE}{\cal{BE}_\kappa^\exists}
\newcommand{\BEkA}{\cal{BE}_\kappa^\forall}

\newcommand{\halffrac}{{}^1\kern-.2em/\!{}_2}

\newcommand{\ft}{\mathbbm 1}

\newcommand{\dins}[1]{\fr d_\kappa^{#1}({\ins})}
\newcommand{\bins}[1]{\fr b_\kappa^{#1}({\ins})}
\newcommand{\Loc}{\bb L\mathbbm{oc}}
\newcommand{\LockB}{\Loc_\kappa(B)}
\newcommand{\Lockn}{\Loc_\kappa^\rm{new}}

\title{The Higher Closed Null Ideal(s)}

\subjclass[2020]{03E05, 03E17, 03E35, 03E55}
\keywords{closed null ideal, F$_\sigma$-null ideal, higher Baire space, generalised Baire space, additive ideal, localisation forcing}

\author[A. Marton]{Adam Marton}
\address{Department of Applied Mathematics and Business Informatics, Faculty of Economics, Technical University of Košice, B. Němcovej 32, 040 01 Košice, Slovakia}
\email{adam.marton@tuke.sk}
\author[J. \v{S}upina]{Jaroslav \v{S}upina}
\address{Institute of Mathematics, Pavol Jozef Šafárik University in Košice, Jesenná 5, 040 01 Košice, Slovakia}
\email{jaroslav.supina@upjs.sk}
\author[M. Repick\'y]{Miroslav Repick\'y}
\address{Mathematical Institute, Slovak Academy of Sciences, Grešákova 6, 040 01 Košice, Slovakia}
\email{repicky@saske.sk}
\author[T. van~der~Vlugt]{Tristan van~der~Vlugt}
\address{Institut für Diskrete Mathematik und Geometrie, Technische Universität Wien, Wiedner Hauptstrasse 8-10/104, 1040 Vienna, Austria}
\curraddr{Kurt Gödel Research Center, Institute of Mathematics, University of Vienna, Kolingasse 14–16, 1090 Vienna, Austria}
\email{tristan@tvdvlugt.nl}

\thanks{J.~\v{S}upina was supported by the Slovak Research and Development Agency under the contract no.\ APVV-25-0144. M.~Repick\'y was supported by Slovak Grant Agency VEGA 2/0104/24. T.~van der Vlugt was funded by the Austrian Science Fund (FWF) [grants PAT8484324 and P33895] during the research and writing of this article.}

\begin{document}

	\begin{abstract}
		We generalise the $\sigma$-ideal $\mathcal E$ generated by the closed Lebesgue null subsets of the Cantor space ${}^\omega 2$ to the context of higher Cantor spaces ${}^\kappa 2$ for an uncountable cardinal $\kappa$ satisfying $\kappa=\kappa^{<\kappa}$. We introduce candidates for a higher closed null ideal based on slaloms on partitions of $\kappa$ into sets of size ${<}\kappa$, and study their combinatorial properties. We also give some relations between the cardinal characteristics of the higher closed null ideals and other cardinal characteristics.
	\end{abstract}
	
	\maketitle

	\section{Introduction}\label{sec:introduction}
	
	In the last two decades, a considerable effort has gone into the generalisation of combinatorics of the Baire space $\omom$ to the context of the \emph{higher} Baire space $\kaka$, where $\kappa$ is regular uncountable. Although the higher context resembles the classical context closely in many cases, it can behave very differently in others: statements that are consistent in the classical context (e.g., $\fr b<\fr s$), turn out to have inconsistent higher analogues ($\fr s_\kappa\leq \fr b_\kappa$ is a theorem, \cite{RaghavanShelah2017}), or vice versa. 
	
	In some cases, the classically defined object does not have a clear generalisation.  This is the case for Lebesgue measure, which is, due to its $\sigma$-additivity, heavily dependent on the separability of the reals, cf. \cite{AgostiniBarreraDimonte}. For this reason, it is hard, or even impossible, to say what the generalisation of a Lebesgue null set should be (although several attempts have been made, see \cite{ShelahNull,FriedmanLaguzzi,CohenShelah}). 
	Nevertheless, through classically equivalent combinatorial notions, it is possible to generalise stronger concepts. For instance, even without Lebesgue measure, it is not hard to generalise the notion of strongly null, also known as strong measure zero, cf. \cite{Halko1996,HalkoShelah2001}. 
	
	In this article we will take the classical $\sigma$-ideal $\cal E$ that is generated by the \emph{closed Lebesgue null} sets of the Cantor space $\omcs$, and generalise it to the context of the higher Cantor space $\kacs$. We do so by carefully avoiding Lebesgue measure\,---\,instead, we use a characterisation of closed null sets due to \citet{BartoszynskiShelah}, which is based on slaloms and interval partitions. This allows us to introduce two $\kappa$-complete ideals, $\Ek$ and $\BEk$, that are equivalent in the classical context, but distinct in the higher context.

	\subsection{Overview}\label{sec:overview}
	
	In \cref{sec:generalising the closed null ideal}, we introduce candidates for a higher closed null ideal and study their combinatorial properties. We summarise the classical definition and combinatorial characterisation of the closed null ideal $\cal E$ on the Cantor space (\cref{sec:classical closed null}), before we define the higher analogues $\Ek$ and $\BEk$ and study their basic properties (\cref{sec:higher closed null}). We use a family of parametrised generating ideals $\Ek(B)$ for this definiton, which also provides us with the families $\EkE$, $\EkA$, $\BEkE$ and $\BEkA$ consisting of unions or intersections of such generating ideals (\cref{dfn:families}). We compare $\Ek$ and $\BEk$ to the $\kappa$-meagre ideal $\Mk$ (\cref{sec:relations with higher meagre}), and we investigate the relations between individual generating ideals based on their parameters (\cref{sec:glueing}).
	
	\cref{sec:additive ideals} concerns the closure of the higher closed null ideals under binary addition. After we provide the necessary definitions and introduce the additive ideals $\NAk$, $\MAk$ and the strongly null ideal $\SNk$ (\cref{sec:additive ideals definitions}), we study the relation between these ideals, the higher closed null additive ideals and the ideals $\EkA$ and $\BEkA$ (\cref{sec:higher closed null additive}). We finish the section by considering failures of additivity, where adding together two small sets results in a large sum (\cref{sec:decomposing})
	
	Finally, \cref{sec:cardinal characteristics} analyses the cardinal characteristics of these ideals, and compares them with each other (\cref{sec:comparison within closed null}) and with some other cardinal characteristics (\cref{sec:comparison with other cardinals}), and \cref{sec:consistency} looks at the effect of some forcing notions on the values of these cardinal characteristics.
	
	\subsection{Preliminaries}\label{sec:preliminaries}
	
	We will always assume that $\kappa$ is a regular uncountable cardinal with $\kappa=\kappa^{<\kappa}$, but often we have to make stronger assumptions, e.g.\ that $\kappa$ is regular limit or strongly inaccessible. Let $\abset{X_\alpha}{\alpha\in\kappa}$ be a sequence of sets such that $2\leq \card{X_\alpha}\leq\kappa$ for all $\alpha$, and let $\cal X=\prod_{\alpha\in\kappa}X_\alpha$. We give $\cal X$ a topology generated by the base $\set[big]{[s]}{s\in\Cup_{\gamma<\kappa}\prod_{\alpha\in\gamma}X_\alpha}$, where we define $\sqb s=\set{f\in\cal X}{s\subset f}$. We will refer to such spaces $\cal X$ simply as \emph{higher product spaces}. Among the most well-studied such spaces are the \emph{higher Baire space} $\kaka$ and the \emph{higher Cantor space} $\kacs$, in analogy with the \emph{classical Baire space} $\omom$ and \emph{classical Cantor space} $\omcs$. 
	
	We say a set $X\subset\cal X$ is \emph{$\kappa$-meagre} if it is a union of $\kappa$-many nowhere dense sets, where a set $Y$ is \emph{nowhere dense} if every basic set $\sqb s$ contains some basic set $\sqb t$ with $Y\cap \sqb t=\emp$. We let $\NWD[\cal X]$ be the family of all nowhere dense subsets of $\cal X$ and $\Mk[\cal X]$ be the family of all $\kappa$-meagre subsets of $\cal X$. It is not hard to see that $\NWD[\cal X]$ is a ${<}\kappa$-complete proper ideal on $\cal X$ and $\Mk[\cal X]$ is a ${\leq}\kappa$-complete proper ideal on $\cal X$. Here we call an ideal \emph{${<}\kappa$-complete}, if it is closed under unions of size ${<}\kappa$, and \emph{${\leq}\kappa$-complete} if it is ${<}\kappa^+$-complete. We will write $\NWD$ and $\Mk$ if the space $\cal X$ is clear from context.
	
	A property holds for \emph{almost all} $\alpha\in\kappa$ if there is  $\alpha_0\in\kappa$ such that it holds for all $\alpha\geq\alpha_0$. Similarly, a property holds for \emph{cofinally many} $\alpha\in\kappa$ if for each $\alpha_0\in\kappa$ there is $\alpha\geq\alpha_0$ for which it holds. We abbreviate these as the quantifiers $\forall^\infty\alpha\in\kappa$ and $\exists^\infty\alpha\in\kappa$, respectively. We write $[X]^\kappa$ for the set of subsets of $X$ of size $\kappa$; similarly we write $[X]^{<\kappa}$ for $\Cup_{\lambda<\kappa}[X]^\lambda$ and $[X]^{\leq\kappa}$ for $[X]^{<\kappa^+}$. Two sets $X,Y\in\sqkaka$ are called \emph{almost disjoint} if $\card{X\cap Y}<\kappa$, and $X$ is called an \emph{almost subset} of $Y$, denoted $X\subset^* Y$ if $X$ and $\kappa\setminus Y$ are almost disjoint, or equivalently if $\card{X\setminus Y}<\kappa$.
	
	\subsection{Acknowledgements}
	
	No LLMs were used in the writing and research of this article, except for \cref{lmm:BE additive failure,lmm:E additive failure}, which were strengthened from an earlier form to their current statement while testing ChatGPT 5.6. All proofs are our own.

	\section{Generalising the closed null ideal}\label{sec:generalising the closed null ideal}
	
	The $\sigma$-ideal $\cal E$ generated by closed sets of Lebesgue measure zero (in the space $\omcs$)  was given a combinatorial characterisation in terms of slaloms by \citet[Theorem~4.3]{BartoszynskiShelah}. We will repeat this characterisation in \cref{sec:classical closed null}, providing the classical framework with which the generalised context could be compared. For detailed information about the classical ideal $\cal E$ we refer to \cite[\S~2.6]{BartoszynskiJudah}, or recently  \cite{CardonaMartonSupina2025}. 
	
	Lebesgue measure does not appear in the abovementioned combinatorial characterisation of $\cal E$. This allows us to define a higher closed null ideal for uncountable regular cardinals $\kappa$. We provide two different approaches in \cref{sec:higher closed null} and compare the resulting ideals $\Ek$ and $\BEk$ to each other and to the $\kappa$-meagre ideal $\Mk$. These three ideals turn out to be comparable yet distinct, or to be precise $\Ek\subsetneq\BEk\subsetneq\Mk$.
	
	\subsection{The classical closed null ideal}\label{sec:classical closed null}
	The ideal $\cal E$ is the $\sigma$-ideal generated by the family of all closed subsets of the Cantor space $\omcs$ that have Lebesgue measure zero, where we recall that the Lebesgue measure $\mu$ on $\omcs$ is the standard product measure, where we endow $2=\ac{0,1}$ with the \emph{coin-flip measure} $\mu_{\halffrac}$, that is, $\mu_{\halffrac}\rb{\ac{0,1}}=1$ and $\mu_{\halffrac}\rb{\ac 0}=\mu_{\halffrac}\rb{\ac 1}=\frac12$. It is well-known that the theory of Lebesgue measure cannot be easily generalised to the higher context, and recent work (cf.\ \cite{AgostiniBarreraDimonte,Wontner}) has shown that any such generalisation has to be severely limited. We will therefore use a different definition of $\cal E$ that is classically equivalent. Our presentation of this characterisation is the one provided in \cite[Theorem~2.3]{CardonaMartonSupina2025}, which itself is a structural refinement of the characterisation provided in \cite[Theorem~4.3]{BartoszynskiShelah}.
	
	Let $\ell^1_+$ be the space of sequences $\epsilon\colon \omega\to\opiv{0}{\infty}$ of positive real numbers, such that $\sum_{n\in\omega}\epsilon_n$ converges.  Given such $\epsilon\in\ell^1_+$, let us for convenience define the multiplicative inverse $h_\epsilon\colon n\mapsto \frac1{\epsilon_n}$. An \emph{interval partition} of $\omega$ is a sequence $I=\abset{I_n}{n\in\omega}$ of mutually disjoint nonempty intervals of natural numbers such that $\omega=\Cup_{n\in\omega}I_n$. We will adopt the convention that $\min\rb{I_{n+1}}=\max\rb{I_n}+1$ for all $n\in\omega$, which also implies $\min(I_0)=0$. We denote by $\rm{IP}$ the family of all such interval partitions on $\omega$. We define a family of slaloms parametrised by $\epsilon\in\ell^1_+$ and $I\in\rm{IP}$ (cf.\ \cite[Definition~2.1]{CardonaMartonSupina2025}), and rewrite it in order to make it more easily comparable to our definition for the uncountable case:
	\begin{align*}
		\Sigma\rb{I,\epsilon}&=\set[m]{\phi\in\ts\prod_{n\in\omega}\cal P\rb[big]{{}^{I_n}2}}{ \lim_{n\to\infty}\textstyle\frac{\card{\phi(n)}}{2^{\card{I_n}}\cdot \epsilon_n}=0}\\
		&=\set[m]{\phi\in\ts\prod_{n\in\omega}\cal P\rb[big]{{}^{I_n}2}}{ \forall N\in\omega\fforall^\infty n\in\omega\,\rb[m]{\textstyle\frac{h_\epsilon(n)\cdot\card{\phi(n)}}{2^{\card{I_n}}}<\frac{1}N}}\\
		&=\set[m]{\phi\in\ts\prod_{n\in\omega}\cal P\rb[big]{{}^{I_n}2}}{ \forall N\in\omega\fforall^\infty n\in\omega\,\rb[m]{N\cdot h_\epsilon(n)\cdot |\phi(n)|<2^{\card{I_n}}}}.\tag{$*$}\label{eqn:classical sigma}
	\end{align*}
	We may now define for a given slalom $\phi\in\Sigma\rb{I,\epsilon}$ the following families:
	\begin{align*}
		\sqb{\phi}_*&=\set{x\in\omcs}{\forall^\infty n\in\omega\,\rb{x\restriction I_n\in\phi\rb n}}, \\
		\cal E\rb{I,\epsilon}&=\set{X\subset\omcs}{\exists \phi\in\Sigma\,\rb{I,\epsilon}\rb{X\subset\sqb \phi_*}}.
	\end{align*}
	We are now ready to give the combinatorial characterisation of $\cal E$.
	\begin{thm}[{\cite[Theorem~2.3]{CardonaMartonSupina2025}}]\label{thm:classical E}
		$\cal E=\Cup_{I\in\rm{IP},\,\epsilon\in\ell^1_+}\cal E(I,\epsilon)$.
		\qed
	\end{thm}
	
	Let us take a moment to see how we may generalise this characterisation to uncountable $\kappa$. The use of converging series of real numbers has no clear analogue, but our inverse function~$h_\epsilon$ provides a solution to this problem; the intuitive idea is that a series $\sum \epsilon_n$ is convergent if $h_\epsilon$ grows fast enough. For instance, if $\epsilon_n=n^{-r}$ for some $r>1$, then $\sum\epsilon_n$ will converge, and hence $h_\epsilon\colon n\mapsto n^r$ grows fast enough. We will see later that the role of the function $h$ is rather unrestrictive in the higher analogue, since, in fact, any choice of $h$ that is eventually nonzero will result in the same family of slaloms (see \cref{lmm:compatible}).
	
	We propose two alternative characterisations of $\cal E$. Let us call a sequence $B=\abset{B_n}{n\in\omega}$ a \emph{bounded partition} of $\omega$, if the sets $B_n$ are mutually disjoint finite sets such that $\omega=\Cup_{n\in\omega}B_n$. We define $\Sigma(B,\epsilon)$, $[\phi]_*$ for $\phi\in\Sigma(B,\epsilon)$, and $\cal E(B,\epsilon)$ exactly as before, replacing the role of the sets $I_n$ by $B_n$. This allows us to define $\cal{BE}=\Cup_{B\in\rm{BP},\,\epsilon\in\ell^1_+}\cal E(B,\epsilon)$, where $\rm{BP}$ is the family of all bounded partitions of $\omega$.
	
	\begin{thm}\label{thm:classical E and BE}
		$\cal E=\cal{BE}$.
	\end{thm}
	\begin{proof}
		Clearly every interval partition of $\omega$ is a bounded partition as well, hence $\cal E\subset\cal{BE}$.
		Reversely, it is easy to see that for every $B\in\rm{BP}$, $\epsilon\in\ell^1_+$ and $\phi\in\Sigma(B,\epsilon)$ we have  $[\phi]_*=\Cup_{m\in\omega}[\phi]_m$, where we define $[\phi]_m=\set{x\in\omcs}{\forall n\geq m\rb{x\restriction B_n\in\phi\rb n}}$. Now, as each $[\phi]_m$ is a closed set of measure zero, we have $\cal{BE}\subset\cal E$.
	\end{proof}
	
	Alternatively, we may consider the following tree-based characterisation of $\cal E$, which also avoids mention of Lebesgue measure. Let $T\subset\fomcs$ be a tree and let $\sqb T\subset\omcs$ denote the set of all branches through $T$. We say that $T$ is a \emph{halving tree} if for every $s\in T$ there is $n_s\in\omega$ such that 
	\[\card[m]{\set[big]{t\in{}^{\dom(s)+n_s}2}{s\subset t\in T}}\leq\card[m]{ \set[big]{t\in{}^{\dom(s)+n_s}2}{s\subset t\notin T}}.\]
	That is, for each $s\in T$, there is a level of the tree in which at most half of the extensions of $s$ at that level belong to the tree. Let us define $\cal H$ as the least $\sigma$-ideal containing $\sqb T$ for every halving tree $T\subset\fomcs$.
	
	\begin{thm}\label{thm:classical E and H}
		$\cal E=\cal H$.
	\end{thm}
	\begin{proof}
		If $T$ is a halving tree and $k\in\omega$, we build a sequence $\abset{m_i\in\omega}{i\leq k}$ with $m_0=0$, such that for each $s\in T\cap {}^{m_i}2$ we have $\dom(s)+n_s\leq m_{i+1}$, where $n_s$ is defined as in the halving property. It then follows that $2^{m_k}\geq 2^k\cdot\card{T\cap {}^{m_k}2}$, thus $\mu(\sqb T)\leq 2^{-k}$. This shows that $\sqb T\in\cal E$.	
		
		On the other hand, if $X\subset\omcs$ is a closed null set, then there is a tree $T\subset\fomcs$ with $X= \sqb T$. Let $s\in T$, then there is an open set $U\subset\sqb s$ such that $\mu \rb{U}>\frac12\mu\rb{\sqb s}$ and $U\cap \sqb T=\emp$. Then we may find $n_s\in\omega$ such that $\mu(\sqb s)\leq 2\cdot \mu\rb{\Cup\set{t\in {}^{\dom(s)+n_s}2}{[t]\subset U\text{ and }s\subset t}}$, from which it follows that $T$ is halving.
	\end{proof}
	
	We will see that neither \cref{thm:classical E and BE} nor \cref{thm:classical E and H} applies to the higher context.

	\subsection{The higher closed null ideals}\label{sec:higher closed null}
	
	An \emph{interval partition} of $\kappa$ is a sequence $I=\abset{I_\alpha}{ \alpha\in\kappa}$ of mutually disjoint nonempty intervals of $\kappa$ such that $\kappa=\Cup_{\alpha\in\kappa}I_\alpha$. Once again we adopt the convention that $\min\rb{I_\alpha}=\min\rb{\kappa\setminus\Cup_{\xi\in\alpha} I_\xi}$ for each $\alpha\in\kappa$. It is good to note that $\set{\min\rb{ I_\alpha}}{ \alpha\in\kappa}$ is a club subset of $\kappa$, and equivalently that the function $f\colon\alpha\mapsto \min(I_\alpha)$ is a continuous strictly increasing function. We let $\IPk$ be the set of all interval partitions of $\kappa$. 
	A \emph{bounded partition} of $\kappa$ is a sequence $B=\abset{B_\alpha}{\alpha\in\kappa}$ of mutually disjoint nonempty  bounded subsets of $\kappa$ such that $\kappa=\Cup_{\alpha\in\kappa}B_\alpha$. We let $\BPk$ denote the set of all bounded partitions of $\kappa$. Note that $\IPk\subset\BPk$. We will fix the notation that if $I\in\IPk$ or $B\in\BPk$, then $I_\alpha$ and $B_\alpha$ are the $\alpha$-th element of $I$ and $B$, respectively. Moreover, we write $i_\alpha=\min(I_\alpha)$, whence $I_\alpha=\iv{i_\alpha}{i_{\alpha+1}}$ for each $\alpha\in\kappa$.
	
	Often, it will be useful to consider only partitions in which the parts become as large as possible. Recall that $\lim_{\alpha<\kappa}\lambda_\alpha=\lambda$ if for every $\mu<\lambda$ we have $\forall^\infty\alpha\in\kappa\,\rb{\mu<\lambda_\alpha}$. Let us define 
	\begin{align*}
		\BPklim&=\begin{cases}
			\set{B\in\BPk}{\lim_{\alpha<\kappa}\card{B_\alpha}=\kappa} & \text{if }\kappa\text{ is a limit cardinal,}\\
			\set{B\in\BPk}{\forall^\infty{\alpha\in\kappa}\,\rb{\card{B_\alpha}=\lambda}} & \text{if }\kappa\text{ is the successor of }\lambda,
		\end{cases}\\
		\IPklim&=\BPklim\cap\IPk.
	\end{align*}

	Given $h\in{}^\kappa(\kappa\setminus\ac 0)$ and $B=\abset{B_\alpha}{\alpha\in\kappa}\in\BPk$, we generalise \cref{eqn:classical sigma} as:
	\begin{align*}
		\Sigma_\kappa^*\rb{B,h}=\set[m]{\phi\in\ts\prod_{\alpha\in\kappa}\cal P\rb[big]{{}^{B_\alpha}2}}{ \forall \lambda<\kappa\fforall^\infty \alpha\in\kappa\,\rb[big]{\lambda\cdot \card{h\rb{\alpha}}\cdot\card{\phi\rb{\alpha}}<2^{\card{B_\alpha}}}}.
	\end{align*}
	We analogously define, for $\phi\in\Sigma_\kappa^*\rb{B,h}$ and $\beta\in\kappa$, the families
	\begin{align*}
		\sqb\phi_\beta& = \set{x\in\kacs}{ \forall\alpha\in\kappa\setminus \beta\, \rb{x\restriction B_\alpha\in\phi\rb{\alpha}}},\\
		\sqb\phi_*& = \ts\Cup_{\beta\in\kappa}[\phi]_\beta= \set{x\in\kacs}{ \forall^\infty\alpha\in\kappa\, \rb{x\restriction B_\alpha\in\phi\rb{\alpha}}}.
	\end{align*}
	
	It is good, at this point, to reflect on our definition for a moment and show that it can be simplified. Specifically, neither $\lambda$ nor $h$ is relevant for the definition of $\Sigma^*_\kappa\rb{B,h}$, as long as we assume that $B\in\BPklim$. 
	
	\begin{dfn}
		For $B\in\BPk$ let us define  
		\begin{align*}
			\Sigma_\kappa^*\rb B&=\set[m]{\phi\in\ts\prod_{\alpha\in\kappa}\cal P\rb[big]{{}^{B_\alpha}2}}{ \forall^\infty \alpha\in\kappa\,\rb[big]{\card{\phi\rb{\alpha}}<2^{\card{B_\alpha}}}},\\
			\Sigma_\kappa\rb B&=\set[m]{\phi\in\ts\prod_{\alpha\in\kappa}\cal P\rb[big]{{}^{B_\alpha}2}}{ \forall \alpha\in\kappa\,\rb[big]{\card{\phi\rb{\alpha}}<2^{\card{B_\alpha}}}}.
		\end{align*}
	\end{dfn}

	\begin{lmm}\label{lmm:compatible}
		Let $h\in{}^\kappa(\kappa\setminus\ac 0)$ and let $B\in\BPklim$, then the following are equivalent:
		\begin{enumerate}[label=(\arabic*)]
			\item\label{lmm:compatible1} $\exists \phi\in\Sigma^*_\kappa\rb{B,h}\,\rb{\sqb\phi_*\neq\emp}$.
			\item\label{lmm:compatible2} $\forall^\infty\alpha\in\kappa\,\rb{\card{h\rb{\alpha}}<2^{\card{B_\alpha}}}$.
			\item\label{lmm:compatible3} $\Sigma^*_\kappa\rb{B,h}=\Sigma_\kappa^*\rb B$.
		\end{enumerate}
	\end{lmm}
	\begin{proof}
		
		\ref{lmm:compatible1} $\imp$ \ref{lmm:compatible2}.\quad Let $\phi\in\Sigma_\kappa^*\rb{B,h}$ have $\sqb\phi_*\neq\emp$, then there is $x\in\kacs$ and $\beta\in\kappa$ with $x\in\sqb \phi_\beta$. Particularly,  $0<\card{h\rb\alpha}\cdot\card{\phi\rb\alpha}<2^{\card{B_\alpha}}$ for almost all $\alpha\in\kappa$, thus $\card{h\rb{\alpha}}<2^{\card{B_\alpha}}$ for almost all $\alpha\in\kappa$.
		
		\noindent\ref{lmm:compatible2} $\imp$ \ref{lmm:compatible3}.\quad Let $\phi\in\Sigma^*_\kappa(B,h)$ and let $\lambda=1$, then $\card{\phi\rb\alpha}\leq \card{h(\alpha)}\cdot\card{\phi\rb\alpha}$ for all $\alpha\in\kappa$, hence $\phi\in\Sigma_\kappa^*\rb B$. On the other hand, since $B\in\BPklim$, for any $\lambda<\kappa$ we have $\forall^\infty\alpha\in\kappa\,(|B_\alpha|\geq\lambda)$, hence $\forall^\infty\alpha\in\kappa\,\rb{\lambda<2^{\card{B_\alpha}}}$, and therefore \ref{lmm:compatible2} implies that $\Sigma^*_\kappa(B)\subset\Sigma^*_\kappa(B,h)$.
		
		\noindent\ref{lmm:compatible3} $\imp$ \ref{lmm:compatible1}.\quad Since $\lim_{\alpha\in\kappa}{\card{B_\alpha}}=\kappa$ implies that there exist $\beta\in\kappa$ and $\phi\in\Sigma_\kappa^*\rb B$ such that $\forall\alpha\geq\beta\rb{\phi\rb\alpha\neq\emp}$, it follows that $\emp\neq\sqb \phi_\beta\subset \sqb\phi_*$. By \ref{lmm:compatible3} we see that $\phi\in\Sigma_\kappa^*\rb{B,h}$.
	\end{proof}
	\begin{lmm}\label{lmm:almost to everywhere}
		Let $B\in\BPk$ and $\phi\in\Sigma_\kappa^*\rb B$, then there exists $\phi'\in\SkB$ such that $\sqb\phi_*=\sqb{\phi'}_*$.
	\end{lmm}
	\begin{proof}
		It is easy to see that the following $\phi'$ suffices:
		\begin{align*}
			\phi'&\colon\alpha\mapsto\begin{cases}
				\phi\rb\alpha &\text{if }\card{\phi\rb\alpha}<2^{\card{B_\alpha}},\\
				\emp&\text{otherwise}.
			\end{cases}\qedhere
		\end{align*}
	\end{proof}
	
	In light of \cref{lmm:compatible,lmm:almost to everywhere}, we will simply work with $\SkB$ from now on. 
	
	\begin{dfn}\label{dfn:families}
		Let $B\in\BPk$, then we define the following families:
		\begin{align*}
			\Ek\rb B&=\mathrlap{\set{X\subset\kacs}{\exists\phi\in\SkB\rb{X\subset\sqb\phi_*}},}\\
			\EkE&=\ts\Cup_{I\in\IPklim}\cal E_\kappa(I),&
			\BEkE&=\ts\Cup_{B\in\BPklim}\cal E_\kappa(B),\\
			\EkA&=\ts\Cap_{I\in\IPklim}\cal E_\kappa(I),&
			\BEkA&=\ts\Cap_{B\in\BPklim}\cal E_\kappa(B),\\
			\Ek&=\set[big]{\ts\Cup E}{E\in\sqb{\EkE}^{\leq\kappa}},&
			\BEk&=\set[big]{\ts\Cup E}{E\in\sqb{\BEkE}^{\leq\kappa}}.
		\end{align*}
	\end{dfn}
	
	The following properties are easily verified from the above definition.
	
	\begin{prp}\label{prp:simple relations}
		The following hold  for all $I\in\IPklim$ and $B\in\BPklim$.
		\begin{enumerate}
			\item\label{prp:simple relations:1} $\EkA\subset\Ek(I)\subset\EkE\subset\Ek$,
			\item\label{prp:simple relations:2} $\BEkA\subset\Ek(B)\subset\BEkE\subset\BEk$,
			\item\label{prp:simple relations:3} $\BEkA\subset\EkA$ and $\EkE\subset\BEkE$ and $\Ek\subset\BEk$,
			\item\label{prp:simple relations:4} $\Ek$ and $\BEk$ are ${\leq}\kappa$-complete ideals.\qed
		\end{enumerate}
	\end{prp}
	
	We will show in \cref{prp:completeness of Ek(B)} that when $B\in\BPklim$ also $\Ek(B)$, $\EkA$ and $\BEkA$ are ${\leq}\kappa$-complete ideals. In \cref{sec:glueing}, we study the relations between these $\Ek(B)$ for different partitions $B\in\BPk$; particularly we will see  in \cref{thm:Ek(B) minmax} that no $\Ek(B)$ is $\subseteq$-least with respect to all ideals of the form $\Ek(B)$, which implies that $\EkA\neq\Ek(I)$ and $\BEkA\neq\Ek(B)$ for any $I\in\IPklim$ and $B\in\BPklim$. We also mention that not all of the families from \cref{dfn:families} are ideals, as we will see in \cref{thm:BEkm is no ideal} that $\BEkE$ is not closed under finite unions. For strongly inaccessible $\kappa$, also $\EkE$ is not an ideal, as shown in \cref{thm:eveq equi}, but $\EkE=\Ek$ if $\kappa$ is successor by \cref{crl:Ekm is ideal}. Consequently, all inclusions in \eqref{prp:simple relations:1} and \eqref{prp:simple relations:2} are strict for certain $\kappa$. Finally, the last two inclusions in \eqref{prp:simple relations:3} will be shown to be strict as well, in \cref{thm:BEk is not Ek}. The first inclusion in \eqref{prp:simple relations:3} is an equality if $\kappa$ is successor, as we will see in \cref{thm:successor BEkA=EkA}. See also \cref{fig:ideals}.

	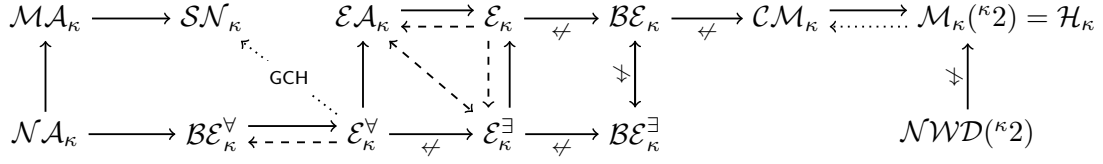
\begin{figure}[h]
		\begin{tikzpicture}[xscale=2,yscale=1.53]
			\node (NAk) at (-2.9,0) {$\NAk$};
			\node (MAk) at (-2.9,1) {$\MAk$};
			\node (SNk) at (-1.8,1) {$\SNk$};
			\node (BEkA) at (-1.8,0) {$\BEkA$};
			\node (EkA) at (-.8,0) {$\EkA$};
			\node (EAk) at (-.8,1) {$\EAk$};
			\node (EkE) at (0.1,0) {$\EkE$};
			\node (Ek) at (0.1,1) {$\Ek$};
			\node (BEkE) at (1,0) {$\BEkE$};
			\node (BEk) at (1,1) {$\BEk$};
			\node (CMk) at (2,1) {$\CMk$};
			\node (nwd) at (3.2,0) {$\NWD[\kacs]$};
			\node (Mk) at (3.2,1) {$\Mk[\kacs]\rlap{${}=\cal H_\kappa$}$};
			
			\draw 	(NAk) 
						edge[->,line width=.7pt] 
					(BEkA);
					
			\draw 	(NAk) 
						edge[->,line width=.7pt] 
					(MAk);
					
			\draw 	(MAk) 
						edge[->,line width=.7pt] 
					(SNk);
					
			\draw 	(EkA) 
						edge[->,line width=.7pt,dotted] 
						node[midway, fill=white] {\tiny$\sf{GCH}$} 
					(SNk);
					
			\draw 	(EkA) 
						edge[->,line width=.7pt] 
					(EAk);
					
			\draw 	([yshift=2pt]EAk.east) 
						edge[->,line width=.7pt] 
					([yshift=2pt]Ek.west);
					
			\draw 	([yshift=-2pt]Ek.west) 
						edge[->,line width=.7pt,dashed] 
					([yshift=-2pt]EAk.east);
			
			\draw 	([yshift=2pt]BEkA.east) 
						edge[->,line width=.7pt] 
					([yshift=2pt]EkA.west);
					
			\draw 	([yshift=-2pt]EkA.west) 
						edge[->,line width=.7pt,dashed] 
					([yshift=-2pt]BEkA.east);
			
			\draw 	(EkA) 
						edge[->,line width=.7pt] 
						node[midway,yshift=-5pt,rotate=180] {\tiny$\not\to$}
					(EkE);
					
			\draw 	(BEk) 
						edge[->,line width=.7pt] 
						node[midway,xshift=-5pt,rotate=-90] {\tiny$\not\to$}
					(BEkE);
					
			\draw 	([xshift=2pt]EkE.north) 
						edge[->,line width=.7pt] 
					([xshift=2pt]Ek.south);
					
			\draw 	([xshift=-2pt]Ek.south) 
						edge[->,line width=.7pt,dashed] 
					([xshift=-2pt]EkE.north);
					
			\draw 	(EAk) 
						edge[<->,line width=.7pt,dashed] 
					(EkE);
					
			\draw 	(EkE) 
						edge[->,line width=.7pt] 
						node[midway,yshift=-5pt,rotate=180] {\tiny$\not\to$}
					(BEkE);
					
			\draw 	(BEkE) 
						edge[->,line width=.7pt] 
					(BEk);
					
			\draw 	(Ek) 
						edge[->,line width=.7pt] 
						node[midway,yshift=-5pt,rotate=180] {\tiny$\not\to$}
					(BEk);
					
			\draw 	(BEk) 
						edge[->,line width=.7pt] 
						node[midway,yshift=-5pt,rotate=180] {\tiny$\not\to$}
					(CMk);
					
			\draw 	([yshift=-2pt]Mk.west) 
						edge[->,line width=.7pt,dotted] 
					([yshift=-2pt]CMk.east);
					
			\draw 	([yshift=2pt]CMk.east) 
						edge[->,line width=.7pt] 
					([yshift=2pt]Mk.west);
					
			\draw 	(nwd) 
						edge[->,line width=.7pt]
						node[midway,xshift=-5pt,rotate=-90] {\tiny$\not\to$} 
					(Mk);
		\end{tikzpicture}
		\caption{Comparison of families of subsets of $\kacs$, where an arrow indicates inclusion. An inclusion is proper if the corresponding arrow is labelled with a small negated arrow in the other direction. Dotted arrows are provable if $\kappa$ is inaccessible, dashed arrows are provable if $\kappa$ is successor. For definitions, see \cref{dfn:families,dfn:combinatorial meagre,dfn:halving ideal,sec:additive ideals definitions}.}\label{fig:ideals}
	\end{figure}

	\begin{prp}\label{prp:completeness of Ek(B)}\leavevmode
		\begin{enumerate}[label=(\arabic*)]
			\item\label{prp:completeness:1} If $B\in\BPk$ and $\lambda:=\liminf_{\alpha\in\kappa}{\card{B_\alpha}}$ is infinite, then $\Ek\rb B$ is a ${\leq}\lambda$-complete ideal. 
			\item\label{prp:completeness:2} If $B\in\BPklim$, then $\Ek(B)$ is a ${\leq}\kappa$-complete ideal. 
			\item\label{prp:completeness:3} $\EkA$ and $\BEkA$ are ${\leq}\kappa$-complete ideals.
		\end{enumerate}
	\end{prp}
	\begin{proof}
		\ref{prp:completeness:1}.\quad Let $B\in\BPk$, let $\abset{\phi_\xi}{\xi\in\lambda}$ be a sequence of slaloms in $\SkB$, and define 
		\[\phi\colon\alpha\mapsto\ts\Cup\set[m]{\phi_\xi\rb\alpha}{\xi<\min(\lambda,\card{B_\alpha})\text{ and }\card{\phi_\xi\rb\alpha}<2^{\card{B_\alpha}}}.\]
		If $\card{B_\alpha}$ is infinite, then  $\card{B_\alpha}<\cf\rb{2^{\card{B_\alpha}}}$ implies that $\card{\phi\rb\alpha}<2^{\card{B_\alpha}}$. But this applies to almost all $\alpha\in\kappa$, since $\liminf_{\alpha\in\kappa}\card{B_\alpha}=\lambda$ is infinite and $\kappa$ is regular. Moreover, for almost all $\alpha\in\kappa$ we see that $\phi_\xi(\alpha)\subset\phi(\alpha)$, and thus it follows that $\Cup_{\xi\in\lambda}\sqb{\phi_\xi}_*\subset\sqb\phi_*$.
		
		\noindent\ref{prp:completeness:2}. If $\kappa$ is regular limit, this follows from \ref{prp:completeness:1}. Suppose $\kappa=\lambda^+$ and let $B\in\BPklim$, then we may assume that $\card{B_\alpha}=\lambda$ for all $\alpha\in\kappa$. Let $\phi_\xi\in\SkB$ for each $\xi\in\kappa$, and note that  $\kappa=2^{<\kappa}=2^\lambda$ implies $\card{\phi_\xi(\alpha)}\leq\lambda$ for each $\alpha\in\kappa$. Then, particularly, if we define $\phi\colon \alpha\mapsto \Cup_{\xi\in\alpha}\phi_\xi(\alpha)$, we see that $\card{\phi(\alpha)}\leq\lambda$ for each $\alpha\in\kappa$, and thus $\phi\in\SkB$, and $\Cup_{\xi\in\kappa}\sqb{\phi_\xi}_*\subset\sqb\phi_*$.
		
		\noindent\ref{prp:completeness:3}. This is clear, since both $\BEkA$ and $\EkA$ are the intersection of a family of ${\leq}\kappa$-complete ideals.
	\end{proof}
	
	\begin{thm}\label{thm:BEk is not Ek}
		Let $\kappa$ be regular uncountable, then $\BEkE\nsubseteq\Ek$. Hence, $\EkE\subsetneq\BEkE$ and $\Ek\subsetneq\BEk$. 
	\end{thm}
	\begin{proof}
		The inclusions easily follow from $\IPklim\subset\BPklim$. 
		
		Let $E=\set{2\cdot\alpha}{\alpha\in\kappa}$ and $X=\set{x\in\kacs}{x^{-1}(1)\subset E}$. That is, $E$ is the set of even ordinals below $\kappa$, and $X$ is the set of $x\in\kacs$ that map all odd ordinals to $0$. We claim that $X\in\BEkE\setminus\Ek$.
		
		On the one hand, we may choose $B\in\BPklim$ such that for each $\alpha\in\kappa$ there exists $\xi_\alpha$ with $E\cap B_\alpha=\ac{\xi_\alpha}$, i.e., each $B_\alpha$ contains exactly one even ordinal. Let us define 
		\[\phi\colon\alpha\mapsto\set[m]{s\in{}^{B_\alpha}2}{\forall \xi\in B_\alpha\,\rb{\xi\neq\xi_\alpha\to s\rb\xi=0}},
		\]
		then $\card{\phi(\alpha)}=2$ and $X\subset\sqb\phi_*\in\Ek(B)\subset\BEkE$. 
		
		On the other hand, suppose $Y\in\Ek$, then there are $I^\zeta\in\IPklim$ and $\phi^\zeta\in\Sigma_\kappa(I^\zeta)$ for each $\zeta\in\kappa$ such that $Y\subset\Cup_{\zeta\in\kappa}\sqb{\phi^\zeta}_*$. For each $\alpha\in\kappa$ choose $\beta_\alpha$, $\zeta_\alpha$ such that
		\begin{enumerate}
			\item  the intervals $I_\alpha:=I^{\zeta_\alpha}_{\beta_\alpha}$ are  infinite and pairwise disjoint, 
			\item  $\card[big]{\set{\alpha\in\kappa}{\zeta=\zeta_\alpha}}=\kappa$ for every $\zeta\in\kappa$. 
		\end{enumerate}
		Since $I_\alpha$ is an infinite interval, $\card{X\restriction I_\alpha}=2^{\card{I_\alpha}}$, thus we may choose some $s_\alpha\in (X\restriction I_\alpha)\setminus \phi^{\zeta_\alpha}(\beta_\alpha)$. Now note that there exists $x\in X$ that is an extension of $\Cup_{\alpha\in\kappa}s_\alpha$ and clearly  $x\in X\setminus Y$.
	\end{proof}
	
	\subsection{Relations with the higher meagre ideal}\label{sec:relations with higher meagre}
	
	We will now compare our ideals $\Ek$ and $\BEk$ to the $\kappa$-meagre ideal and show that $\BEk\subsetneq\Mk[\kacs]$. In fact, we will show a slightly stronger result. For this, we will introduce the ideal of \emph{combinatorially $\kappa$-meagre} sets, due to \citet{BlassHyttinenZhang}.
	
	\begin{dfn}\label{dfn:combinatorial meagre}
		Let $x\in\kacs$ and $B\in\BPk$, then we define the sets
		\begin{align*}
			\dab{x,B}&=\set{y\in\kacs}{\forall\alpha\in\kappa\,\rb{y\restriction B_\alpha\neq x\restriction B_\alpha}},\\
			\dab{x,B}_*&=\set{y\in\kacs}{\forall^\infty\alpha\in\kappa\,\rb{y\restriction B_\alpha\neq x\restriction B_\alpha}}.
		\end{align*}
		A set $Y\subset\kacs$ is \emph{combinatorially $\kappa$-meagre} if there exist $x\in\kacs$ and $I\in\IPk$ such that $Y\subset\dab{x,I}_*$. Let $\CMk$ denote the family of combinatorially $\kappa$-meagre sets.
	\end{dfn}

	\begin{prp}[{\citet[\S~4]{BlassHyttinenZhang}}]\label{lmm:CM properties}
		Let $\kappa$ be regular uncountable, then
		\begin{enumerate}[label=(\arabic*)]
			\item\label{lmm:CM properties:1} $\CMk$ is a ${\leq}\kappa$-complete ideal,
			\item\label{lmm:CM properties:2} $\CMk\subset\Mk[\kacs]$,
			\item\label{lmm:CM properties:3} $\CMk=\Mk[\kacs]$ if and only if $\kappa$ is strongly inaccessible.\qed
		\end{enumerate}
	\end{prp}

	Note that we defined combinatorially $\kappa$-meagre sets only using interval partitions, but that we may as well have used bounded partitions instead, as a corollary of the following observation.

	\begin{lmm}\label{capturing BP with IP}
		If $B\in\BPk$, then there exist $I\in\IPk$ and a sequence of sets $\abset{A_\alpha}{\alpha\in\kappa}$ such that $I_\alpha=\Cup_{\beta\in A_\alpha}B_\beta$ for each $\alpha\in\kappa$. 
	\end{lmm}
	\begin{proof}
		One constructs $I_\alpha$, $A_\alpha$ simultaneously by recursion. Given that we have constructed $I_\xi$, $A_\xi$ for each $\xi<\alpha$, let $\gamma_0=\Cup_{\xi<\alpha}I_\xi$ (which is an ordinal) and let $A_\alpha^0=\ac\beta$ where $\beta$ is such that $\gamma_0\in B_\beta$. For each $n\in\omega$, let $\gamma_{n+1}$ be the least ordinal such that $\forall \gamma\in\Cup_{\beta\in A_\alpha^n}B_\beta\,(\gamma_{n+1}> \gamma)$, and define 
		\[A_\alpha^{n+1}=\set{\beta\in\kappa}{\exists\gamma\in\iv{\gamma_0}{\gamma_{n+1}}\rb{\gamma\in B_\beta}}.\]
		Finally, $A_\alpha=\Cup_{n\in\omega}A_\alpha^n$ and $I_\alpha=\iv{\gamma_0}{\Cup_{n\in\omega}\gamma_n}$.
	\end{proof}
	\begin{crl}
		$\CMk=\set{Y\subset\kacs}{\exists x\in\kacs\eexists B\in\BPk\,(Y\subset\dab{x,B}_*)}$.\qed
	\end{crl}
	
	In order to compare $\BEk$ to $\CMk$, we will be considering $\dab{x,I}$ for $x\in\kacs$ and $I\in\IPk$ as a subspace of the higher Cantor space $\kacs$. Particularly, we will use the $\kappa$-meagre ideal $\Mk[\dab{x,I}]$ of the space $\dab{x,I}$. First we note that the $\kappa$-Baire category theorem holds in $\dab{x,I}$:
	
	\begin{lmm}\label{lmm:proper CM ideal}
		Let $x\in\kacs$ and let $I\in\IPk$ be such that $\card{I_\alpha}\geq 2$ for all $\alpha\in\kappa$, then $\Mk[\dab{x,I}]$ is a proper ideal on $\dab{x,I}$.
	\end{lmm}
	
	\begin{proof}
		Let $X_\alpha={}^{I_\alpha}2\setminus \{x\restriction I_\alpha\}$ for $\alpha\in\kappa$
		and consider the higher product space $\cal X=\prod_{\alpha\in\kappa}X_\alpha$. Since $|X_\alpha|\geq2$ for all $\alpha\in\kappa$, one easily sees that $\cal X\notin\Mk[\cal X]$, as in the classical proof of the Baire category theorem. Finally note that $\dab{x,I}$ is homeomorphic to $\cal X$.
	\end{proof}
	
	\begin{thm}\label{thm:BE < CM}
		Let $\kappa$ be regular uncountable, then $\BEk\subsetneq\CMk$. 
	\end{thm}
	\begin{proof}
		Let $Y\in\BEkE$, let $B\in\BPklim$ and let $\phi\in\Sigma_\kappa(B)$ be such that $Y\subset\sqb\phi_*$. For each $\alpha\in\kappa$ pick $s_\alpha\in{}^{B_\alpha}2$ such that $s_\alpha\notin\phi(\alpha)$ for almost all $\alpha\in\kappa$, and define $x=\Cup_{\alpha\in\kappa}s_\alpha$. If $y\in Y$, then $y\restriction B_\alpha\in \phi(\alpha)$ for almost all $\alpha$, hence $y\restriction B_\alpha\neq x\restriction B_\alpha$ for almost all $\alpha$, thus $y\in \dab{x,B}_*$. Therefore $Y\in\CMk$, and thus $\BEkE\subset\CMk$. By \cref{lmm:CM properties}\,\ref{lmm:CM properties:1}, it follows that $\BEk\subset\CMk$.
		
		Let $x\in\kacs$ and $I\in\IPklim$ with $\card{I_\alpha}\geq 2$ for all $\alpha\in\kappa$, then $\dab{x,I}\in\CMk\setminus\Mk[\dab{x,I}]$ by \cref{lmm:proper CM ideal} and the definition of $\CMk$. To conclude the proof, it suffices to show that $A\cap\dab{x,I}\in\Mk[\dab{x,I}]$ for all $A\in\BEkE$. This also implies for all $A'\in\BEk$ that $A'\cap \dab{x,I}\in\Mk[\dab{x,I}]$, since $\Mk[\dab{x,I}]$ is ${\leq}\kappa$-complete. Therefore $\dab{x,I}\notin\BEk$.
		Fix some $B\in\BPklim$ and $\phi\in\Sigma_\kappa(B)$, then we show that $\sqb \phi_\beta\cap \dab{x,I}\in\NWD[\dab{x,I}]$ for each $\beta<\kappa$, from which it follows that  $\sqb\phi_*\cap\dab{x,I}\in\Mk[\dab{x,I}]$.
		
		Fix some $s\in\fkacs$ such that $\dom(s)=\Cup_{\alpha\in\delta}I_\alpha$ for some $\delta\in\kappa$ and $[s]\cap\dab{x,I}\neq\emp$. Let $\gamma\geq\beta$ and $\delta'>\delta$ be such that $B_\gamma$ is infinite and $B_\gamma\subset\Cup_{\alpha\in\iv{\delta}{\delta'}}I_\alpha$. 
		For each $\alpha\in\iv{\delta}{\delta'}$, pick $\xi_\alpha\in I_\alpha$ with the only restriction that $\xi_\alpha\notin B_\gamma$ whenever $I_\alpha\nsubset B_\gamma$. Now note that $\card{B_\gamma\setminus\set{\xi_\alpha}{ \alpha\in\iv{\delta}{\delta'}}}=\card{B_\gamma}$, so there is $t\in{}^{B_\gamma}2\setminus\phi(\gamma)$ such that $t(\xi_\alpha)\neq x(\xi_\alpha)$ for all $\alpha\in \iv{\delta}{\delta'}$. We may now extend $s\cup t$ to some $r$ with $\dom(r)=\Cup_{\alpha\in\delta'}I_\alpha$ where $r(\xi_\alpha)\neq x(\xi_\alpha)$ for all $\alpha\in\iv{\delta}{\delta'}$. Then $[r]\cap\dab{x,I}\neq\emp$ and also $[r]\cap[\phi]_\beta=\emp$.
	\end{proof}
	
	As a consequence of \cref{thm:BE < CM}, we observe that $\NWD[\kacs]\nsubset\BEk$, since the set $\dab{x,I}$ from the above proof is nowhere dense in $\kacs$. Reversely, we can also observe that $\BEkA\nsubset\NWD[\kacs]$: because $\kappa=2^{<\kappa}$, it follows that $[\kacs]^\kappa$ contains a dense set, and we have $[\kacs]^\kappa\subset\BEkA$.
	
	Now let us take a brief look at the generalisation of halving trees and the ideal $\cal H$ from \cref{thm:classical E and H}. 
	
	\begin{dfn}\label{dfn:halving ideal}
		Call a tree $T\subset\fkacs$ a \emph{halving tree} if for every $s\in T$ there is $\alpha_s\in\kappa$ such that 
	\[\card[m]{\set[big]{t\in{}^{\dom(s)+\alpha_s}2}{s\subset t\in T}}\leq\card[m]{ \set[big]{t\in{}^{\dom(s)+\alpha_s}2}{s\subset t\notin T}}.\]
	Let $\cal H_\kappa$ be the least $\kappa$-complete ideal containing $\sqb T$ for every halving tree $T\subset\fkacs$. 
	\end{dfn}
	In contrast with the classical context, this ideal is actually equal to the $\kappa$-meagre ideal:
	
	\begin{thm}
		$\Mk[\kacs]=\cal H_\kappa$.
	\end{thm}
	\begin{proof}
		Easily $[T]\in\NWD[\kacs]$ for every halving $T\subset\fkacs$.   We show that reversely $\NWD[\kacs]\subset\cal H_\kappa$.  
		
		Let $N\subset\kacs$ be nowhere dense, and without loss of generality, we may assume $N$ is closed.   Let $T\subset\fkacs$ be such that $N=[T]$   and $s\in T$,   then there is $t\in\fkacs$ with $s\subset t$ and $t\notin T$.   We may also assume without loss of generality that $\dom(t)\geq \dom(s)+\gamma$ for some infinite $\gamma$,   then $\alpha_s=\gamma\cdot 2$ suffices:   there are $\card{{}^\gamma2}=2^{\card{\gamma}}$ extensions of $t$ at the level $\dom(s)+\alpha_s$,   but $s$ also has $\card{{}^{\gamma\cdot 2}2}=2^{\card{\gamma}}$ extensions at this level.
	\end{proof}

	\subsection{Comparing the generating ideals}\label{sec:glueing} 
	Remember the classical \cref{thm:classical E}:
	\[\cal E=\Cup_{I\in\rm{IP},\epsilon\in\ell^1_+}\cal E(I,\epsilon).\] 
	Since $\cal E$ is a $\sigma$-ideal by definition, it follows from this theorem that, given $I^n\in\rm{IP}$ and $\epsilon^n\in\ell^1_+$ and slaloms $\phi^n\in\Sigma\rb{I^n,\epsilon^n}$ for each $n\in\omega$, there are $I\in\rm{IP}$ and $\epsilon\in\ell^1_+$ such that $\Cup_{n\in\omega} \sqb {\phi^n}_*\in \cal E\rb{I,\epsilon}$. Exactly, the theorem implies that any countably many slaloms are contained in a single generating set $\cal E(I,\epsilon)$ of the ideal $\cal E$.
	
	In this section we show that the higher analogue to \cref{thm:classical E} is false. In fact, we will show that $\BEkE$ and $\EkE$ are not even ideals (under suitable assumptions on $\kappa$), as they are not closed under  finite unions. We also present results comparing the generating sets $\Ek\rb{B}$ for different $B\in\BPklim$.

	Consider the equivalence relation $\eqs$ on $\kacs$, where $f\eqs g$ if $\forall^\infty\alpha\in\kappa\,(f(\alpha)=g(\alpha))$. If $\phi\in\SkB$ for some $B\in\BPk$, then $\sqb\phi_*$ is closed under $\eqs$, in the sense that $f\in\sqb \phi_*$ and $g\in\kacs$ with $f\eqs g$ implies that $g\in\sqb\phi_*$ as well. For $A\subset\kacs$, let us define $(A)_*=\set{f\in\kacs}{\exists g\in A\,\rb{f\eqs g}}$. We observe the following basic properties.
	\begin{prp}\label{lmm:almost equal properties}
		Let $A,A'\subset\kacs$, $B\in\BPk$, $\phi\in\SkB$, and $\beta\in\kappa$, then 
		\begin{enumerate}[label=\normalfont(\arabic*)]
			\item\label{lmm:almost equal properties:1} $A\subset \rb A_*$ and if $A\subset A'$, then $\rb A_*\subset\rb{A'}_*$,
			\item\label{lmm:almost equal properties:2} $\rb{\sqb \phi_\beta}_*=\rb{\sqb\phi_*}_*=\sqb\phi_*$,
			\item\label{lmm:almost equal properties:3} $A\in\Ek\rb B$ if and only if $(A)_*\in\Ek\rb B$,
			\item\label{lmm:almost equal properties:4} $A\subset\sqb{\phi}_*$ if and only if $\rb A_*\subset\sqb{\phi}_*$. \qed
		\end{enumerate}
	\end{prp}
	
	For a fixed $B\in\BPk$ and $\phi\in\SkB$, we may consider $\sqb\phi_0=\set{x\in\kacs}{\forall\alpha\in\kappa\,\rb{x\restriction B_\alpha\in\phi(\alpha)}}$ as a subspace of $\kacs$, and naturally we may define the $\kappa$-meagre ideal $\Mk[\sqb\phi_0]$ on this space, similar to what we did in \cref{lmm:proper CM ideal}. Note, however, that $\Mk[\sqb\phi_0]$ need not be proper for certain pathological choices of $\phi$ (and $B$). E.g., if $\phi(\alpha)$ is a singleton for all $\alpha\in\kappa$, then $\sqb\phi_0$ is the one-point space. 
	
	\begin{lmm}\label{lmm:nontrivial meagre for slaloms}
		Let $B\in\BPk$ and $\phi\in\SkB$ be such that $\exists^\infty\alpha\in\kappa\,\rb{\card{\phi\rb\alpha}>1}$, then $\Mk[\sqb\phi_0]$ is a proper ${\leq}\kappa$-complete ideal on $\sqb\phi_0$.
	\end{lmm}
	\begin{proof}
		This follows quickly from the fact that $\sqb\phi_0$ is closed.
	\end{proof}
	
	Observe that the conditions of the above lemma can only be fulfilled if $\exists^\infty\alpha\in\kappa\,\rb{\card{B_\alpha}>1}$. 
	We are now ready to give a characterisation of the conditions on $B,C\in\BPk$ for which $\Ek\rb B\subset\Ek\rb C$.
	
	\begin{lmm}\label{lmm:Ek(B) comparison}
		Let $B,C\in\BPk$ and for $\alpha\in\kappa$ let
		$S_\alpha=\set{\delta\in\kappa}{B_\delta\cap C_\alpha\neq\emptyset}$.
		The following conditions are equivalent:
		\begin{enumerate}[label=\normalfont(\arabic*)]
			\item\label{lmm:Ek(B) comparison:1}
			$\Ek(B)\subseteq\Ek(C)$.
			\item\label{lmm:Ek(B) comparison:2}
			For all $\phi\in\Sigma_\kappa(B)$ there exists $\phi'\in\Sigma_\kappa(C)$ such that $
			\sqb\phi_*\subseteq\sqb{\phi'}_*$.
			\item\label{lmm:Ek(B) comparison:3}
			For all $\phi\in\Sigma_\kappa(B)$ for which $\exists^\infty\alpha\in\kappa\,\rb{\card{\phi\rb\alpha}>1}$ holds, there exists $\phi'\in\Sigma_\kappa(C)$ such that $\sqb{\phi'}_*\cap \sqb\phi_0\notin\Mk[\sqb\phi_0]$.
			\item\label{lmm:Ek(B) comparison:4}
			$\forall^\infty\alpha\in\kappa\,\rb[big]{\prod_{\delta\in S_\alpha}\kappa_{\alpha,\delta}<2^{\card{C_\alpha}}}$ holds for any system $\abset{\kappa_{\alpha,\delta}}{\alpha\in\kappa\text{, }\delta\in S_\alpha}$ of cardinals such that $\kappa_{\alpha,\delta}<2^{\card{B_\delta}}$ and $\kappa_{\alpha,\delta}\leq2^{\card{B_\delta\cap C_\alpha}}$ for all  $\delta\in S_\alpha$ and $\alpha\in\kappa$.
		\end{enumerate}
	\end{lmm}
	\begin{proof}
		\ref{lmm:Ek(B) comparison:1} $\bimp$ \ref{lmm:Ek(B) comparison:2}.\quad Obvious.
		
		\noindent\ref{lmm:Ek(B) comparison:2} $\imp$ \ref{lmm:Ek(B) comparison:3}.\quad Let $\phi\in\SkB$ be such that $\exists^\infty\alpha\in\kappa\,\rb{\card{\phi\rb\alpha}>1}$ holds. By \ref{lmm:Ek(B) comparison:2}, there exists some $\phi'\in\Sigma_\kappa\rb C$ such that $\sqb\phi_0\subset \sqb\phi_*\subset\sqb{\phi'}_*$, and thus $\sqb{\phi'}_*\cap \sqb\phi_0=\sqb\phi_0$. By \cref{lmm:nontrivial meagre for slaloms}, we have $\sqb\phi_0\notin\Mk[\sqb\phi_0]$.
		
		\noindent\ref{lmm:Ek(B) comparison:3} $\imp$ \ref{lmm:Ek(B) comparison:4}.\quad We prove the contrapositive. That is, we will construct $\phi\in\SkB$ such that for every $\phi'\in\Sigma_\kappa(C)$ we have $\sqb{\phi'}_*\cap\sqb\phi_0\in\Mk[\sqb\phi_0]$, under the assumption that there is $K\in\sqkaka$ such that  $2^{\card{C_\alpha}}\leq\prod_{\delta\in S_\alpha}\kappa_{\alpha,\delta}$ for all $\alpha\in K$, where  $\abset{\kappa_{\alpha,\delta}}{\alpha\in\kappa\text{, }\delta\in S_\alpha}$ is such that $\kappa_{\alpha,\delta}<2^{\card{B_\delta}}$ and $\kappa_{\alpha,\delta}\leq2^{\card{B_\delta\cap C_\alpha}}$ for all  $\delta\in S_\alpha$ and $\alpha\in\kappa$. 
		
		We may assume, by shrinking $K$ if necessary, that $S_\alpha\cap S_{\alpha'}=\emp$ for all distinct $\alpha,\alpha'\in K$. For each $\alpha\in K$ and $\delta\in S_\alpha$, we have $\kappa_{\alpha,\delta}\leq 2^{\card{B_\delta\cap C_\alpha}}$, thus we may choose $F_\delta\subset{}^{B_\delta}2$ such that $\kappa_{\alpha,\delta}=\card{F_\delta}=\card[big]{\set{s\restriction \rb{B_\delta\cap C_\alpha}}{s\in F_\delta}}$. Furthermore, since $\kappa_{\alpha,\delta}<2^{\card{B_\delta}}$, we may choose some $\phi\in\SkB$ such that $\phi(\delta)=F_\delta$ for all $\delta\in\Cup_{\alpha\in K}S_\alpha$. Note that for any $\alpha\in K$, we have $2\leq 2^{\card{C_\alpha}}\leq\prod_{\delta\in S_\alpha}\kappa_{\alpha,\delta}$, thus there is $\delta\in S_\alpha$ such that $2\leq\kappa_{\alpha,\delta}$, and therefore \cref{lmm:nontrivial meagre for slaloms} tells us that $\Mk[\sqb\phi_0]$ is a proper ideal on $\sqb\phi_0$.
		
		Let $\phi'\in\Sigma_\kappa(C)$ and $\beta\in\kappa$ be arbitrary, then we show that $\sqb{\phi'}_\beta\cap\sqb\phi_0\in\NWD[\sqb\phi_0]$, and consequently $\sqb{\phi'}_*\cap\sqb\phi_0\in\Mk[\sqb\phi_0]$. Let $s\in\fkacs$ be such that $\sqb s\cap\sqb\phi_0\neq\emp$, then we may choose $\alpha\in K\setminus\beta$ such that $\dom(s)\cap \Cup_{\delta\in S_\alpha}B_\delta=\emp$. By our above construction 
		\[\ts2^{\card{C_\alpha}}\leq\prod_{\delta\in S_\alpha}\kappa_{\alpha,\delta}=\prod_{\delta\in S_\alpha}\card[big]{\set{s\restriction\rb{B_\delta\cap C_\alpha}}{s\in \phi\rb\delta}},\] 
		hence we may choose $t_\delta\in \phi\rb\delta$ for each $\delta\in S_\alpha$ such that $t=\Cup_{\delta\in S_\alpha}t_\delta$ has the property that $t\restriction C_\alpha\notin \phi'\rb\alpha$. Finally, there is $x\in\sqb s\cap\sqb\phi_0$ with $t\subset x$, and thus if $\gamma$ is large enough such that $\Cup_{\delta\in S_\alpha}B_\delta\subset\gamma$, then $\sqb{x\restriction\gamma}\cap \sqb\phi_0\neq \emp$, and because $x\restriction C_\alpha\notin \phi'\rb\alpha$ we also have $\sqb{x\restriction\gamma}\cap\sqb{\phi'}_\beta=\emp$.

		\noindent\ref{lmm:Ek(B) comparison:4} $\imp$ \ref{lmm:Ek(B) comparison:2}.\quad Let $\phi\in\SkB$ be given and let $F_{\alpha,\delta}=\set{s\restriction\rb{B_\delta\cap C_\alpha}}{s\in \phi\rb\delta}$ and $\kappa_{\alpha,\delta}=\card{F_{\alpha,\delta}}$ for each $\alpha\in\kappa$ and $\delta\in S_\alpha$. Define 
		\[\phi'\rb\alpha=\set[big]{s\in{}^{C_\alpha}2}{\forall \delta\in S_\alpha\,\rb{s\restriction\rb{B_\delta\cap C_\alpha}\in F_{\alpha,\delta}}},\] 
		then $\sqb\phi_*\subset\sqb{\phi'}_*$ and $\card{\phi'\rb\alpha}=\prod_{\delta\in S_\alpha}\kappa_{\alpha,\delta}$. Hence, by \ref{lmm:Ek(B) comparison:4}, we see that $\card{\phi'\rb\alpha}<2^{\card{C_\alpha}}$ for almost all $\alpha\in\kappa$, which means that $\phi'\in\Sigma_\kappa^*\rb C$. By \cref{lmm:almost to everywhere}, we are done. 
	\end{proof}
	
	\begin{crl}\label{crl:noninclusion}
		Let $B,C\in\rm{BP}_\kappa$. Let  $S_\alpha=\set{\delta\in\kappa}{B_\delta\cap C_\alpha\neq\emp}$.
		\begin{enumerate}[label=(\arabic*)]
			\item\label{crl:noninclusion:Sa} If $\forall^\infty\delta\in\kappa\,(\card{B_\delta}\geq 2)$ and  $\exists^\infty\alpha\in\kappa\,\rb{\card{S_\alpha}=\card{C_\alpha}}$, then $\Ek(B)\nsubset\Ek(C)$.
			\item\label{crl:noninclusion:powers} If $\exists^\infty\alpha\in\kappa\fforall\delta\in\kappa\,\rb[big]{2^{\card{B_\delta\cap C_\alpha}}<2^{\card{B_\delta}}}$, then $\Ek(B)\nsubset\Ek(C)$.
		\end{enumerate}
	\end{crl}
	\begin{proof}
		\ref{crl:noninclusion:Sa}. Note that $2\leq 2^{\card{B_\delta\cap C_\alpha}}$ for all $\alpha\in\kappa$ and $\delta\in S_\alpha$, and that by assumption $2<2^{\card{B_\delta}}$ for almost all $\delta\in\kappa$. Therefore, if we let $\kappa_{\alpha,\delta}=2$ for all $\alpha\in\kappa$ and $\delta\in S_\alpha$, then $\exists^{\infty}\alpha\in\kappa\,(\card S_\alpha=\card{C_\alpha})$ implies that $\exists^{\infty}\alpha\in\kappa\,\rb[big]{\prod_{\delta\in S_\alpha}\kappa_{\alpha,\delta}=2^{\card {S_\alpha}}=2^{\card{C_\alpha}}}$. By \cref{lmm:Ek(B) comparison}, we see that $\Ek(B)\nsubset\Ek(C)$.
		
		\noindent\ref{crl:noninclusion:powers}. This follows from \cref{lmm:Ek(B) comparison} by letting $\kappa_{\alpha,\delta}=2^{\card{B_\delta\cap C_\alpha}}$.
	\end{proof}
	
	\begin{crl}\label{crl:incomparable Ek(B)s}
		If $B,C\in\BPklim$ and there is some regular $\lambda$ with $\lambda<\kappa$ such that $\card{B_\delta\cap C_\alpha}<\lambda$ for almost all $\alpha,\delta\in\kappa$, then $\Ek(B)\nsubseteq\Ek(C)$ and $\Ek(C)\nsubseteq\Ek(B)$.
	\end{crl}
	\begin{proof}
		Since $C\in\BPklim$, we see that $\forall^\infty\alpha\in\kappa\,(\card{C_\alpha}\geq\lambda)$. Hence, for large enough $\alpha$, we have $|S_\alpha|\geq\lambda$, where $S_\alpha=\set{\delta\in\kappa}{B_\delta\cap C_\alpha\neq\emp}$. It follows from \cref{crl:noninclusion} that $\Ek(B)\nsubset\Ek(C)$. The reverse follows by a symmetric argument.
	\end{proof}
	
	We will use the above corollary to make some notes about the existence of greatest, least and incomparable elements of the partial orders $\ab[big]{\set{\Ek(B)}{B\in\BPklim},\subset}$ and $\ab[big]{\set{\Ek(I)}{I\in\IPklim},\subset}$.
	
	\begin{prp}\label{thm:Ek(B) minmax}
		\leavevmode
		\begin{enumerate}[label=(\arabic*)]
			\item\label{thm:Ek(B) minmax:1} If $\kappa$ is regular uncountable, then $\ab[big]{\set{\Ek(B)}{B\in\BPklim},\subset}$  and $\ab[big]{\set{\Ek(I)}{I\in\IPklim},\subset}$ have no least element.
			\item\label{thm:Ek(B) minmax:2} If $\kappa$ is regular uncountable, then for every $B\in\BPklim$ there exists $C\in\BPklim$ such that $\Ek(B)$ and $\Ek(C)$ are incomparable. Hence, $\ab[big]{\set{\Ek(B)}{B\in\BPklim},\subset}$ has no greatest element.
			\item\label{thm:Ek(B) minmax:3} If $\kappa$ is regular limit, then for every $I\in\IPklim$ there exists $J\in\IPklim$ such that $\Ek(I)$ and $\Ek(J)$ are incomparable. Hence, $\ab[big]{\set{\Ek(I)}{I\in\IPklim},\subset}$ has no greatest element.
			\item\label{thm:Ek(B) minmax:4} If $\kappa$ is successor, then $\ab[big]{\set{\Ek(I)}{I\in\IPklim},\subset}$ has a greatest element.
		\end{enumerate}
	\end{prp}
	\begin{proof}
		\ref{thm:Ek(B) minmax:1}.\quad We will show this for $\BPklim$, but the situation for $\IPklim$ is identical. Suppose that $B\in\BPklim$, then we construct $C\in\BPklim$ such that $\Ek(B)\nsubset\Ek(C)$. Note that we can recursively construct a continuous strictly increasing sequence of ordinals $\abset{\xi_\alpha}{\alpha\in\kappa}$ such that 
		\[\card[big]{\iv{\xi_\alpha}{\xi_{\alpha+1}}}=\card[big]{\ts\Cup\set{B_\delta}{\xi_\alpha\leq \delta<\xi_{\alpha+1}}}.\]
		Then it suffices to define $C$ by $C_\alpha:=\Cup\set{B_\delta}{\xi_\alpha\leq\delta<\xi_{\alpha+1}}$ for all $\alpha\in\kappa$ and apply \cref{crl:noninclusion}\,\ref{crl:noninclusion:Sa}.

		\noindent\ref{thm:Ek(B) minmax:2}.\quad Suppose that $B\in\BPklim$, and without loss assume that $\card{B_\delta}$ is infinite for every $\delta\in\kappa$. Let $\abset{\xi_\alpha}{\alpha\in\kappa}$ be a sequence as defined in \ref{thm:Ek(B) minmax:1} above. We define $C$ as follows. 
		
		If $\alpha\in\kappa$ is an even ordinal, then we let $C_\alpha:=\Cup\set{B_\delta}{\xi_\alpha\leq\delta<\xi_{\alpha+1}}$. For each odd ordinal $\gamma\in\kappa$ and $\delta$ such that $\xi_\gamma\leq\delta<\xi_{\gamma+1}$, let $B'_\delta$ be a partition of $B_\delta$ into $\card{B_\delta}$-many parts of size $\card{B_\delta}$. We then let $\set{C_\alpha}{\alpha\in\kappa\text{ is odd}}$ be an enumeration of 
		$\Cup\set{B'_\delta}{\exists \gamma\in\kappa\,(\gamma\text{ is odd and }\xi_\gamma\leq\delta<\xi_{\gamma+1})}$.
		We may again apply \cref{crl:noninclusion}\,\ref{crl:noninclusion:Sa} to see that $\Ek(B)\nsubset\Ek(C)\nsubset\Ek(B)$.

		\noindent\ref{thm:Ek(B) minmax:3}.\quad The argument is similar to \ref{thm:Ek(B) minmax:2} above, but we have to take a little more care in order to deal with intervals. Given $I\in\IPklim$, we define $J\in\IPklim$ recursively by iterating through three stages. At each step of the iteration, suppose we have defined $J_\xi$ for all $\xi<\alpha$ and that $\alpha'$ is such that $\Cup_{\xi\in\alpha}J_\xi=\Cup_{\xi\in\alpha'}I_\xi$. In the first stage, we then define $J_\alpha=\Cup\set{I_\delta}{\alpha'\leq\delta<\alpha''}$, where $\alpha''$ is chosen large enough such that $\card{J_\alpha}=\card[big]{\iv{\alpha'}{\alpha''}}$. In the second stage, we define $J_{\alpha+1}=\iv{\min(I_{\alpha''})}{\min(I_\delta)}$, with $\delta>\alpha''$ large enough such that $\card{I_\xi}>\card\alpha$ for all $\xi\geq\delta$. Finally, in the third stage, we define $J_{\alpha+2+\xi}$ for each $\xi<\card{I_\delta}$ to be intervals of length $\alpha$, such that $\Cup_{\xi<\card{I_\delta}}J_{\alpha+2+\xi}$ covers the initial segment of $I_\delta$ of length $\card{I_\delta}$. If at this point $\Cup_{\xi\in\card{I_\delta}}J_{\alpha+2+\xi}\neq I_\delta$, then we define $J_{\alpha+\card{I_\delta}}$ as the remaining part of the interval $I_\delta$. This concludes the third stage, and we continue by iteration until $J_\alpha$ has been defined for all $\alpha\in\kappa$.
		
		\noindent\ref{thm:Ek(B) minmax:4}.\quad Let $\kappa=\lambda^+$ and define $I\in\IPklim$ by $I_\alpha=\iv{\lambda\alpha}{\lambda\alpha+\lambda}$ for each $\alpha\in\kappa$. Now, if $J\in\IPklim$, then for almost all $\delta\in\kappa$ we have $\card{J_\delta}=\lambda$, and thus we see that for almost all $\alpha\in\kappa$ there is a $\delta\in\kappa$ such that $I_\alpha\subset J_\delta\cup J_{\delta+1}$. By  \cref{lmm:Ek(B) comparison}, $\Ek(J)\subset\Ek(I)$.
	\end{proof}
	
	\begin{crl}\label{crl:Ekm is ideal}
		If $\kappa$ is successor, then let $I\in\IPklim$ be such that $\Ek(J)\subset\Ek(I)$ for all $J\in\IPklim$, then $\Ek(I)=\EkE=\Ek$. In particular, $\EkE$ is ${\leq}\kappa$-complete ideal when $\kappa$ is successor.\qed
	\end{crl}
	
	\begin{crl}\label{thm:EkA proper Ek(I)}
		Let $\kappa$ be regular uncountable. Then $\EkA\subsetneq\Ek(I)$ for all $I\in\IPklim$, and $\BEkA\subsetneq\Ek(B)$ for all $B\in\BPklim$.
	\end{crl}
	\begin{proof}
		Fix some $I\in\IPklim$, then \cref{thm:Ek(B) minmax}\,\ref{thm:Ek(B) minmax:1} shows that $\Ek(I)$ is not least, thus there is $J\in\IPklim$ with $\Ek(I)\nsubset\Ek(J)$. It follows from this that $\EkA\subset\Ek(I)\cap \Ek(J)\subsetneq\Ek(I)$. The same argument shows that $\BEkA\subsetneq\Ek(B)$ for all $B\in\BPklim$.
	\end{proof}
	
	We will use the characterisation given in \cref{crl:noninclusion} to show that $\BEkE$ is not an ideal.

	\begin{lmm}\label{lmm:BE additive failure}
		Suppose that 
		\begin{enumerate}[label=(\arabic*)]
		\item\label{lmm:addF:s1} $B,B'\in\BPklim$ are such that $\card{B_\alpha\cap B'_\beta}\leq 1$ for almost all $\alpha,\beta\in\kappa$,
		\item\label{lmm:addF:s2} $\phi\in\SkB$ and $\phi'\in\Sigma_\kappa(B')$ have the property that for any $\alpha,\alpha'\in\kappa$, any $\xi\in B_\alpha\cap B'_{\alpha'}$ and any $i\in\ac{0,1}$, there are $s\in\phi(\alpha)$ and $s'\in\phi'(\alpha')$ such that $s(\xi)=s'(\xi)=i$, and
		\item\label{lmm:addF:s3} $C\in\BPklim$ and $\psi\in\Sigma_\kappa(C)$,
		\end{enumerate} 
		then  $\sqb\psi_*\cap\sqb\phi_0\in\cal M_\kappa(\sqb\phi_0)$ or $\sqb\psi_*\cap\sqb{\phi'}_0\in\cal M_\kappa(\sqb{\phi'}_0)$.\footnote{This lemma originally only showed the existence of such $\phi,\phi'$ for which  $\sqb\phi_*\cup\sqb{\phi'}_*\notin\BEkE$, and was sharpened to its current form by ChatGPT~5.6. The proof presented here, however, was written by the authors.}
	\end{lmm}
	\begin{proof}
		Let $S_\alpha:=\set{\delta\in\kappa}{B_\delta\cap C_\alpha\neq \emp}$ and $S'_\alpha:=\set{\delta\in\kappa}{B'_\delta\cap C_\alpha\neq\emp}$, then we first show that for almost all $\alpha\in\kappa$ we have $\card{S_\alpha}=\card{C_\alpha}$ or $\card{S'_\alpha}=\card{C_\alpha}$. Namely, if $\alpha\in\kappa$ and $\card{S_\alpha}<\card{C_\alpha}$, then we see that \[\ts\card{C_\alpha}=\card[m]{\Cup_{\delta\in S_\alpha}B_\delta\cap C_\alpha}=\card{S_\alpha}\cdot \sup_{\delta\in S_\alpha}\card{B_\delta\cap C_\alpha}=\sup_{\delta\in S_\alpha}\card{B_\delta\cap C_\alpha}.\]
		If $\lambda$ is a cardinal with $\lambda<\card{C_\alpha}$, then there is $\delta\in S_\alpha$ such that $\lambda\leq\card{B_\delta\cap C_\alpha}$. Thus, if $\alpha$ is sufficiently large, there are distinct $\epsilon_\xi$ for each $\xi<\lambda$ for which $\card{B'_{\epsilon_\xi}\cap B_\delta\cap C_\alpha}=1$. Hence $\epsilon_\xi\in S_\alpha'$ for each $\xi<\lambda$, showing that $\lambda\leq\card{S_\alpha'}$. Since $\lambda$ was arbitrary, it follows that $\card{C_\alpha}=\card{S_\alpha'}$.
		
		Now assume that $\card[big]{\set{\alpha\in\kappa}{\card{C_\alpha}=\card{S_\alpha}}}=\kappa$ (the alternative with $S_\alpha'$ in place of $S_\alpha$ is analogous), then we show that $\sqb\psi_\beta\cap \sqb\phi_0$ is nowhere dense in the space $\sqb\phi_0$ for each $\beta\in\kappa$. This suffices to prove the lemma, because $\sqb\psi_*=\Cup_{\beta\in\kappa}\sqb\psi_\beta$. Fix some arbitrary $s\in\fkacs$ such that $[s]\cap \sqb\phi_0\neq\emp$. Let $\alpha_0\in\kappa$ be large enough such that 
		\begin{enumerate}[label=(\roman*)]
			\item\label{lmm:addF:2} for any $\alpha,\beta\in\kappa$ with $\card{B_\alpha\cap B'_{\beta}}>1$ we have $\sup(B_\alpha\cup B'_\beta)<\alpha_0$,
			\item\label{lmm:addF:3} $\sup\rb[big]{\Cup_{\xi\in\beta}C_\xi}<\alpha_0$, and
			\item\label{lmm:addF:1} if $\alpha\in\kappa$ is such that $\dom(s)\cap B_\alpha\neq\emp$, then $\sup(B_\alpha)<\alpha_0$.
		\end{enumerate}
		Now we fix $\alpha\geq \alpha_0$ such that $\card{C_\alpha}=\card{S_\alpha}$. Since $\card{\psi(\alpha)}<2^{\card{C_\alpha}}$, we may fix for each $\delta\in S_\alpha$ some distinct $\epsilon_\delta\in B_\delta\cap C_\alpha$, and some $t\colon\set{\epsilon_\delta}{\delta\in S_\alpha}\to 2$ such that no extension of $t$ is contained in $\psi(\alpha)$. By \ref{lmm:addF:3}, we then see that $[t]\cap \sqb\psi_\beta=\emp$. For each $\delta\in S_\alpha$, fix $t_\delta\in\phi(\delta)$ such that $t_\delta(\epsilon_\delta)=t(\epsilon_\delta)$, which is possible by assumption \ref{lmm:addF:s2} made in the lemma statement. By \ref{lmm:addF:1},  we see that $s':=s\cup \Cup_{\delta\in S_\alpha}t_\delta$ is a well-defined function and that $[s']\cap \sqb\phi_0\neq\emp$. This proves that $\sqb\psi_\beta$ is nowhere dense in $\sqb\phi_0$.
	\end{proof}

	\begin{thm}\label{thm:BEkm is no ideal}Let $\kappa$ be regular uncountable.
		\begin{enumerate}[label=(\arabic*)]
			\item \label{thm:BEKm is no ideal:1}
			There exist $B,B'\in\BPklim$ such that for all $C\in\BPklim$ we have $\Ek(B)\cup\Ek(B')\nsubset\Ek(C)$. 
			\item\label{thm:BEKm is no ideal:2} There exist $B,B'\in\BPklim$, $\phi\in\Sigma_\kappa(B)$ and $\phi'\in\Sigma_\kappa(B')$ such that $\sqb\phi_*\cup\sqb{\phi'}_*\notin\BEkE$.
		\end{enumerate}
	\end{thm}
	\begin{proof}
		\ref{thm:BEKm is no ideal:1}.\quad
		By recursion, we can construct $B,B'\in\BPklim$ be as in \cref{lmm:BE additive failure}, and define  $S_\alpha=\set{\delta\in\kappa}{B_\delta\cap C_\alpha\neq\emp}$ and $S'_\alpha=\set{\delta\in\kappa}{B'_\delta\cap C_\alpha\neq\emp}$. Then, as in the proof of \cref{lmm:BE additive failure}, we see that for any $C\in\BPklim$, for almost every $\alpha\in\kappa$ one of $\card{S_\alpha}=\card{C_\alpha}$ or $\card{S'_\alpha}=\card{C_\alpha}$ holds. Therefore, by \cref{crl:noninclusion}, we have $\Ek(B)\nsubset \Ek(C)$ or $\Ek(B')\nsubset \Ek(C)$.

		\noindent\ref{thm:BEKm is no ideal:2}.\quad Let $B,B'\in\BPklim$ be as in \cref{lmm:BE additive failure} and let $\phi\in\Sigma_\kappa(B)$ contain only the constant functions on each interval, and similarly for $\phi'\in\Sigma_\kappa(B')$, then $\phi$ and $\phi'$ satisfy the assumptions of \cref{lmm:BE additive failure}. For any $C\in\BPklim$ and $\psi\in\Sigma_\kappa(C)$, \cref{lmm:BE additive failure}  tells us that $\sqb\psi_*\cap\sqb\phi_0\in\cal M_\kappa(\sqb\phi_0)$ or $\sqb\psi_*\cap\sqb{\phi'}_0\in\cal M_\kappa(\sqb{\phi'}_0)$. This obviously implies that $\sqb\phi_*\cup\sqb{\phi'}_*\nsubset \sqb\psi_*$, and therefore it follows that \mbox{$\sqb\phi_*\cup\sqb{\phi'}_*\notin\BEkE$}.
	\end{proof}
	
	With a slightly different argument we can show that $\EkE$ is not an ideal either under certain assumptions, including when $\kappa$ is strongly inaccessible, but excluding the successor case given by \cref{crl:Ekm is ideal}. The following lemma is the interval version of \cref{lmm:BE additive failure}.

	\begin{lmm}\label{lmm:E additive failure}
		Suppose that
		\begin{enumerate}[label=(\arabic*)] 
			\item\label{lmm:EaddF-1} $\abset{L_\alpha}{\alpha\in\kappa}\in\IPklim$ is such that $2^{\card{L_\alpha}}<2^{\card{L_\beta}}$ for any $\alpha<\beta\in\kappa$,
			\item\label{lmm:EaddF-2} $x,x'\in\kaka$ are strictly increasing such that $\card{\ran(x)\cap \ran(x')}<\kappa$,
			\item\label{lmm:EaddF-3} $I,I'\in\IPklim$ have $\sup(I_\alpha)=\sup(L_{x(\alpha)+1})$ and $\sup(I'_\alpha)=\sup(L_{x'(\alpha)+1})$ for each $\alpha\in\kappa$,
			\item\label{lmm:EaddF-4} $\phi\in\SkI$ and $\phi'\in\Sigma_\kappa(I')$ are such that 
			\begin{enumerate}[label=(\alph*)]
				\item each $s\colon I_\alpha\setminus L_{x(\alpha)+1}\to 2$ has some $t\in\phi(\alpha)$ such that $s\subset t$, and 
				\item each $s'\colon I'_\alpha\setminus L_{x'(\alpha)+1}\to 2$ has some $t'\in\phi'(\alpha)$ such that $s'\subset t'$,
				\item $2\leq\min\ac{\card{\phi(\alpha)},\card{\phi'(\alpha)}}$ for all $\alpha\in\kappa$,
			\end{enumerate}
			\item\label{lmm:EaddF-5} $B\in\BPk$ and $\psi\in\Sigma_\kappa(B)$,
		\end{enumerate}
		then $\sqb\psi_*\cap\sqb\phi_0\in\cal M_\kappa(\sqb\phi_0)$ or $\sqb\psi_*\cap\sqb{\phi'}_0\in\cal M_\kappa(\sqb{\phi'}_0)$.
	\end{lmm}
	
	Before proving \cref{lmm:E additive failure}, let us make some remarks on the consistency of its assumptions. Clearly, a sequence as in \ref{lmm:EaddF-1} exists for strongly inaccessible $\kappa$, but it also may exist for weakly inaccessible $\kappa$. The situation in which such a sequence exists is useful enough to give a name.
	
	\begin{dfn}
	Define the \emph{continuum function} on cardinal numbers as the function $\lambda\mapsto 2^\lambda$, and say it is \emph{eventually constant} below $\kappa$ if for some $\lambda<\kappa$ we have $2^\lambda=2^{\lambda'}$ for all cardinals $\lambda'\in \iv{\lambda}{\kappa}$.
	\end{dfn}
	Observe that the continuum function is ``eventually'' constant below any successor cardinal, and is not eventually constant below any strongly inaccessible cardinal, although it is independent for weakly inaccessible cardinals.
	
	We also note that \ref{lmm:EaddF-4} can be satisfied if a sequence as in \ref{lmm:EaddF-1} exists: if $\abset{L_\alpha}{\alpha\in\kappa}$ is as in \ref{lmm:EaddF-1}, then $2^{\card{\Cup_{\xi\leq\alpha}L_\alpha}}<2^{\card{L_{\alpha+1}}}$; it follows that the cardinality of the set of all $s\colon I_\alpha\setminus L_{x(\alpha)+1}\to 2$ equals $2^{\card{L_{x(\alpha)}}}$, which is strictly smaller than $2^{\card{I_\alpha}}=2^{\card{L_{x(\alpha)+1}}}$, and thus $\phi$ as in \ref{lmm:EaddF-4} can be found.
	
	\begin{proof}[Proof of \cref{lmm:E additive failure}]
		Let $L:=\Cup_{\xi\in\kappa} (I_\xi\setminus L_{x(\xi)+1})$ and $L':=\Cup_{\xi\in\kappa} (I'_\xi\setminus L_{x'(\xi)+1})$, then \ref{lmm:EaddF-2} implies $\card{\kappa\setminus (L\cup L')}<\kappa$. Therefore, there are either cofinally many $\alpha\in\kappa$ for which $\card{B_\alpha\cap L}=\card{B_\alpha}$ holds, or for which $\card{B_\alpha\cap L'}=\card{B_\alpha}$ holds. Let us assume the former, as the two options are symmetric. We then show that $\sqb\psi_\beta\cap \sqb\phi_0$ is nowhere dense in $\sqb\phi_0$.
		
		Fix $s\in\fkacs$ with $\sqb s\cap \sqb\phi_0\neq\emp$ and choose $\alpha\in\kappa$ such that
		\begin{enumerate}[label=(\roman*)]
			\item\label{lmm:EaddF-s1} $\card{B_\alpha\cap L}=\card{B_\alpha}$, and
			\item\label{lmm:EaddF-s2} if $\xi\in\kappa$ is such that $I_\xi\cap B_\alpha\neq \emp$, then $\dom(s)<\min(I_\xi)$.
		\end{enumerate}
		Since $\card{\psi(\alpha)}<2^{\card{B_\alpha}}$, we may choose $t\colon B_\alpha\cap L\to 2$ such that no extension of $t$ is in $\psi(\alpha)$. For each $\xi\in\kappa$ such that $I_\xi\cap B_\alpha\neq \emp$, let $t_\xi = t\restriction I_\xi$, then assumption \ref{lmm:EaddF-4} implies that there exists $t'_\xi\in\phi(\xi)$ such that $t_\xi\subset t'_\xi$. Moreover, \ref{lmm:EaddF-s2} implies that $s':= s\cup \Cup \set{t_\xi'}{\xi\in\kappa\text{ and }I_\xi\cap B_\alpha\neq\emp}$ is a well-defined function, and we see that $\sqb{s'}\cap \sqb\phi_0\neq\emp$. Since $t\subset s'$ we have $\sqb{s'}\cap \sqb\psi_\beta=\emp$. Finally, $s$ and $\beta$ were arbitrary, so $\sqb \psi_*\cap \sqb\phi_0\in\Mk(\sqb\phi_0)$.
	\end{proof}

	\begin{thm}\label{thm:eveq equi}
		The following are equivalent for regular uncountable $\kappa$.
		\begin{enumerate}[label=(\arabic*)]
			\item\label{eveq equi-1} The continuum function is eventually constant below $\kappa$.
			\item\label{eveq equi-2} The family $\EkE$ is closed under finite unions.
			\item\label{eveq equi-3} For every $B\in\BPklim$ there exists $I\in\IPklim$ with $\Ek(I)\subset\Ek(B)$.
		\end{enumerate}
	\end{thm}
	\begin{proof}
		\ref{eveq equi-1} $\imp$ \ref{eveq equi-2}.\quad If $\kappa$ is successor, then this follows from \cref{crl:Ekm is ideal}, so let us assume that $\kappa$ is weakly inaccessible, and that $\lambda<\kappa$ is such that $2^{\lambda'}=2^\lambda$ for all cardinals $\lambda'\in\iv\lambda\kappa$. Suppose that $I,I'\in\IPklim$ and $\phi\in\SkI$ and $\phi'\in\Sigma_\kappa(I')$. Let $J=\abset{J_\alpha}{\alpha\in\kappa}$ be the enumeration of the set $\set{I_\alpha\cap I'_{\alpha'}}{\alpha,\alpha'\in\kappa}\setminus\ac\emp$ such that $J\in\IPk$. We don't necessarily have $J\in\IPklim$; however, it is easy to see that $\lim_{\alpha\to\kappa}\,\card{J_\alpha\cup J_{\alpha+1}}=\kappa$, and thus we may find $J'\in\IPklim$ such that  for each $\alpha'\in\kappa$, there is $\alpha\in\kappa$ such that $J'_{\alpha'}=J_\alpha$ or $J'_{\alpha'}=J_{\alpha}\cup J_{\alpha+1}$.
		
		Now define $\psi\in\Sigma_\kappa(J')$ as follows. Given $J'_\alpha$, either there is one or there are two $\xi\in\kappa$ such that $J'_\alpha\cap I_\xi\neq\emp$. For any such $\xi$, let $\set{s\restriction J'_\alpha}{s\in \phi(\xi)}\subset \psi(\alpha)$. We similarly incorporate $\phi'(\xi')$ into $\psi(\alpha)$ for the relevant one or two $\xi'$. If $\alpha$ is large enough, then $2^{\card{J'_\alpha}}=2^{\card{I_\xi}}$, and hence $\card{\phi(\xi)}<2^{\card{J'_\alpha}}$, hence we may find such $\psi$ with $\psi\in\Sigma^*_\kappa(J')$. It follows that $\sqb \phi_*\cup\sqb{\phi'}_*\subset \sqb\psi_*$.

		\noindent\ref{eveq equi-2} $\imp$ \ref{eveq equi-1}.\quad We prove the contrapositive. If the continuum function is not eventually constant below $\kappa$, then the assumptions of \cref{lmm:E additive failure} can be satisfied. The rest follows as in \cref{thm:BEkm is no ideal}\,\ref{thm:BEKm is no ideal:2}, using \cref{lmm:E additive failure} instead of \cref{lmm:BE additive failure}.
		
		\noindent\ref{eveq equi-1} $\imp$ \ref{eveq equi-3}.\quad Assume that the continuum function is eventually constant below $\kappa$, and let $\lambda<\kappa$ be such that $2^\lambda=2^{\lambda'}$ for all $\lambda'\in\iv{\lambda}{\kappa}$. Given $B\in\BPklim$, we see that $\card{B_\alpha}\geq\lambda$ for almost all $\alpha\in\kappa$. Let $I\in\IPk$ be defined as in \cref{capturing BP with IP} such that every $I_\delta$ has a set $A_\delta\subset\kappa$ with $I_\delta=\Cup_{\alpha\in A_\delta}B_\alpha$, and let $\delta_\alpha$ be such that $B_\alpha\subset I_{\delta_\alpha}$. Hence, if $S_\alpha:=\set{\delta\in\kappa}{I_\delta\cap B_\alpha\neq\emp}$, then $S_\alpha=\ac{\delta_\alpha}$. Note that $I\in\IPklim$, and thus for large enough $\alpha$ we see that $\card{I_{\delta_\alpha}}\geq \lambda$, which implies that $2^{\card{I_{\delta_\alpha}}}=2^{\card{B_\alpha}}$ for almost all $\alpha\in\kappa$. Therefore, for any $\abset{\kappa_{\alpha,\delta}}{\alpha\in\kappa, \delta\in S_\alpha}$ such that $\kappa_{\alpha,\delta}<2^{\card{I_{\delta}}}$, we see that, for almost all $\alpha\in\kappa$, 
		\[\ts\prod_{\delta\in S_\alpha}\kappa_{\alpha,\delta}=\kappa_{\alpha,\delta_\alpha}<2^{\card{I_{\delta_\alpha}}}=2^{\card{B_\alpha}}.\] Thus, by \cref{lmm:Ek(B) comparison} it follows that $\Ek(I)\subset\Ek(B)$.

		\noindent\ref{eveq equi-3} $\imp$ \ref{eveq equi-1}.\quad 
		Suppose $\abset{\lambda_\alpha}{\alpha\in\kappa}$ is a sequence of infinite cardinals  below $\kappa$ such that $2^{\lambda_\alpha}<2^{\lambda_{\alpha'}}$ for each $\alpha<\alpha'$. We define disjoint sets $B_\alpha\subset\kappa$ such that for each $\alpha\in\kappa$ we have $\card{B_\alpha}=\lambda_{\alpha}$, but such that for any $\zeta\in B_\alpha$ we have $\iv{\zeta}{\zeta+\lambda_{\alpha+1}}\cap B_\alpha=\ac\zeta$. It is not hard to see that we can additionally require $\Cup B_\alpha=\kappa$, meaning that $B:=\abset{B_\alpha}{\alpha\in\kappa}\in\BPklim$. Let $I\in\IPklim$ and $\delta\in\kappa$. If $\card{I_\delta}<\lambda_{\alpha+1}$, then $2^{\card{I_\delta\cap B_\alpha}}\leq 2^1<2^{\card{I_\delta}}$, and otherwise we have  $2^{\card{I_\delta\cap B_\alpha}}\leq2^{\card{B_\alpha}}=2^{\lambda_\alpha}< 2^{\card{I_\delta}}$. By \cref{crl:noninclusion}\,\ref{crl:noninclusion:powers}, we see that $\Ek(I)\nsubset\Ek(B)$.
	\end{proof}
	
	\begin{crl}
		If $\kappa$ is regular uncountable and the continuum function is not eventually constant below $\kappa$, then $\Ek\nsubset\BEkE$.
	\end{crl}
	\begin{proof}
		This holds, because the set $\sqb\phi_*\cup\sqb{\phi'}_*$ constructed in \cref{thm:eveq equi} \ref{eveq equi-2} $\imp$ \ref{eveq equi-1} is contained in $\Ek\setminus\BEkE$.
	\end{proof}

	Finally, \cref{thm:eveq equi} allows us to conclude something about the inclusion $\BEkA\subset\EkA$. 
	
	\begin{thm}\label{thm:successor BEkA=EkA}
		 If the continuum function is eventually constant below $\kappa$, then $\EkA=\BEkA$.
	\end{thm}
	\begin{proof}
		This follows as in \cref{thm:eveq equi} \ref{eveq equi-1} $\imp$ \ref{eveq equi-3} and the definitions of $\EkA$ and $\BEkA$.
	\end{proof}

	\section{Additive ideals}\label{sec:additive ideals}
	
	In this section we will study additive closure of the ideals from \cref{dfn:families}.
	
	\subsection{Definitions of additive ideals and related ideals}\label{sec:additive ideals definitions}
	
	For $x,y\in\kacs$, let us define addition on $\kacs$ by the sum $x+y$ being the function that maps $\alpha\mapsto 0$ if $x(\alpha)=y(\alpha)$ and $\alpha\mapsto 1$ otherwise. We expand this to subsets $X,Y\subset\kacs$ by defining 
	\begin{align*}
		x+Y&=\set{x+y}{y\in Y},&
		X+y&=\set{x+y}{x\in X},\\
		X+Y&=\set{x+y}{x\in X\text{ and }y\in Y}.
	\end{align*}
	For an ideal $\cal I_\kappa$ on $\kacs$, we call a set $X\subset\kacs$ an \emph{$\cal I_\kappa$-additive set} if the sum $X+Y$ is in $\cal I_\kappa$ for all $Y\in\cal I_\kappa$. The family of $\cal I_\kappa$-additive sets is an ideal, which we will denote by $\cal{IA}_\kappa$. We call an ideal that is closed under addition an \emph{additive ideal}. We make some easy observations.
	\begin{prp}\leavevmode
		\begin{enumerate}
			\item $\cal {IA}_\kappa\subset\cal I_\kappa$,
			\item $\cal{IAA}_\kappa=\cal{IA}_\kappa$, that is, the family of $\cal{IA}_\kappa$-additive sets is simply $\cal{IA}_\kappa$,
			\item If $\cal I_\kappa$ is ${<}\lambda$-complete, then so is $\cal{IA}_\kappa$,
			\item If $X+Y\in\cal I_\kappa$ for all $X,Y\in\cal I_\kappa$, then $\cal I_\kappa=\cal{IA}_\kappa$.\qed
		\end{enumerate}
	\end{prp} 
	In this section, we will study the additive counterparts to the ideals defined in \cref{dfn:families}. Before we do so,  we will need two additional ideals.

	\begin{dfn}
		A set $X\subset\kacs$ is \emph{strongly $\kappa$-null} if for every $A\in\sqb\kappa^\kappa$ there exists $s_\alpha\in{}^\alpha2$ for every $\alpha\in A$ such that $X\subset\Cup_{\alpha\in A}\sqb{s_\alpha}$. We write $\SNk$ for the ideal of strongly $\kappa$-null subsets of $\kacs$. 
	\end{dfn}
	
	The above definition is a straightforward generalisation of the strongly null ideal on $\omcs$, or of the ideal of strong measure zero sets on the reals. The $\kappa$-version of this ideal has been previously considered in \cite{Halko1996,HalkoShelah2001,ChapmanSchurz2024}.
	
	\begin{dfn}\label{dfn:NAk}
		A set $X\subset\kacs$ is \emph{$\kappa$-null-additive} if for every $I\in\IPk$ there exists $\phi\in\prod_{\alpha\in\kappa}\sqb[big]{{}^{I_\alpha}2}^{\leq\card\alpha}$ such that $X\subset\sqb\phi_*$. We write $\NAk$ for the ideal of $\kappa$-null-additive subsets of $\kacs$.
	\end{dfn}
	
	As the name suggests, classically the property described above is equivalent to Lebesgue measure null-additive subsets of $\omcs$, as was shown by \citet{Shelah1995}. The above characterisation was used to show that Lebesgue null-additive sets are  meagre-additive, and the proof generalises without any hardship. 
	\begin{thm}[cf.\ \cite{Shelah1995}]
		$\NAk\subset\MAk$.\qed
	\end{thm}
	
	The Galvin--Mycielski--Solovay theorem \cite{GalvinMycielskiSolovay1979,GalvinMycielskiSolovay2017} connects meagre-additive sets to the strongly null ideal. It generalises to the higher context for strongly inaccessible $\kappa$:
	
	\begin{thm}[\citet{Wohofsky,Schurz2023}]\label{thm:GMS}
		If a set $X\subset\kacs$ has the property that for every $M\in\Mk$ we have $X+M\neq\kacs$, then $X\in\SNk$. If $\kappa$ is strongly inaccessible, then the reverse holds as well. \qed
	\end{thm}
	\begin{crl}
		$\MAk\subset\SNk$.\qed
	\end{crl}
	
	\subsection{Higher closed null-additive}\label{sec:higher closed null additive}
	
	Note that the definition of $\NAk$ essentially involves an intersection of a family of ideals based on slaloms on interval partitions, so it should come as no surprise that $\NAk$ is related to the ideals $\EkA$ and $\BEkA$. Specifically, we will show that $\NAk\subset\BEkA$. Before we do so, we will prove the weaker claim that $\NAk\subset\EkA$ and use that, for successor $\kappa$, $\EkA=\BEkA$ by \cref{thm:successor BEkA=EkA}.
	
	\begin{lmm}\label{thm:NAk subset EkA}
		Let $\kappa$ be regular uncountable, then $\NAk\subset\EkA$.
	\end{lmm}
	\begin{proof}
		Let $A\in\NAk$ and fix some arbitrary $I\in\IPklim$. We will find $\phi\in\SkI$ such that $A\subset\sqb\phi_*$.
		
		We inductively define a continuous increasing sequence $\abset{\alpha_\xi}{\xi\in\kappa}$ of ordinals in $\kappa$ such that for all $\xi\in\kappa$ and $\alpha\geq\alpha_\xi$ we have $\card \xi\leq \card{I_\alpha}$. Define $J\in\IPklim$ by $J_\xi=\Cup\set{I_\alpha}{\alpha_\xi\leq\alpha<\alpha_{\xi+1}}$ for each $\xi\in\kappa$. Since $A\in\NAk$, there is $\psi\in\prod_{\xi\in\kappa}\sqb[big]{^{J_\xi}2}^{\card\xi}$ such that $A\subset\sqb\psi_*$. Let $\phi$ be defined such that for any $\xi,\alpha\in\kappa$ with $\alpha\in\iv{\alpha_\xi}{\alpha_{\xi+1}}$ we have $\phi\colon \alpha\mapsto\set{s\restriction I_\alpha}{s\in \psi(\xi)}$. Note that \[\card{\phi(\alpha)}\leq\card{\psi(\xi)}\leq\card\xi\leq\card{I_\alpha}<2^{\card{I_\alpha}}.\]Therefore $\phi\in\SkI$, and since $\sqb\psi_*\subset\sqb\phi_*$ it follows that $A\in\Ek(I)$. Finally, as $I$ was arbitrary, we see that $\NAk\subset\EkA$.
	\end{proof}

	\begin{thm}\label{thm:NAk subset BEkA}
		Let $\kappa$ be regular uncountable, then $\NAk\subset\BEkA$.
	\end{thm}
	\begin{proof}
		If $\kappa$ is successor, then by \cref{thm:NAk subset EkA,thm:successor BEkA=EkA} it follows that $\NAk\subset\BEkA$. Let us therefore assume that $\kappa$ is regular limit. Let $A\in\NAk$ and fix some arbitrary $B\in\BPklim$, then we will find $\phi\in\SkB$ such that $A\subset\sqb\phi_*$.
		
		Since $\lim_{\beta\in\kappa}\card{B_\beta}=\kappa$ is limit, we may define $\beta_\xi=\sup\rb[big]{\Cup\set{B_\beta}{\beta\in\kappa\text{ and }\card{B_\beta}\leq\card\xi}}$.\footnote{We assume that $\sup(\emp)=0$.} Let $f\in\kaka$ be defined by 
		\[f\colon \alpha\mapsto \min\set[big]{\gamma\geq\alpha}{\forall\beta\in\kappa\,\rb{B_\beta\cap\gamma\neq\emp\to B_\beta\subset\gamma}}.\]
		By recursion on $\alpha\in\kappa$, we define a continuous strictly increasing sequence of ordinals $\abset{\gamma_\alpha}{\alpha\in\kappa}$ such that $\gamma_0=0$ and $\gamma_{\alpha+1}=f(\max\ac{\gamma_\alpha,\beta_\alpha}+1)$. Let $J\in\IPklim$ be defined by $J_\alpha=\iv{\gamma_\alpha}{\gamma_{\alpha+1}}$.
		
		Since $A\in\NAk$, there exists $\psi\in\prod_{\alpha\in\kappa}\sqb[big]{{}^{J_\alpha}2}^{\leq\card\alpha}$ such that $A\subset\sqb\psi_*$. Now define $\phi$ such that for any $\alpha,\beta\in\kappa$ with $B_\beta\subset J_\alpha$ we have $\phi\colon\alpha\mapsto\set{s\restriction B_\beta}{s\in\psi(\alpha)}$. Since $B_\beta\subset J_\alpha$ implies that for each $\xi\in\alpha$ we have $B_\beta\cap \beta_\xi=\emp$, it follows from the construction of $\beta_\xi$ that $\card\alpha\leq\card{B_\beta}$. In particular, if $B_\beta\subset J_\alpha$, then
		\[\card{\phi(\beta)}\leq\card{\psi(\alpha)}\leq\card\alpha\leq\card{B_\beta}<2^{\card{B_\beta}}.\]
		It follows that $\phi\in\SkB$ and $\sqb\psi_*\subset\sqb\phi_*$, so $A\in\Ek(B)$ and since $B$ was arbitrary $\NAk\subset\BEkA$.
	\end{proof}

	\begin{thm}\label{thm:EkA subset EAk}
		Let $\kappa$ be regular uncountable, then $\EkA\subset\EAk\subset\Ek$.
	\end{thm}
	
	\begin{proof}
		That $\EAk\subset\Ek$ is clear.
		
		Assume $X\in\EkA$ and for each $I\in\IPklim$ pick some $\psi^I\in\SkI$ such that $X\subset\Cap_{I\in\IPklim}\sqb[big]{\psi^I}_*$. Let $E\in\Ek$ be arbitrary, then there is a sequence $\abset{I_\beta}{\beta\in\kappa}$ in $\IPklim$ and $\phi^\beta\in\Sigma_\kappa(I_\beta)$ for each $\beta\in\kappa$ such that $E\subset\Cup_{\beta\in\kappa}\sqb[big]{\phi^\beta}_*$. Given $\beta\in\kappa$, let us define $\chi^\beta\colon \alpha\mapsto\phi^\beta(\alpha)+\psi^{I_\beta}(\alpha)$, and note that $\chi^\beta\in\Sigma_\kappa(I_\beta)$. Finally note that \[E+X\subset\Cup\limits_{\beta\in\kappa}\sqb[big]{\phi^\beta}_*\ +\!\Cap\limits_{I\in\IPklim}\sqb[big]{\psi^I}_*\subset\Cup\limits_{\beta\in\kappa}\rb[m]{\sqb[big]{\phi^\beta}_*+\sqb[big]{\psi^{I_\beta}}}=\Cup\limits_{\beta\in\kappa}\sqb[big]{\chi^\beta}_*\in\Ek.\qedhere\]
	\end{proof}
	
	Both of the inclusions given in \cref{thm:EkA subset EAk} are strict for certain choices of $\kappa$, albeit for different choices. If $\kappa$ is strongly inaccessible, then we will see that $\EAk\subsetneq\Ek$ as a consequence of \cref{EkE adds to everything}. The case for successor $\kappa$ is different. Note that ideals of the form $\Ek(B)$ for $B\in\BPklim$ are additively closed (which also implies that $\EkA$ and $\BEkA$ are additive ideals).
	\begin{lmm}
		If $\kappa$ is successor, then $\EkA\subsetneq \EAk=\Ek$.
	\end{lmm}
	\begin{proof}
		By \cref{crl:Ekm is ideal} we may choose $I\in\IPklim$ with $\Ek=\Ek(I)$, and since $\Ek(I)$ is additive, we see that $\Ek=\EAk$. It then follows from \cref{thm:EkA proper Ek(I)} that $\EkA\subsetneq\Ek(I)=\Ek$.
	\end{proof}

	\subsection{Decomposing the higher Cantor space}\label{sec:decomposing}
	
	In this subsection we study some cases where additivity fails, and where large sets can be decomposed into the sum of small sets. As mentioned above, $\Ek(B)$ is an additive ideal for any $B\in\BPklim$. We will show that $\EkE$ does not have to be additive (and thus neither does $\Ek$), and $\BEkE$ is not additive for any regular uncountable $\kappa$.

	\begin{thm}\label{EkE adds to everything}
		Let $\kappa$ be regular limit, then the following are equivalent:
		\begin{enumerate}
			\item there exist $X,Y\in\EkE$ for which $X+Y=\kacs$.
			\item the continuum function is not eventually constant below $\kappa$,
		\end{enumerate}
	\end{thm}
	\begin{proof}
		We first assume that the continuum function is not eventually constant. Let $\abset{\kappa_\alpha}{\alpha\in\kappa}$ be a sequence of cardinals such that $2^{\kappa_\alpha}<2^{\kappa_\beta}$ for all $\alpha<\beta$. We define $I,J\in\IPklim$, where for each $\alpha\in\kappa$ we have
		\begin{align*}
			J_\alpha&=\iv{\kappa_{2\alpha}}{\kappa_{2\alpha+2}},&
			I_\alpha&=\begin{cases}
				\iv{\kappa_{2\alpha}}{\kappa_{2\alpha+3}} & \text{ if }\alpha\text{ is limit,}\\
				\iv{\kappa_{2\alpha+1}}{\kappa_{2\alpha+3}} & \text{ otherwise.}
			\end{cases}
		\end{align*}
		Now we define the sets
		\begin{align*}
			X&:=\set[big]{x\in\kacs}{\forall\alpha\in\kappa\fforall\xi\in\iv{\kappa_{2\alpha}}{\kappa_{2\alpha+1}}\,\rb{x(\xi)=0}},\\
			Y&:=\set[big]{y\in\kacs}{\forall\alpha\in\kappa\fforall\xi\in\iv{\kappa_{2\alpha+1}}{\kappa_{2\alpha+2}}\,\rb{y(\xi)=0}}.
		\end{align*}
		Then $X\in \Ek(I)$ and $Y\in \Ek(J)$, and $X+Y=\kacs$.
		
		Now let us assume that the continuum function is eventually constant, and fix $\lambda<\kappa$ such that $2^{\lambda'}\leq2^\lambda$ for all $\lambda'<\kappa$. Let $X\in\Ek(I)$ and $Y\in\Ek(J)$ for some $I,J\in\IPklim$, then there exist two strictly increasing sequences $\abset{\alpha_\xi}{\xi\in\kappa}$ and $\abset{\beta_\xi}{\xi\in\kappa}$ such that, for all $\xi\in\kappa$, \[2^{\card[big]{I_{\alpha_\xi}}}=2^{\card[big]{J_{\beta_\xi}}}=2^{\card[big]{I_{\alpha_\xi}\cap J_{\beta_\xi}}}=2^\lambda.\]
		But, since $\card[m]{X\restriction I_{\alpha_\xi}}<2^\lambda$ and $\card[m]{Y\restriction J_{\beta_\xi}}<2^\lambda$, we also see that $(X+Y)\restriction (I_{\alpha_\xi}\cap J_{\beta_\xi})<2^\lambda$, and thus $X+Y\neq\kacs$.
	\end{proof}
	
	\begin{thm}
		Let $\kappa$ be regular uncountable, then there exist $X,Y\in\BEkE$ such that $X+Y=\kacs$.
	\end{thm}
	\begin{proof}
		We define $B,C\in\BPklim$ and $\phi\in\Sigma_\kappa(B)$ and $\psi\in\Sigma_\kappa(C)$ in such a way that for every $\xi,\xi'\in\kappa$ we have $\card{B_\xi\cap C_{\xi'}}\leq 1$, and there exist unique $b_\xi\in B_\xi$ and $c_{\xi'}\in C_{\xi'}$ such that 
		\begin{enumerate}
			\item $s\in \phi(\xi)$ if and only if $s\colon B_\xi\to 2$ and $s(\eta)=0$ for all $\eta\neq b_\xi$,
			\item $t\in \psi(\xi')$ if and only if $t\colon C_{\xi'}\to 2$ and $t(\eta)=0$ for all $\eta\neq c_{\xi'}$, and
			\item $B_\xi\cap C_{\xi'}\subset \ac{b_\xi,c_{\xi'}}$.
		\end{enumerate}
		It is clear that such $B,C,\phi,\psi$ can be constructed recursively. Note that $\card{\phi(\xi)}=\card{\psi(\xi')}=2$ for all $\xi,\xi'\in\kappa$, hence $\phi\in\Sigma_\kappa(B)$ and $\psi\in\Sigma_\kappa(C)$. We also note that $\kappa=\set{b_\xi}{\xi\in\kappa}\cup \set{c_{\xi'}}{\xi'\in\kappa}$, thus if $z\in\kacs$ and if we define $z_B,z_C\in \kacs$ by 
		\begin{align*}
			&z_B\colon\alpha\mapsto\begin{cases}
				z(\alpha)&\text{ if }\alpha\in\set{b_\xi}{\xi\in\kappa},\\
				0 & \text{ otherwise,}
			\end{cases}&
			&z_C\colon\alpha\mapsto\begin{cases}
				z(\alpha)&\text{ if }\alpha\notin\set{b_\xi}{\xi\in\kappa},\\
				0 & \text{ otherwise,}
			\end{cases}
 		\end{align*}
 		then $z=z_B+z_C$ and $z_B\in\sqb\phi_*$ and $z_C\in\sqb\psi_*$. Thus $\sqb\phi_*+\sqb\psi_*=\kacs$.
	\end{proof}
	
	Next we will show that $\Ek(B)$ is not $\kappa$-meagre-additive.
	
	\begin{thm}
		For every $B\in\BPklim$ there are $X\in\Ek(B)$ and $N\in\NWD$ with $X+N=\kacs$.
	\end{thm}
	\begin{proof}
		Let $\phi\in\SkB$ be defined by $s\in \phi(\alpha)$ if and only if $s\colon B_\alpha\to 2$ and $\forall \xi,\xi'\in B_\alpha\,\rb{s(\xi)=s(\xi')}$. Define $N\subset\kacs$ by $x\in N$ if and only if for all $\alpha\in\kappa$ there are $\xi,\xi'\in B_\alpha$ with $x(\xi)\neq x(\xi')$. It is easy to see that $N\in\NWD$ and $\sqb\phi_*+N=\kacs$.
	\end{proof}
	\begin{crl}\label{crl:Ek(B) not subset MAk}
		Let $\kappa$ be regular uncountable and $B\in\BPklim$, then $\Ek(B)\nsubset\MAk$. In particular, $\EkE\nsubset\MAk$.\qed
	\end{crl}
	
	Using the Galvin--Mycielski--Solovay characterisation of $\SNk$  (\cref{thm:GMS}), we may strengthen the above corollary and conclude that $\Ek(B)\nsubset\SNk$ for any $I\in\BPklim$ if $\kappa$ is strongly inaccessible. On the other hand, it is consistent with $\kappa$ strongly inaccessible that $\EkA\subset\SNk$:
	
	\begin{thm}
		Assume $\sf{GCH}$. If $\kappa$ is strongly inaccessible, then $\EkA\subset\SNk$.
	\end{thm}
	\begin{proof}
		Fix some $X\in\EkA$ and let $A\in\sqb\kappa^\kappa$, then we will first find a suitable $I\in\IPklim$ by inductively defining a continuous strictly increasing sequence of cardinals $\abset{\lambda_\alpha}{\alpha\in\kappa}$ with $\lambda_0=0$ and letting $I_\alpha=\iv{\lambda_\alpha}{\lambda_{\alpha+1}}$. We only need to consider the successor case, so assume we have defined $\lambda_\alpha$, then we choose $\lambda_{\alpha+1}$ such that $\lambda_{\alpha+1}= \card[big]{A\cap I_\alpha}$; for instance, let $j_0>\lambda_\alpha$ be large enough such that $\lambda_\alpha^+\leq \card[big]{A\cap \iv{\lambda_\alpha}{j_0}}$, for each $n\in\omega$ pick $j_{n+1}$ large enough such that $\card{j_{n}}<\card[big]{A\cap \iv{\lambda_\alpha}{j_{n+1}}}$, and let $\lambda_{\alpha+1}=\sup_{n\in\omega}j_n$.
		
		Since $X\in\EkA$, there exists $\phi\in\SkI$ such that $X\subset\sqb\phi_*$. For a given $\alpha\in\kappa$, let us define $X_\alpha:=\set{x\in X}{x\restriction I_\alpha\in \phi(\alpha)}$.  Since $\card{\phi(\alpha)}<2^{\card{I_\alpha}}=\lambda_{\alpha+1}^+$ and $\card[m]{{}^{\lambda_\alpha}2}=\lambda_\alpha^+\leq \lambda_{\alpha+1}$, we can enumerate $X_\alpha\restriction\lambda_{\alpha+1}$ as $\abset{x_\xi}{\xi\in \lambda_{\alpha+1}}$ and  $A\cap I_\alpha$ as $\abset{a^\alpha_\xi}{\xi\in \lambda_{\alpha+1}}$. Then  $X_\alpha\subset\Cup_{\xi\in\lambda_{\alpha+1}}[x_\xi\restriction a^\alpha_\xi]$. Finally note that $X=\Cup_{\alpha\in\kappa}X_\alpha$. Since $A$ was arbitrary, it follows that $X\in\SNk$.
	\end{proof}

	We remark that in the classical case $\cal{MA}=\cal{EA}$, as was shown by \citet{Zindulka2019}. For successor~$\kappa$ this cannot be generalised, since successor $\kappa$ satisfy $\EAk=\Ek\nsubset\MAk$ by \cref{crl:Ek(B) not subset MAk}, but we do not know more than this.

	\section{Cardinal characteristics}\label{sec:cardinal characteristics}

	In this section we will study the cardinal characteristics of the families from \cref{dfn:families} and compare them to several well-known cardinal characteristics of the higher Baire space. For the $\sf{ZFC}$-provable results, it is often possible to provide a Galois--Tukey connection between two relational systems in order to compare cardinal characteristics (cf.\ \cite{Blass1994}). Although the use of relational systems provides more general theorems, we will, for the sake of clarity, give our proofs in a direct fashion, and leave the translation to relational systems to the reader.

	\begin{dfn}
		Let $\cal I_\kappa$ be a ${\leq}\kappa$-complete ideal on $\kacs$, then we define the following cardinal characteristic parametrised by $\cal I_\kappa$:
		\begin{enumerate}
			\item The \emph{additivity number} $\add(\cal I_\kappa)=\min\set[big]{\card{\cal A}}{\cal A\subset\cal I_\kappa\text{ and }\Cup \cal A\notin\cal I_\kappa}$,
			\item The \emph{covering number} $\cov(\cal I_\kappa)=\min\set[big]{\card{\cal C}}{\cal C\subset\cal I_\kappa\text{ and }\Cup \cal C=\kacs}$,
			\item The \emph{uniformity number} $\non(\cal I_\kappa)=\min\set[big]{\card N}{N\subset\kacs\text{ and }N\notin\cal I_\kappa}$,
			\item The \emph{cofinality number} $\cof(\cal I_\kappa)=\min\set[big]{\card{\cal F}}{\cal F\subset\cal I_\kappa\text{ and }\forall X\in\cal I_\kappa\eexists Y\in \cal F\,(X\subset Y)}$.
		\end{enumerate}
	\end{dfn}

	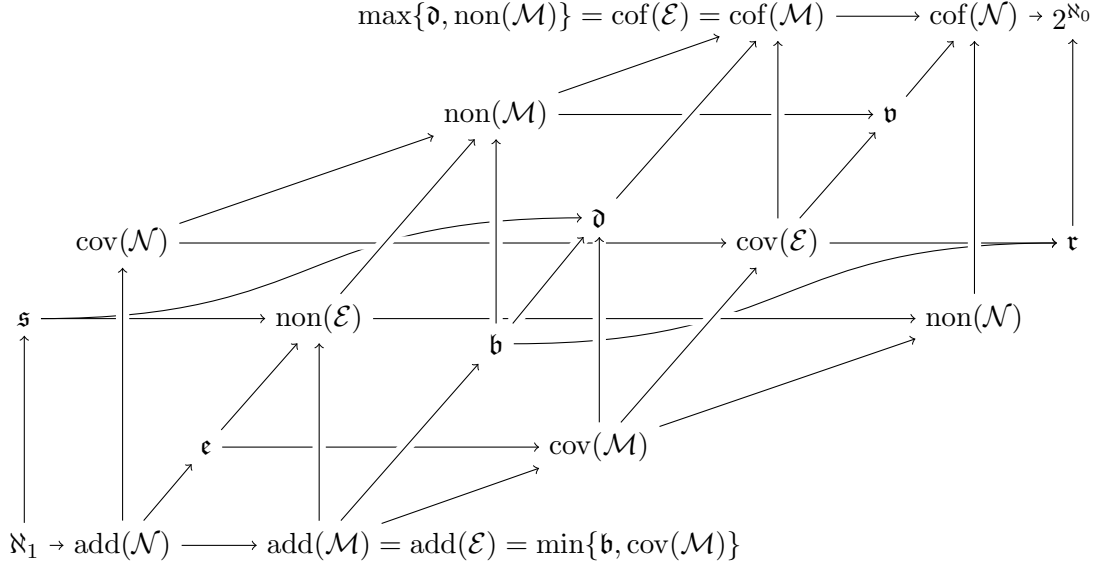
\begin{figure}[t]
		\begin{tikzpicture}[xscale=1.3,yscale=1]
			\node (a1) at (0,0) {$\aleph_1$};
			\node (c)  at (10.667,7) {$2^{\aleph_0}$};
			
			\node (aN) at (1,0) {$\add(\cal N)$};
			\node (fN) at (9.667,7) {$\cof(\cal N)$};
			
			\node (s)  at (0,3) {$\fr s$};
			\node (r)  at (10.667,4) {$\fr r$};
			
			\node (cN) at (1,4) {$\cov(\cal N)$};
			\node (nN) at (9.667,3) {$\non(\cal N)$};
			
			\node (aM) at (3,0) {$\add(\cal M)$\rlap{${}=\add(\cal E)=\min\ac{\fr b,\cov(\cal M)}$}};
			\node (fM) at (7.667,7) {\llap{$\max\ac{\fr d,\non(\cal M)}=\cof(\cal E)={}$}$\cof(\cal M)$};
			
			\node (nE) at (3,3) {$\non(\cal E)$};
			\node (cE) at (7.667,4) {$\cov(\cal E)$};
			
			\node (b)  at (4.8,2.667) {$\fr b$};
			\node (d)  at (5.85,4.333) {$\fr d$};
			
			\node (cM) at (5.85,1.3) {$\cov(\cal M)$};
			\node (nM) at (4.8,5.7) {$\non(\cal M)$};
			
			\node (e)  at (1.85,1.3) {$\fr e$};
			\node (v)  at (8.817,5.7) {$\fr v$};

			\draw (a1) edge[->] (aN);
			\draw (fN) edge[->] (c);
			
			\draw (cE) edge[->] (fM);
			\draw (aM) edge[->] (nE);
			
			\draw (cM) edge[->] (cE);
			\draw (nE) edge[->] (nM);
			
			\draw (cN) edge[->] (cE);
			\draw (nE) edge[->] (nN);
			
			\draw (b)  edge[->,in=180,out=0,looseness=1.3] (r);
			\draw (s)  edge[->,in=180,out=0,looseness=1.3] (d);
			
			\draw (cE) edge[->] (r);
			\draw (s)  edge[->] (nE);
			
			\draw (a1)  edge[->] (s);
			\draw (r)  edge[->] (c);
			
			\draw (e)  edge[->] (cM);
			\draw (nM)  edge[->] (v);
			
			\draw (aN) edge[->] (e);
			\draw (v) edge[->] (fN);
			
			\draw (e)  edge[->] (nE);
			\draw (cE)  edge[->] (v);

			\draw (cN) edge[-,white,line width=5pt] (cE);
			\draw (nE) edge[-,white,line width=5pt] (nN);
			\draw (cN) edge[->] (cE);
			\draw (nE) edge[->] (nN);
			
			\draw (cE) edge[-,white,line width=5pt] (fM);
			\draw (cM) edge[-,white,line width=5pt] (cE);
			\draw (cE) edge[->] (fM);
			\draw (cM) edge[->] (cE);
			
			\draw (aM) edge[-,white,line width=5pt] (nE);
			\draw (nE) edge[-,white,line width=5pt] (nM);
			\draw (aM) edge[->] (nE);
			\draw (nE) edge[->] (nM);
			
			\draw (aM) edge[-,white,line width=5pt] (b);
			\draw (d)  edge[-,white,line width=5pt] (fM);
			\draw (aM) edge[->] (b);
			\draw (d)  edge[->] (fM);
			
			\draw (aN) edge[-,white,line width=5pt] (cN);
			\draw (cN) edge[-,white,line width=5pt] (nM);
			\draw (aN) edge[->] (cN);
			\draw (cN) edge[->] (nM);
			
			\draw (aN) edge[-,white,line width=5pt] (aM);
			\draw (fM) edge[-,white,line width=5pt] (fN);
			\draw (aN) edge[->] (aM);
			\draw (fM) edge[->] (fN);
			
			\draw (b)  edge[-,white,line width=5pt] (d);
			\draw (b)  edge[->] (d);
			
			\draw (b)  edge[-,white,line width=5pt] (nM);
			\draw (cM) edge[-,white,line width=5pt] (d);
			\draw (b)  edge[->] (nM);
			\draw (cM) edge[->] (d);
			
			\draw (aM) edge[-,white,line width=5pt] (cM);
			\draw (nM) edge[-,white,line width=5pt] (fM);
			\draw (aM) edge[->] (cM);
			\draw (nM) edge[->] (fM);
			
			\draw (cM) edge[-,white,line width=5pt] (nN);
			\draw (nN) edge[-,white,line width=5pt] (fN);
			\draw (cM) edge[->] (nN);
			\draw (nN) edge[->] (fN);
		\end{tikzpicture}
		\caption{A diagram of $\sf{ZFC}$-provable relations between cardinal characteristics of the (classical) continuum related to the ideal $\cal E$. An arrow $\fr x\to\fr y$ means that $\fr x\leq\fr y$ is provable.}\label{fig:classical extended Cichon}
	\end{figure}

	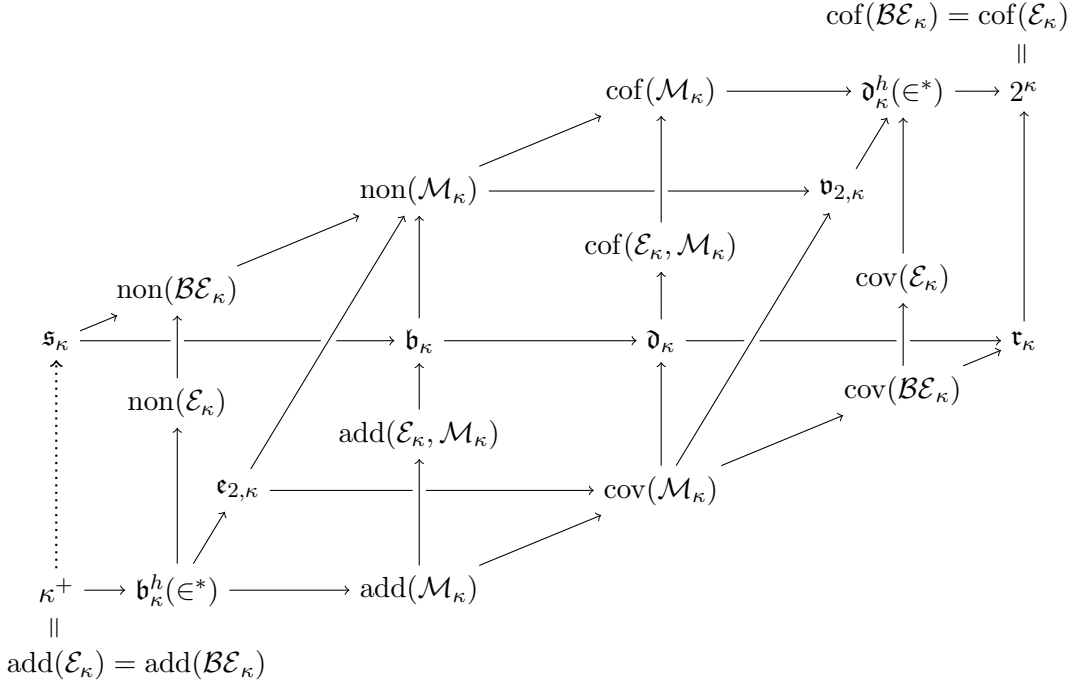
\begin{figure}[t]
		\begin{tikzpicture}[xscale=1.6,yscale=1.65]
			
			\node (a1) at (0,0) {$\kappa^+$};
			\node (c)  at (8,4) {$2^{\kappa}$};
			\node (s)  at (0,2) {$\fr s_\kappa$};
			\node (r)  at (8,2) {$\fr r_\kappa$};
			\node (b)  at (3,2) {$\fr b_\kappa$};
			\node (d)  at (5,2) {$\fr d_\kappa$};
			\node (aN) at (1,0) {$\bins{h}$};
			\node (fN) at (7,4) {$\dins{h}$};
			\node (aM) at (3,0) {$\add(\cal M_\kappa)$};
			\node (fM) at (5,4) {$\cof(\cal M_\kappa)$};
			\node (nM) at (3,3.2) {$\non(\cal M_\kappa)$};
			\node (cM) at (5,.8) {$\cov(\cal M_\kappa)$};
			\node (e)  at (1.5,.8) {$\fr e_{2,\kappa}$};
			\node (v)  at (6.5,3.2) {$\fr v_{2,\kappa}$};

			\node (nE) at (1,1.5) {$\non(\cal E_\kappa)$};
			\node (cE) at (7,2.5) {$\cov(\cal E_\kappa)$};
			\node (nBE) at (1,2.4) {$\non(\cal {BE}_\kappa)$};
			\node (cBE) at (7,1.6) {$\cov(\cal {BE}_\kappa)$};
			
			\node (aEM) at (3,1.25) {$\add(\Ek,\cal M_\kappa)$};
			\node (fEM) at (5,2.75) {$\cof(\Ek,\cal M_\kappa)$};
			
			\node at (8,4.6) {\llap{$\cof(\BEk)={}$}$\cof(\Ek)$};
			\node at (0,-.6) {$\add(\Ek)$\rlap{${}=\add(\BEk)$}};
			\node[rotate=90] at (8,4.3) {$=$};
			\node[rotate=90] at (0,-.3) {$=$};

			\draw (a1) edge[->] (aN);
			\draw (aN) edge[->] (nE);
			\draw (nBE) edge[->] (nM);
			
			\draw (cM) edge[->] (cBE);
			\draw (cE) edge[->] (fN);
			\draw (aN) edge[->] (aM);
			\draw (aM) edge[->] (cM);
			\draw (aEM) edge[->] (b);
			\draw (b)  edge[->] (d);
			\draw (b)  edge[->] (nM);
			\draw (nM) edge[->] (fM);
			\draw (d)  edge[->] (fEM);
			\draw (cM) edge[->] (d);
			\draw (fM) edge[->] (fN);
			\draw (fN) edge[->] (c);
			
			\draw (d)  edge[->] (r);
			\draw (cBE) edge[->] (r);
			\draw (s)  edge[->] (b);
			\draw (s)  edge[->] (nBE);
			\draw (a1)  edge[->,line width=.8pt, dotted] (s);
			\draw (r)  edge[->] (c);

			\draw (e)  edge[->] (cM);
			\draw (aN) edge[->] (e);
			
			\draw (v) edge[->] (fN);
			\draw (nM)  edge[->] (v);

			\draw (nE) edge[-,white,line width=5pt] (nBE);
			\draw (e) edge[-,white,line width=5pt] (nM);
			\draw (aM) edge[-,white,line width=5pt] (aEM);
			\draw (cM) edge[-,white,line width=5pt] (v);
			\draw (fEM) edge[-,white,line width=5pt] (fM);
			\draw (cBE) edge[-,white,line width=5pt] (cE);
			
			\draw (nE) edge[->] (nBE);
			\draw (e) edge[->] (nM);
			\draw (aM) edge[->] (aEM);
			\draw (cM) edge[->] (v);
			\draw (fEM)  edge[->] (fM);
			\draw (cBE) edge[->] (cE);

		\end{tikzpicture}
		\caption{A diagram of $\mathsf{ZFC}$-provable relations between cardinal characteristics of the higher continuum under the assumption that $\kappa$ is strongly inaccessible. The dotted arrow requires the assumption that $\kappa$ is weakly compact.}\label{fig:higher extended Cichon}
	\end{figure}
	
	The following properties are well-known and easy to see. Note that the last two properties imply various relations between the uniformity and covering numbers of the ideals given in \cref{fig:ideals}, reflecting the relations displayed in that diagram.
	\begin{prp}\label{cardinal functions comparison}
		Let $\cal I_\kappa\subset\cal J_\kappa$ be \mbox{${\leq}\kappa$-complete} ideals on $\kacs$.
		\begin{enumerate}
			\item $\kappa^+\leq \add(\cal I_\kappa)\leq\min\ac{\cov(\cal I_\kappa),\non(\cal I_\kappa)}$,
			\item $\max\ac{\cov(\cal I_\kappa),\non(\cal I_\kappa)}\leq\min\ac{2^\kappa,\cof(\cal I_\kappa)}$,
			\item $\cof(\cal I_\kappa)\leq 2^{2^\kappa}$,
			\item $\non(\cal I_\kappa)\leq \non(\cal J_\kappa)$, 
			\item $\cov(\cal J_\kappa)\leq \cov(\cal I_\kappa)$.\qed
		\end{enumerate}
	\end{prp}

	\subsection{Comparison between the higher closed null cardinal characteristics}\label{sec:comparison within closed null}
	
	We will first compare the cardinal characteristics of the families defined in \cref{dfn:families} to each other. The following theorem gives a general condition for comparing the cardinal characteristics of the ideal $\Ek(B)$ with those of $\Ek(C)$, for $B,C\in\BPklim$.
	
	\begin{thm}
		If $B,C\in\BPklim$ and $f\in\kaka$ is injective such that $2^{\card{B_\alpha}}=2^{\card{C_{f(\alpha)}}}$ for almost all $\alpha\in\kappa$, then 
		\begin{align*}
			\non(\Ek(C))&\leq \non(\Ek(B)), & \cov(\Ek(B))\leq \cov(\Ek(C)),\\
			\add(\Ek(C))&\leq \add(\Ek(B)), & \cof(\Ek(B))\leq \cof(\Ek(C)).
		\end{align*}
	\end{thm}
	\begin{proof}
		We first show that $\non(\Ek(B))\geq\non(\Ek(C))$ and remark that $\cov(\Ek(B))\leq \cov(\Ek(C))$ is proved by the same argument.
		Suppose $B,C\in\BPklim$, $\alpha_0\in\kappa$ and  $f\in\kaka$ are such that $f$ is injective and $2^{\card{B_\alpha}}=2^{\card{C_{f(\alpha)}}}$ for all $\alpha\in\iv{\alpha_0}{\kappa}$. For each $\alpha\in\iv{\alpha_0}\kappa$ we fix a bijection $\pi_\alpha\colon{}^{B_\alpha}2\to {}^{C_{f(\alpha)}}2$.
		
		Let us fix some $X\subset\kacs$ with $X\notin\Ek(B)$. For each $x\in X$, let $x'\in\kacs$ be defined by
		\begin{align*}
			x'\colon\xi\mapsto\begin{cases}
				\pi_\alpha(x\restriction B_\alpha)(\xi) & \text{ if there is }\alpha\in\iv{\alpha_0}\kappa\text{ such that }\xi\in C_{f(\alpha)},\\
				0 & \text{ otherwise.}
			\end{cases}
		\end{align*} 
		Let $X'=\set{x'}{x\in X}$, then $\card{X'}\leq\card X$. Let $\phi'\in\Sigma_\kappa(C)$, then we define $\phi$ by
		\begin{align*}
			\phi\colon\alpha\mapsto\begin{cases}
				\pi_\alpha^{-1}\sqb[big]{\phi'(f(\alpha))} & \text{ if }\alpha\in\iv{\alpha_0}\kappa,\\
				\emp & \text{ otherwise}.
			\end{cases}
		\end{align*}
		Then $\phi\in\Sigma_\kappa(B)$, since $\card{\phi(\alpha)}\leq\card{\phi'(f(\alpha))}<2^{\card{C_{f(\alpha)}}}=2^{\card{B_\alpha}}$. Because $X\notin\Ek(B)$, there is $x\in X\setminus \sqb\phi_*$, and thus $\exists^\infty\alpha\in\kappa\,\rb[big]{x\restriction B_\alpha\notin\phi(\alpha)}$. However, note that $x'\restriction C_{f(\alpha)}=\pi_\alpha(x\restriction B_\alpha)$, therefore $\exists^\infty\alpha\,\rb[big]{x'\restriction C_{f(\alpha)}\notin \phi'(\alpha)}$, and thus $X'\nsubset\sqb{\phi'}_*$. Since $\phi'$ was arbitrary, we see that $X'\notin\Ek(C)$.
		
		Now we show $\add(\Ek(C))\leq\add(\Ek(B))$, and remark that $\cof(\Ek(B))\leq \cof(\Ek(C))$ is proved similarly. Given any $\psi\in\Sigma_\kappa(B)$, we define 
		\begin{align*}
			\psi'\colon\xi\mapsto\begin{cases}
				\pi_\alpha\sqb[big]{\psi(\alpha)} & \text{ if }\xi\in\ran(f)\text{ and }\alpha=f^{-1}(\xi),\\
				\emp & \text{ otherwise,}
			\end{cases}
		\end{align*}
		and given $\phi'\in\Sigma_\kappa(C)$, we define $\phi\in\SkB$ as above. Then note that $\psi'(f(\alpha))\subset \phi'(f(\alpha))$ implies $\psi(\alpha)\subset \phi(\alpha)$. Now, if $\psi_\beta\in\Sigma_\kappa(B)$ for each $\beta\in\gamma$, where $\gamma<\add(\Ek(C))$, then $\Cup_{\beta\in\gamma}\sqb[big]{\psi'_\beta}_*\in\Ek(C)$, and thus there is some $\phi'\in\Sigma_\kappa(C)$ such that $\psi'_\beta(f(\alpha))\subset\phi'(f(\alpha))$ for almost all $\alpha\in\kappa$, from which we see that $\psi_\beta(\alpha)\subset\phi(\alpha)$ for almost all $\alpha$, and thus $\Cup_{\beta\in\gamma}\sqb[big]{\psi_\beta}_*\subset\sqb\phi_*\in\Ek(B)$. Therefore $\gamma<\add(\Ek(B))$.
	\end{proof}
	
	\begin{crl}\label{unif cov of E for successor}
		For successor $\kappa$ and $B\in\BPklim$,  $\non(\Ek(B))=\non(\Ek)$ and $\cov(\Ek(B))=\cov(\Ek)$.
	\end{crl}
	\begin{proof}
		For $\kappa$ successor, $B\in\BPklim$ if and only if $\card{B_\alpha}=\lambda$ for almost all $\alpha$. Moreover, by \cref{crl:Ekm is ideal}, $\Ek=\Ek(B)$ for some $B\in\BPklim$.
	\end{proof}

	\begin{thm}
		For each $B\in\BPklim$, $\cof(\Ek(B))\leq 2^\kappa$ and $\cof(\Ek)\leq 2^\kappa$ and $\cof(\BEk)\leq 2^\kappa$.
	\end{thm}
	\begin{proof}
		Remember that we assume $\kappa=\kappa^{<\kappa}$. Now $\Sigma_\kappa(B)\subseteq\prod_{\alpha\in\kappa}\sqb[big]{{}^{B_\alpha}2}^{<2^{\card{B_\alpha}}}$, and since $\card{B_\alpha}<\kappa$, we see that $\card[big]{{}^{B_\alpha}2}\leq 2^{<\kappa}=\kappa$, and thus $\card[Big]{\sqb[big]{{}^{B_\alpha}2}^{<2^{\card{{B_\alpha}}}}}\leq\kappa^{<\kappa}=\kappa$. That is, $\card{\Sigma_\kappa(B)}=2^\kappa$. Moreover, $\card[big]{\BPklim}=2^\kappa$, since  $\BPklim\subseteq\prod_{\alpha\in\kappa}[\kappa]^{<\kappa}$ and $\card[big]{[\kappa]^{<\kappa}}=\kappa$.
	\end{proof}
	
	\begin{thm}
		There exist some $I\in\IPklim$ and $B\in\BPklim$ for which $\add(\Ek(I))\leq \add(\EkA)$ and $\add(\Ek(B))\leq \add(\BEkA)$. Moreover, there exist some $I\in\IPklim$ and $B\in\BPklim$ for which $\non(\BEkA)=\non(\Ek(B))$ and $\non(\EkA)=\non(\Ek(I))$.
	\end{thm}
	\begin{proof}
		By definition: suppose $\abset{X_\alpha}{\alpha\in\add(\EkA)}$ is a sequence in $\EkA$ such that $\Cup_{\alpha}X_\alpha\notin\EkA$, then there is $I\in\IPklim$ such that $\Cup_{\alpha}X_\alpha\notin\Ek(I)$. Similarly for the uniformity numbers.
	\end{proof}
	
	In terms of additivity numbers, \cref{thm:BEkm is no ideal,thm:eveq equi} showed that $\add(\BEkE)=2$ and (under assumptions on $\kappa$) $\add(\EkE)=2$. We will use the same method of proof to determine the values of $\add(\BEk)$, $\add(\Ek)$, $\cof(\BEk)$ and $\cof(\Ek)$. Note that the higher case does not follow our classical intuition, where $\add(\cal E)=\add(\cal M)$ and $\cof(\cal E)=\cof(\cal M)$.
	
	\begin{thm}\label{thm:add BEk and Ek are absolute}
		Assume $\kappa=\kappa^{<\kappa}$.
		\begin{enumerate}[label=(\arabic*)]
			\item\label{thm:add BEk is absolute} $\add(\BEk)=\kappa^+$.
			\item\label{thm:cof BEk is absolute} $\cof(\BEk)=2^\kappa$.
			\item\label{thm:add Ek is absolute} If the continuum function is not eventually constant below $\kappa$, then $\add(\Ek)=\kappa^+$.
			\item\label{thm:cof Ek is absolute} If the continuum function is not eventually constant below $\kappa$, then $\cof(\Ek)=2^\kappa$.
		\end{enumerate}
	\end{thm}
	\begin{proof}
		\ref{thm:add BEk is absolute}.\quad By recursion we may construct a sequence $\abset{B^\xi}{\xi\in\kappa^+}$ such that for any distinct $\xi,\xi'\in\kappa^+$ and for almost all $\alpha,\beta\in\kappa$ we have $\card[big]{B^\xi_\alpha\cap B^{\xi'}_\beta}\leq 1$. For each $\xi$ we define $\phi^\xi\in\Sigma_\kappa(B^\xi)$ such that for any $\alpha\in\kappa$, any $\eta\in B^\xi_\alpha$ and any $i\in\ac{0,1}$, there is some $s\in\phi^\xi(\alpha)$ with $s(\eta)=i$. Then we claim that $\Cup_{\xi\in\kappa^+}\sqb{\phi^\xi}_*\notin\BEk$.
		
		Suppose, towards contriction, that $\abset{C^\eta}{\eta\in\kappa}$ is a sequence in $\BPklim$ and $\abset{\psi^\eta}{\eta\in\kappa}$ is such that $\psi^\eta\in\Sigma_\kappa(C^\eta)$ for each $\eta\in\kappa$ and $\Cup_{\xi\in\kappa^+}\sqb{\phi^\xi}_*\subset \Cup_{\eta\in\kappa}\sqb{\psi^\eta}_*$. For each $\xi\in\kappa^+$, we see that $\sqb{\phi^\xi}_0=\Cup_{\eta\in\kappa}\rb[big]{\sqb{\psi^\eta}_*\cap \sqb{\phi^\xi}_0}$ and, by \cref{lmm:nontrivial meagre for slaloms}, $\Mk(\sqb{\phi^\xi}_0)$ is ${\leq}\kappa$-complete and proper, therefore for each $\xi\in\kappa^+$ there is $\eta_\xi\in\kappa$ for which $\sqb{\psi^{\eta_\xi}}_*\cap \sqb{\phi^\xi}_0\notin\Mk(\sqb{\phi^\xi}_0)$. Then by the pigeonhole principle there are distinct $\xi,\xi'\in\kappa^+$ such that $\eta:=\eta_\xi=\eta_{\xi'}$. But this contradicts \cref{lmm:BE additive failure}, as $B^\xi$, $B^{\xi'}$, $\phi^\xi$ and $\phi^{\xi'}$ satisfy all of its assumptions.

		\noindent\ref{thm:cof BEk is absolute}.\quad The idea is dual to \ref{thm:add BEk is absolute}. We show that there exists a sequence $\abset{B^\xi}{\xi\in2^\kappa}$ in $\BPklim$ such that for any distinct $\xi,\xi'\in2^\kappa$ and almost all $\alpha,\beta\in\kappa$ we have $\card[big]{B^\xi_\alpha\cap B^{\xi'}_\beta}\leq 1$. This suffices: if we define $\phi^\xi$ for each $\xi$ as above, and we fix a sequence $\abset{C^\eta}{\eta\in\lambda}$ in $\BPklim$ with $\lambda<2^\kappa$, as well as corresponding $\psi^\eta\in\Sigma_\kappa(C^\eta)$, then \cref{lmm:BE additive failure} shows that for each $\eta$ there is at most one $\xi\in2^\kappa$ for which $\sqb{\psi^\eta}_*\cap \sqb{\phi^\xi}_0\notin\Mk(\sqb{\phi^\xi}_0)$, and therefore, no $\kappa$-unions of sets of the form $\sqb{\psi^\eta}_*$ will be able to cover more than $\kappa$-many sets of the form $\sqb{\phi^\xi}_*$, that is, $\cof(\BEk)> \lambda$.
		
		To conclude, we argue that a sequence $\abset{B^\xi}{\xi\in 2^\kappa}$ in $\BPklim$ with the above properties exists. We fix an almost disjoint family $\set{A_\xi}{\xi\in 2^\kappa}\subset\sqkaka$, that is, a family such that $\card{A_\xi\cap A_{\xi'}}<\kappa$ for all distinct $\xi,\xi'\in 2^\kappa$. We also fix $C\in\BPklim$. Now, for any $\xi\in 2^\kappa$, let $\abset[big]{a^\xi_\alpha}{\alpha\in\kappa}$ and $\abset[big]{b^\xi_\alpha}{\alpha\in\kappa}$  enumerate $A_\xi$ and  $\Cup_{\alpha\in\kappa\setminus A_\xi}C_\alpha$, respectively, and define $B^\xi_\alpha=C_{a^\xi_\alpha}\cup\ac{b^\xi_\alpha}$. 
		
		It is easy to see that for distinct $\xi,\eta\in 2^\kappa$ and for large enough $\alpha,\delta\in\kappa$ we have $\card[big]{B^\xi_\alpha\cap B^{\eta}_{\delta}}\leq 2$. This is actually good enough to prove a corresponding version of \cref{lmm:BE additive failure}, or alternatively we can bring the ${\leq}\,2$ down to ${\leq}\,1$ with a little more precision in the set-up. It suffices to assume that $C\in\IPklim$ (i.e., consists of consecutive intervals), that the enumerations $\abset[big]{a^\xi_\alpha}{\alpha\in\kappa}$ and $\abset[big]{b^\xi_\alpha}{\alpha\in\kappa}$ are increasing, and each $A_\xi$ is sufficiently sparse to ensure $b^\xi_\alpha<\min\rb[big]{C_{a^\xi_\alpha}}$. Under these additional assumptions, for large enough $\alpha,\delta$ we would have $\card[big]{B^\xi_\alpha\cap B^{\eta}_{\delta}}=2$ if and only if $B^\xi_\alpha\cap B^{\eta}_{\delta}=\ac[big]{b^\xi_\alpha,b^{\eta}_{\delta}}$, but this implies $b^{\xi}_\alpha\in C_{a^{\eta}_{\delta}}$ and $b^{\eta}_{\delta}\in C_{a^\xi_\alpha}$, contradicting that both $b^\xi_\alpha<\min\rb[big]{C_{a^\xi_\alpha}}$ and $b^{\eta}_{\delta}<\min\rb[big]{C_{a^{\eta}_{\delta}}}$.
		
		\noindent\ref{thm:add Ek is absolute} \& 
		\noindent\ref{thm:cof Ek is absolute}.\quad Similar to \ref{thm:add BEk is absolute} \& \ref{thm:cof BEk is absolute}, but we use families $\abset{x^\xi}{\xi\in\kappa^+}$ and $\abset{x^\xi}{\xi\in2^\kappa}$, respectively, of strictly increasing functions in $\kaka$ whose ranges are mutually almost disjoint, and then apply \cref{lmm:E additive failure} instead of \cref{lmm:BE additive failure}.
	\end{proof}

	\subsection{Comparison with other cardinal characteristics}\label{sec:comparison with other cardinals}
	
	As a starting ground, we could consider the cardinal characteristics of the classical closed null ideal, such as displayed in \cref{fig:classical extended Cichon}. See \cite{Blass1994,BartoszynskiJudah} for the classical definitions of the cardinal characteristics and proofs of the theorems involved.	
	With the exception of $\cov(\cal N)$ and $\non(\cal N)$, each of the cardinal characteristics from \cref{fig:classical extended Cichon} has a natural generalisation to the higher context of $\kacs$, and we will introduce these in the relevant sections. We will look at relations with the bounding and dominating numbers (\cref{sec:cc b}), splitting and reaping numbers   (\cref{sec:cc s}), localisation numbers (\cref{sec:cc loc}), and evasion and prediction numbers   (\cref{sec:cc e}), and finally we give a relation between cardinal characteristics of the ideal $\NAk$ and the ideals $\Ek(I)$ (\cref{sec:NAk and EkI}).

	\subsubsection{The bounding and dominating numbers}\label{sec:cc b} We refer to \cite{CummingsShelah,BrendleBrookeTaylorFriedmanMontoya} for more information on the higher bounding and dominating numbers, here we just provide the definition and the known relations to the cardinal characteristics of $\Mk$.
	
	\begin{dfn}\label{dfn:dom bd}
		For $f,g\in\kaka$, let $f\leqs g$ if and only if $\card[big]{\set{\alpha\in\kappa}{g(\alpha)<f(\alpha)}}<\kappa$. We call a set $B\subset\kaka$ \emph{unbounded} if no $g\in\kaka$ exists such that $f\leqs g$ for all $f\in B$, and we call a set $D\subset\kaka$ \emph{dominating} if every $f\in\kaka$ has some $g\in D$ such that $f\leqs g$. The \emph{bounding number} $\fr b_\kappa$ is the least size of an unbounded subset of $\kaka$ and the \emph{dominating number} $\fr d_\kappa$ is the least size of a dominating subset of $\kaka$.
	\end{dfn}
	
	\begin{fct}[\citet{CummingsShelah,BrendleBrookeTaylorFriedmanMontoya}]
		Assume $\kappa=\kappa^{<\kappa}$.
		\begin{enumerate}
			\item $\fr b_\kappa\leq\fr d_\kappa$,
			\item $\fr b_\kappa\leq\non(\Mk)$ and $\cov(\Mk)\leq \fr d_\kappa$,
			\item $\add(\Mk)=\min\ac{\fr b_\kappa,\cov(\Mk)}$ and $\cof(\Mk)=\max\ac{\fr d_\kappa,\non(\Mk)}$.\qed
		\end{enumerate}
	\end{fct}
	
	We will need the following generalisations of the additivity and cofinality numbers, which we will use to generalise a result by \citet{Miller} (see also \cite[Theorems~2.6.10 and~2.2.2]{BartoszynskiJudah} for classical analogues).
	
	\begin{dfn}
		Suppose that $\cal J $ is a subset of $\cal I $, then $\add(\cal J ,\cal I )$ denotes the least size of a set $J\subset\cal J $ such that $\Cup J\notin \cal I $, and $\cof(\cal J ,\cal I )$ denotes the least size of a family $\cal F\subset\cal I $ such that every $J\in \cal J $ is contained in some $I\in\cal F$.
	\end{dfn}
	
	\begin{fct}
		If $\cal J \subset\cal K \subset\cal I $, then $\add(\cal K )\leq\add(\cal J ,\cal I )$ and $\cof(\cal K )\geq\cof(\cal J ,\cal I )$.\qed
	\end{fct}

	\begin{thm}\label{thm:add E M and bk}
		Let $\kappa$ be strongly inaccessible, then $\add(\Ek,\Mk)\leq\fr b_\kappa$ and $\cof(\Ek,\Mk)\geq \fr d_\kappa$.
	\end{thm}
	\begin{proof}
		We prove that $\add(\Ek,\Mk)\leq\fr b_\kappa$, and remark that $\cof(\Ek,\Mk)\geq\fr d_\kappa$ is proved similarly. 
		
		Let $B\subset\kaka$ be an unbounded set. Fix a sequence of cardinals $\abset{\kappa_\alpha}{\alpha\in\kappa}$ such that $\abset{2^{\kappa_\alpha}}{\alpha\in\kappa}$ is strictly increasing. Given $b\in B$ we define $b'\in\kaka$ by $b'(0)=0$, and for limit $\gamma\in\kappa$ let $b'(\gamma)=\sup_{\alpha\in\gamma}b'(\alpha)$, and finally 
		\begin{align}
			b'(\alpha+1)=\sup\rb[big]{\ac{b'(\alpha)+1}\cup\set{b(\xi)}{\xi\leq b'(\alpha)}}.\tag{$\dagger$}\label{eqn:b'}
		\end{align}
		We will write $\kappa^b_\alpha:=\kappa_{b'(\alpha)}$ for all $\alpha\in\kappa$. Note that $b'$ (and thus also $\abset{\kappa^b_\alpha}{\alpha\in\kappa}$) is strictly increasing and continuous.  Now we will proceed in a similar manner to the proof of \cref{EkE adds to everything}. Define the sets
		\begin{align*}
			X^b&:=\set[m]{x\in\kacs}{\forall\alpha\in\kappa\fforall\xi\in\iv[big]{\kappa^b_{2\alpha}}{\kappa^b_{2\alpha+1}}\,\rb{x(\xi)=0}},\\
			Y^b&:=\set[m]{x\in\kacs}{\forall\alpha\in\kappa\fforall\xi\in\iv[big]{\kappa^b_{2\alpha+1}}{\kappa^b_{2\alpha+2}}\,\rb{x(\xi)=0}},\\
			Z&:=\set[big]{X^b}{b\in B}\cup\set[big]{Y^b}{b\in B}.
		\end{align*}
		We first show that $Z\subset\EkE$. Namely, let us define
		\begin{align*}
			J_\alpha&=\iv{\kappa^b_{2\alpha}}{\kappa^b_{2\alpha+2}},&
			I_\alpha&=\begin{cases}
				\iv{\kappa^b_{2\alpha}}{\kappa^b_{2\alpha+3}} & \text{ if }\alpha\text{ is limit,}\\
				\iv{\kappa^b_{2\alpha+1}}{\kappa^b_{2\alpha+3}} & \text{ otherwise,}
			\end{cases}
		\end{align*}
		then $I,J\in\IPklim$ and $X^b\in\Ek(I)$ and $Y^b\in\Ek(J)$. 
		
		To finish the proof, it suffices to show that $\Cup Z\notin\cal M_\kappa$. Suppose otherwise, and let $\abset{N_\eta}{\eta\in\kappa}$ be a sequence of nowhere dense subsets of $\kacs$ such that $\Cup Z=\Cup_{\eta\in\kappa}N_\eta$. Using that $\kappa$ is strongly inaccessible, we inductively build a continuous sequence of ordinals $\abset{\xi_\eta}{\eta\in\kappa}$ and a sequence $\abset{s_\eta}{\eta\in\kappa}$ of elements of $\fkacs$ such that for each $\eta\in\kappa$,
		\begin{enumerate}
			\item  if $t\colon\kappa_{\xi_\eta}\to2$, then $\sqb{t^\frown s_\eta}$ is disjoint from $\Cup_{\zeta<\eta}N_\zeta$ (note that this is possible, since $\Cup_{\zeta<\eta}N_\zeta\in\NWD$, by ${<}\kappa$-completeness of $\NWD$), and
			\item $\kappa_{\xi_{\eta+1}}$ is larger than the ordinal sum $\kappa_{\xi_\eta}+\dom(s_\eta)$.
		\end{enumerate}
		Define $g\in\kaka$ by 
		\[g\colon\beta\mapsto \min\set{\gamma\in\kappa}{\exists\eta\in\kappa\,(\beta\leq\xi_\eta\text{ and }\xi_{\eta+2}\leq\gamma)}.\]
		Since $B$ is unbounded, there is some $b\in B$ such that $b\nleqs g$, and thus the set $\set{\beta\in\kappa}{g(\beta)<b(\beta)}$ has size $\kappa$. Fix an element $\beta$ such that $g(\beta)<b(\beta)$, and let $\alpha\in\kappa$ be minimal such that $\beta\leq b'(\alpha)$, with $b'$ defined from $b$ as in \eqref{eqn:b'}, then by definition of $b'$ we have $b(\beta)\leq b'(\alpha+1)$. Now by definition of $g$, there exists $\eta\in\kappa$ such that 
		\[\beta\leq\xi_\eta<\xi_{\eta+1}<\xi_{\eta+2}\leq g(\beta)<b(\beta)\leq b'(\alpha+1).\]
		We consider two cases. Either $b'(\alpha)\leq\xi_{\eta+1}<\xi_{\eta+2}< b'(\alpha+1)$, or else $\xi_{\eta+1}<b'(\alpha)$ and thus by minimality of $\alpha$ we see that $\alpha$ is successor, say $\alpha=\alpha'+1$ with $b'(\alpha')<\beta$, and then we see that
		\[b'(\alpha')<\beta \leq \xi_\eta<\xi_{\eta+1}< b'(\alpha)=b'(\alpha'+1).\] 
		In conclusion, for each $\beta$ with $g(\beta)<b(\beta)$, there are $\alpha,\eta\in\kappa$ with $b'(\alpha)\leq \xi_\eta<\xi_{\eta+1}< b'(\alpha+1)$. Let us define the sets
		\begin{align*}
			A_0&:=\set[m]{(\alpha,\eta)\in\kappa\times\kappa}{b'(2\alpha)\leq \xi_\eta<\xi_{\eta+1}<b'(2\alpha+1)},\\
			A_1&:=\set[m]{(\alpha,\eta)\in\kappa\times\kappa}{b'(2\alpha+1)\leq \xi_\eta<\xi_{\eta+1}<b'(2\alpha+2)},
		\end{align*}
		then $\card{A_0\cup A_1}=\kappa$, thus let us assume without loss that $\card{A_0}=\kappa$. We define $x\in\kacs$ by
		\begin{align*}
			x(\zeta)=\begin{cases}
				s_\eta(\zeta') & \text{if there are }(\alpha,\eta)\in A_0\text{ and }\zeta'\in\dom(s_\eta)\text{ with }\zeta=\kappa_{\xi_\eta}+\zeta',\\
				0 & \text{otherwise},
			\end{cases}
		\end{align*}
		then $x$ is only nonzero in intervals $\iv{\kappa_{\xi_\eta}}{\kappa_{\xi_{\eta+1}}}$ for which there exists an $\alpha\in\kappa$ with $(\alpha,\eta)\in A_0$. Therefore $x$ is only nonzero on intervals $\iv[big]{\kappa^b_{2\alpha}}{\kappa^b_{2\alpha+1}}$, that is, $x\in Y^b\subset\Cup Z$. But simultaneously, for each $(\alpha,\eta)\in A_0$ we have $\sqb{x\restriction (\kappa_{\xi_\eta}+\dom(s_\eta))}=\sqb{(x\restriction \kappa_{\xi_\eta}){}^\frown s_\eta}$, which is disjoint from $\Cup_{\zeta<\eta} N_\zeta$ by our definitions of $\xi_\eta$ and $s_\eta$. But then $x\notin\Cup_{\eta\in\kappa}N_\eta=\Cup Z$, which is a contradiction.
	\end{proof}
	
	Although we already know the values of $\add(\Ek)$, $\add(\BEk)$, $\cof(\Ek)$ and $\cof(\BEk)$ from \cref{thm:add BEk and Ek are absolute}, \cref{thm:add E M and bk} gives us new information, since we note that $\add(\cal M_\kappa)\leq \add(\cal I,\cal M_\kappa)$ for any $\cal I\subset\Mk$, and thus consistently $\add(\Ek,\Mk)\neq \add(\Ek)$, and so on.
	
	Classically, we have $\add(\cal M)=\add(\cal E)$, and thus $\add(\cal M)\leq \non(\cal E)$. We do not know if the analogous result can be proved in the higher context.

	\subsubsection{The splitting and reaping numbers}\label{sec:cc s} The splitting number has been well-studied in the higher context, and behaves  differently than our classical intuition would suggest.
	
	\begin{dfn}
		For $X,Y\in\sqb\kappa^\kappa$, say that $X$ splits $Y$ if $\card{Y\cap X}=\card{Y\setminus X}=\kappa$. We call a family $S\subset\sqb\kappa^\kappa$ a \emph{splitting family} if every $Y\in\sqb\kappa^\kappa$ is split by some $X\in S$, and we call a family $R\subset\sqb\kappa^\kappa$ an \emph{unreapable family} if no $X\in\sqb\kappa^\kappa$ splits all $Y\in R$. The \emph{splitting number} $\fr s_\kappa$ is the least size of a splitting family, and the \emph{reaping number} $\fr r_\kappa$ is the least size of an unreapable family.
	\end{dfn}
	
	\begin{fct}[\citet{RaghavanShelah2017,Suzuki1998,Zapletal}]\leavevmode
		\begin{enumerate}
			\item $\fr s_\kappa\leq\fr b_\kappa\leq\fr r_\kappa$ and if $\kappa\geq\beth_\omega$, then $\fr r_\kappa\geq\fr d_\kappa$,
			\item  $\fr s_\kappa\leq\non(\Mk)$ and $\fr r_\kappa\geq\cov(\Mk)$,
			\item $\kappa$ is strongly inaccessible if and only if $\kappa\leq\fr s_\kappa$
			\item $\kappa$ weakly compact if and only if $\kappa^+\leq\fr s_\kappa$,
			\item if $\kappa^+<\fr s_\kappa$, an inner model with a measurable $\lambda$ of Mitchell order $o(\lambda)\geq\lambda^{++}$ exists.\qed
		\end{enumerate}
	\end{fct}
	
	It follows from this that, as a lower bound, the only cases in which it is interesting to consider $\fr s_\kappa$ are when $\kappa$ is at least weakly compact.
	
	\begin{thm}We have
		$\fr s_\kappa\leq\non(\BEk)$ and $\fr r_\kappa\geq\cov(\BEk)$.
	\end{thm}
	\begin{proof}
		We prove that $\fr s_\kappa\leq\non(\BEk)$ and remark that $\fr r_\kappa\geq\cov(\BEk)$ is shown similarly.
		
		Suppose that $X\subset\kacs$ and $X\notin\BEk$ and let $X'\subset\sqb\kappa^{\leq\kappa}$ be defined as $X'=\set{x^{-1}(1)}{x\in X}$; that is, $X$ is the set of characteristic functions for the sets in $X'$. We claim that $X'$ is a splitting family. Given $Y\subset\sqb\kappa^\kappa$ for which $\card{\kappa\setminus Y}=\kappa$, we choose some $B\in\BPklim$ such that $\card{B_\alpha\cap Y}=2^{\card \alpha}$ and $\card{B_\alpha\setminus Y}=\card \alpha$ for each $\alpha\in\kappa$, and we define $\phi\in\SkB$ such that 
		\[\phi(\alpha)=\set[big]{s\in {}^{B_\alpha}2}{\forall \xi,\xi'\in B_\alpha\cap Y\,(s(\xi)=s(\xi'))}.\]
		Since $X\notin\BEk$, we may choose $x\in\sqb X\setminus\sqb\phi_*$, then for $\kappa$-many $\alpha$ there are some $\xi_0,\xi_1\in B_\alpha\cap Y$ such that $x(\xi_0)=0$ and $x(\xi_1)=1$. It follows that $Y$ is split by $x^{-1}(1)$.
	\end{proof}
	
	We note that the above proof is highly specific to $\BEk$, and does not carry over to $\Ek$.

	\subsubsection{Localisation numbers}\label{sec:cc loc} We first introduce some terminology in order to define the localisation numbers.
	Consider a cardinal-valued function $b$ with domain $\kappa$ such that $b(\alpha)\geq 2$ for all $\alpha\in\kappa$, and let $\prod b$ denote the product $\prod_{\alpha\in\kappa}b(\alpha)$. If $h$ is another cardinal-valued function with domain $\kappa$, then the elements of $\prod_{\alpha\in\kappa}\sqb{b(\alpha)}^{<h(\alpha)}$ are called \emph{$(b,h)$-slaloms}. In the special case where $b=\bar\kappa$ is the constant function sending $\alpha\mapsto\kappa$ for all $\alpha\in\kappa$, we call a $(\bar \kappa,h)$-slalom simply an \emph{$h$-slalom}.
	
	For a $(b,h)$-slalom $\phi$ and $f\in\prod b$, we say that $f$ is \emph{localised} by $\phi$ if $\forall^\infty\alpha\in\kappa\,(f(\alpha)\in\phi(\alpha))$, denoted by $f\ins\phi$. The \emph{$(b,h)$-localisation number} $\dins{b,h}$ is the least cardinality of a family $\cal F$ of $(b,h)$-slaloms such that every $f\in\prod b$ is localised by a $(b,h)$-slalom in $\cal F$. Dually, the \emph{$(b,h)$-avoidance number} $\bins{b,h}$ is the least cardinality of a family $\cal A\subset\prod b$ such that no single $(b,h)$-slalom localises all $f\in\cal A$. We write $\fr d^h_\kappa(\ins)$ and $\fr b^h_\kappa(\ins)$ if $b=\bar \kappa$.
	
	Classically, localisation numbers were introduced by \citet{BartoszynskiLocalisation} to provide a combinatorial characterisation of the cofinality and additivity numbers of the Lebesgue null ideal:
	
	\begin{fct}[\citet{BartoszynskiLocalisation}]\label{Bartoszynski characterisation}
		Let $h\in\omom$ be strictly increasing, then $\add(\cal N)=\fr b_\omega^h(\ins)$ and $\cof(\cal N)=\fr d_\omega^h(\ins)$.\qed
	\end{fct}
	
	The choice of parameters $(b,h)$ is relevant classically, as \citet{GoldsternShelah,Kellner} showed that there can consistently be continuum many localisation numbers with distinct values. In the higher context, even if we fix $b=\bar \kappa$, the choice of $h$ remains relevant, as shown in  \cite[Theorem~70]{BrendleBrookeTaylorFriedmanMontoya}.
	In the higher context, localisation numbers were furthermore studied in \cite{vanderVlugt2025,SwitzerBenhamou,RoslanowskiLoc}.

	The covering number of the ideal $\Ek(B)$ can be expressed in terms of localisation numbers.
	
	\begin{thm} \label{cov Ek as a localisation number}
		Let $B\in\BPk$ and let $b\in\kaka$ be defined by $b\colon\alpha\mapsto 2^{\card{B_\alpha}}$, then
		$\cov(\Ek(B))=\dins{b,b}$ and $\non(\Ek(B))=\bins{b,b}$.
	\end{thm}
	\begin{proof}
		This is not much more than a notational shuffle. Let $B$ and $b$ be as in the statement of the theorem, then we may consider the bijection $\pi\colon \kacs\to\prod b$ given by $x\mapsto \abset{\pi_\alpha(x\restriction B_\alpha)}{\alpha\in\kappa}$, where each $\pi_\alpha$ is a bijection between ${}^{B_\alpha}2$ and $2^{\card{B_\alpha}}$. Similarly, if $\phi$ is a $(b,b)$-slalom, we define the slalom $\tilde \pi(\phi)\colon\alpha\mapsto \pi_\alpha^{-1}[\phi(\alpha)]$. The conclusion follows by noting that for any $x\in\kacs$ and any $(b,b)$-slalom $\phi$ we have $\pi(x)\ins\phi$ if and only if $x\in\sqb{\tilde \pi(\phi)}_*$.
	\end{proof}
	
	\begin{crl}\label{successor equal to b}
		If $\kappa=2^{<\kappa}$ is successor, then $\cov(\Ek)=\fr d_\kappa$ and $\non(\Ek)=\fr b_\kappa$.
	\end{crl}
	\begin{proof}
		By \cref{unif cov of E for successor,cov Ek as a localisation number} we have $\cov(\Ek)=\cov(\Ek(B))=\dins{b,b}$ for some $B\in\BPklim$ and suitably defined $b$. Since $2^{\card{B_\alpha}}=\kappa$ for large enough $\alpha$ and $\kappa$ is regular, it is easy to see that $\dins{b,b}=\fr d_\kappa$.
	\end{proof}
	
	In light of \cref{Bartoszynski characterisation}, the following could be seen as a generalisation of the classical facts that $\add(\cal N)\leq \non(\cal E)$ and $\cov(\cal E)\leq \cof(\cal N)$.\footnote{Classically, these are corollaries of the fact that $\add(\cal N)\leq \add(\cal M)$ and $\cof(\cal M)\leq\cof(\cal N)$, but we remark that this is not immediately apparent in the higher context: although $\bins{h}\leq\add(\cal M_\kappa)$ and $\cof(\cal M_\kappa)\leq\dins{h}$ are provable for inaccessible $\kappa$, the situation under the weaker assumption that $\kappa=\kappa^{<\kappa}$ is unclear (see \cite[Question~24]{BrendleBrookeTaylorFriedmanMontoya}). Moreover, the relation between $\add(\Mk)$ and  $\non(\Ek)$ is unclear.}
	
	\begin{thm}
		Suppose $\kappa=2^{<\kappa}$ is regular. Let $h$ be cofinally increasing, then $\cov(\Ek)\leq\dins{h}$ and $\bins{h}\leq \non(\Ek)$.
	\end{thm}
	\begin{proof}
		Let $h$ be cofinally increasing, let $I\in\IPklim$ be such that ${\card{I_\alpha}}=h(\alpha)$, and let $b\colon\alpha\mapsto 2^{h(\alpha)}$. Then we have $\cov(\Ek)\leq \cov(\Ek(I))=\dins{b,b}\leq\dins{\bar\kappa,b}\leq\dins{h}$. 
	\end{proof}

	\subsubsection{The evasion and prediction numbers}\label{sec:cc e}
	
	Since we work on the higher Cantor space $\kacs$, we will only consider binary evasion and prediction. Suppose we have a function $p\colon \fkacs\to2$ and $D\in\sqb\kappa^\kappa$, then $f\in\kacs$ is \emph{predicted} by $p$ on $D$ if for almost all $\alpha\in D$ we have $p(f\restriction \alpha)=f(\alpha)$, and otherwise $f$ is \emph{evaded} by $p$ on $D$. We call the pair $(p,D)$ a \emph{predictor}. The (binary) \emph{prediction number} $\fr v_{2,\kappa}$ is the least number of predictors necessary to predict every $f\in\kacs$, and the (binary) \emph{evasion number}~$\fr e_{2,\kappa}$ is the least number of functions in $\kacs$ such that every predictor is evaded by one of them.
	
	It is fairly simple to see that the classical binary evasion number $\fr e_2$ is at most of size $\non(\cal E)$, as was first mentioned in \cite[\S\,3.1]{BrendleEvasionI}. Namely, if we fix a predictor $(p,D)$ on $\omcs$, then the set $\set{x\in\omcs}{\forall n\in D\setminus n_0\,(p(x\restriction n)=x(n))}$ is easily seen to be closed and measure zero.
	On $\kappa$ the situation is different, and our classical intuition does not apply.
	
	\begin{thm}
		There exists a predictor $(p,D)\in {}^{(\fkacs)}2\times \sqb\kappa^\kappa$ such that the set of $f\in\kacs$ predicted by $(p,D)$ is not in $\cal{BE}_\kappa$. 
	\end{thm}
	\begin{proof}
		Let $E=\set{2\cdot\alpha}{\alpha\in\kappa}$ be the set of even  ordinals below $\kappa$ and define $p\colon \fkacs\to2$ such that for any $\alpha\in E$ and $s\in{}^{\alpha+1}2$ we have $p(s)=s(\alpha)$. Then $(p,\kappa\setminus E)$ is such a predictor as claimed in the theorem statement. The reasoning in the proof of \cref{thm:BEk is not Ek}, which shows that $X\notin\Ek$, applies in our situation to show that the set of $f\in\kacs$ predicted by $(p,\kappa\setminus E)$ is not in $\BEk$.
	\end{proof}
	
	This suggests that the evasion number $\fr e_{2,\kappa}$ may be unrelated to $\non(\Ek)$.
	
	\subsubsection{Uniformity and cofinality of the null additive ideal}\label{sec:NAk and EkI}

	In this final subsection we will provide a relation between the uniformity number of $\NAk$ and the additivity number of the ideal $\Ek(I)$ for $I\in\IPklim$. A dual result shows a connection between the covering number of $\NAk$ and the cofinality number of $\Ek(I)$. We can use this result to prove that $\non(\NAk)\leq \min\ac{\add(\EkA),\add(\BEkA)}$ for strongly inaccessible $\kappa$.

	\begin{lmm}\label{lem:NAk general}
		The following are equivalent for $X\subset\kacs$:
		\begin{enumerate}[label=(\arabic*)]
			\item\label{NAk-gen1} $X\in\NAk$.
		 	\item\label{NAk-gen2} For all $f\in\kaka$ with $\lim_{\alpha\in\kappa}f(\alpha)=\kappa$ and all $B\in\BPklim$, there exists $\psi\in\prod_{\alpha\in\kappa}\sqb[big]{{}^{B_\alpha}2}^{\leq \card{f(\alpha)}}$ such that $X\subset \sqb\psi_*$.
		 \end{enumerate}
	\end{lmm}
	\begin{proof}
		\ref{NAk-gen2} $\imp$ \ref{NAk-gen1}.\quad This is clear by \cref{dfn:NAk}, using $f\colon \alpha\mapsto \card{\alpha}$. 
		
		\noindent\ref{NAk-gen1} $\imp$ \ref{NAk-gen2}.\quad Given $f\in\kaka$ with $\lim_{\alpha\in\kappa}f(\alpha)=\kappa$ and $B\in\BPklim$, let $\alpha_0=0$ and, for $\xi>0$, let $\alpha_\xi$ be minimal such that 
		\begin{enumerate}[label=(\alph*)]
			\item $f(\alpha)\geq\card\xi$ for all $\alpha\geq \alpha_\xi$,
			\item $\alpha_\xi>\alpha_{\eta}$ for all $\eta\in\xi$, and
			\item $\Cup_{\alpha\in\iv{\alpha_\xi}{\alpha_{\xi+1}}}B_\alpha$ is an interval (which is possible by \cref{capturing BP with IP}).
		\end{enumerate}
		Since $\abset{\alpha_\xi}{\xi\in\kappa}$ is continuous and strictly increasing, if we define $I_\xi:=\Cup_{\alpha\in\iv{\alpha_\xi}{\alpha_{\xi+1}}}B_\alpha$, then $I\in\IPklim$. Now let $\phi\in\prod_{\xi\in\kappa}\sqb[big]{{}^{I_\xi}2}^{\card \xi}$ be such that $X\subset\sqb\phi_*$ and define $\psi(\alpha)=\set{s\restriction B_\alpha}{s\in\phi(\xi)}$ for each $\alpha\in\iv{\alpha_\xi}{\alpha_{\xi+1}}$, then it is clear that $\psi\in\prod_{\alpha\in\kappa}\sqb[big]{{}^{B_\alpha}2}^{\leq f(\alpha)}$ and $X\subset\sqb\psi_*$.
	\end{proof}	
	
	\begin{thm}\label{thm:NAk addEkI}
		Let $\kappa$ be strongly inaccessible and $B\in\BPklim$, then $\non(\NAk)\leq \add(\Ek(B))$ and $\cof(\Ek(B))\leq \cov(\NAk)$.		
	\end{thm}
	
	\begin{proof}
		We will show that $\non(\NAk)\leq\add(\Ek(B))$ and remark that $\cof(\Ek(B))\leq \cov(\NAk)$ is proved similarly.
		
		Suppose $A\subset\Ek(B)$ is a family of size $\card{A}<\non(\NAk)$, then we show that $\Cup A\in\Ek(B)$. We first pick for each $X\in A$ some $\phi_X\in\SkB$ such that $X\subset\sqb{\phi_X}_*$. Given a set $C\in\sqb\kappa^{<\kappa}$, increasingly enumerated as $\abset{c_\xi}{\xi\in\gamma}$, and a function $t\colon C\to 2$, let us write $\tilde t\colon\gamma\to 2$ defined by $\tilde t\colon \xi\mapsto t(c_\xi)$. 
		
		First, we encode each $\phi_X$ into a single element $x_X\in \kacs$, then $\set{x_X}{X\in A}\in\NAk$ by the assumption on $\card A$. 
		For a fixed $\phi_X$, pick an enumeration $\abset[big]{t_{X,\alpha,\xi}}{\xi\in\card{\phi_X(\alpha)}}$ of $\phi_X(\alpha)$ for each $\alpha\in\kappa$, and define
		\[s_{X,\alpha}:= \tilde t_{X,\alpha,0}{}^\frown \tilde t_{X,\alpha,1}{}^\frown \cdots{}^\frown \tilde t_{X,\alpha,\xi}{}^\frown\cdots{}^\frown 1_\alpha, \]
		where $1_\alpha$ denotes the constant function with domain (the cardinal) $2^{\card{B_\alpha}}$ and value $1$. We note that also $s_{X,\alpha}$ has domain $2^{\card{B_\alpha}}$. We define $x_X\in\kacs$ by 
		\[x_X:= s_{X,0}{}^\frown s_{X,1}\cdots{}^\frown s_{X,\alpha}{}^\frown\cdots.\]
		
		Now define $J\in\IPklim$ to be the (unique) interval partition such that each $J_\alpha$ is an interval whose order-type is the cardinal $2^{\card {B_\alpha}}$. We also (uniquely) partition each $J_\alpha$ into $\abset{J_{\alpha,\beta}}{\beta\in 2^{\card{B_\alpha}}}$, where each $J_{\alpha,\beta}$ is an interval of order-type $\card{B_\alpha}$ and $\min(J_{\alpha,\beta})<\min(J_{\alpha,\beta'})$ for all $\beta<\beta'$. By \cref{lem:NAk general} and $\card A<\non(\NAk)$, we find $\psi\in\prod_{\alpha\in\kappa}\sqb[m]{{}^{J_\alpha}2}^{\leq \card{B_\alpha}}$ such that $\set{x_X}{X\in A}\subset\sqb\psi_*$. 
		
		Finally, we convert $\psi$ into a slalom $\chi\in\SkB$ such that $\sqb{\phi_X}_*\subset\sqb\chi_*$ for each $X\in A$, by letting \[\chi(\alpha):=\Cup\set[big]{t\in{}^{B_\alpha}2}{\exists s\in \psi(\alpha)\,\exists \beta<2^{\card{B_\alpha}}\,\rb[big]{s\text{ is eventually constant and }\tilde t=(s\restriction J_{\alpha,\beta}){\tilde{\phantom{a}}}}}.\]
		It is clear that $\sqb{\phi_X}_*\subset\sqb\chi_*$. To see that $\chi\in\SkB$, note that $\card{\psi(\alpha)}\leq \card{B_\alpha}<\cf(2^{\card{B_\alpha}})$, and that, if $s\in\psi(\alpha)$ is eventually constant, then $\card[m]{\set[big]{(s\restriction J_{\alpha,\beta})\tilde{\phantom a}}{\beta\in 2^{\card{B_\alpha}}}}<2^{\card{B_\alpha}}$. 
	\end{proof}
	
	\begin{crl} We have 
		$\non(\NAk)\leq \add(\EkA)$ and 
		$\non(\NAk)\leq \add(\BEkA)$.\qed
	\end{crl}

	\section{Consistency}\label{sec:consistency}
	
	In this section we will provide a few consistency results of strict inequalities between the cardinal characteristics of $\Ek$ and $\BEk$ and some of the other cardinal characteristics from \cref{fig:higher extended Cichon}. We will assume that $\kappa$ is strongly inaccessible throughout this section. 
	
	There are many unknowns to consistency in the higher Baire space, as the theory of iterated forcing suffers from several complications, compared to the classical context. Two of the main problems are the absence of a suitable form of random forcing, and the difficulty of preserving properties like $\kaka$-bounding under iteration. These two problems form two of the reasons that it is hard to find models in which $\cov(\Mk)<\non(\Mk)$.
	
	The reverse strict inequality $\non(\Mk)<\cov(\Mk)$ is, on the other hand, easily obtained with $\kappa$-Cohen forcing, in exactly the same manner as is done classically. Moreover, $\kappa$-Hechler (dominating) forcing can be iterated with ${<}\kappa$-support to provide the consistency of $\bins h<\add(\Mk)$ and of $\cof(\Mk)<\dins h$. Finally, $\kappa$-localisation forcing proves the consistency of $\kappa^+<\bins h$ and of $\dins h<2^\kappa$. We refer to \cite{BrendleBrookeTaylorFriedmanMontoya} for the details about each of the above constructions. As a corollary:
	
	\begin{thm}
		Each of the inequalities $\kappa^+<\non(\Ek)$, $\non(\BEk)<\cov(\Mk)$, $\non(\Mk)<\cov(\BEk)$, and $\cov(\Ek)<2^\kappa$ is consistent for $\kappa$ strongly inaccessible.\qed
	\end{thm}
	
	We recall the following properties of forcing notions. 
	\begin{enumerate}
		\item A forcing notion $\bb P$ is \emph{${<}\kappa$-distributive}, if it does not add  functions $f\colon \lambda\to\Ord$ for $\lambda<\kappa$. 
		\item A forcing notion $\bb P$ is \emph{${<}\kappa$-closed}, if any ${<}\kappa$-length descending sequence of conditions has a lower bound.
		\item A forcing notion $\bb P$ is \emph{strategically ${<}\kappa$-closed}, if you have a winning strategy for every $\lambda<\kappa$ in the following game. You and your opponent play a game of length $\lambda$, where each round $\alpha\in\lambda$ consists of a condition $p_\alpha^\rm Y$ played by you, followed by a condition $p_\alpha^\rm O$ played by your opponent. You must play $p_0^\rm Y=\ft$; in round $\alpha$, you must play $p_\alpha^\rm{Y}\leq p_\xi^\rm{O}$ for all $\xi<\alpha$, and your opponent must play $p_\alpha^\rm O\leq p_\alpha^\rm Y$. You win if you do not run out of moves.
		\item A forcing notion $\bb P$ is \emph{${<}\kappa^+$-c.c.}\ if every antichain in $\bb P$ has cardinality at most $\kappa$.
		\item A  subset $Q\subset\bb P$ is \emph{${<}\lambda$-centred} if every $P\subset Q$ with $\card{P}<\lambda$ has some $q\in\bb P$ with $q\leq p$ for all $p\in P$. A forcing notion $\bb P$ is \emph{$(\kappa,{<}\lambda)$-centred} if there are ${<}\lambda$-centred sets $P_\alpha\subset\bb P$ for each $\alpha\in\kappa$ such that $\Cup_{\alpha\in\kappa}P_\alpha$ is dense in $\bb P$. 
		\item A  subset $P\subset\bb P$ has \emph{$\lambda$-calibre} if for all $P'\in\sqb P^\lambda$ there exists $q\in\bb P$ such that $\set{p\in P'}{q\leq p}$ has size $\lambda$. A forcing notion $\bb P$ is \emph{$(\kappa,\lambda)$-calibre} if there are $\lambda$-calibre sets $P_\alpha\subset\bb P$ for each $\alpha\in\kappa$ such that $\Cup_{\alpha\in\kappa}P_\alpha$ is dense in $\bb P$.
	\end{enumerate} 
	
	Note that ${<}\kappa$-closed forcing notions are strategically so, and strategically ${<}\kappa$-closed forcing notions are ${<}\kappa$-distributive. Each of these three properties is preserved under ${<}\kappa$- and ${\leq}\kappa$-support iteration. 
	If $\lambda\geq 3$, then $(\kappa,{<}\lambda)$-centred forcing notions are ${<}\kappa^+$-c.c. Although ${<}\kappa^+$-c.c.\ is not generally preserved under iterations, the combination of $(\kappa,{<}\lambda)$-centredness with ${<}\kappa$-closure and a canonicity assumption is preserved under ${<}\kappa$-support iteration of length ${<}(2^\kappa)^+$, see \cite[Lemma 55]{BrendleBrookeTaylorFriedmanMontoya}. Finally, ${<}\kappa$-distributive ${<}\kappa^+$-c.c.\ forcing notions preserve all cofinalities, and hence all cardinal numbers.

	\subsection{Generic slaloms with localisation forcing}
	
	In \cite[Definition 50]{BrendleBrookeTaylorFriedmanMontoya}, $\kappa$-localisation forcing is defined. This forcing notion adds a generic slalom $\phi\in\prod_{\alpha\in\kappa}\sqb \kappa^{\leq\card\alpha}$ that localises all ground model functions $f\in\kaka\cap \bf V$, in the sense that $f(\alpha)\in \phi(\alpha)$ for almost all $\alpha\in\kappa$. We present two variations on this forcing notion.
	
	\subsubsection{Adding a new slalom through an old partition}
	
	The first forcing notion $\LockB$ adds a new slalom $\phi\in\prod_{\alpha\in\kappa}\sqb[big]{{}^{B_\alpha}2}^{\leq\card\alpha}$ for a chosen ground model partition $B\in\BPklim$.

	\begin{dfn}\label{dfn:lockB}
		Fix $B\in\BPklim$. Every condition of $\LockB$ is a pair $p=(\phi_p,\gamma_p)$ such that $\phi_p\in\prod_{\alpha\in\kappa}\sqb[big]{{}^{B_\alpha}2}^{\leq\card\alpha}$ and $\gamma_p\in\kappa$. If $p,q\in\LockB$, then we define the ordering $q\leq p$ if and only if $\phi_q\restriction \gamma_p=\phi_p\restriction\gamma_p$ and $\gamma_q\geq \gamma_p$ and for each $\alpha\in\kappa\setminus\gamma_p$ we have $\phi_q(\alpha)\supset\phi_p(\alpha)$.
	\end{dfn}
	
	\begin{lmm}
		If $B\in\BPklim$, then $\LockB$ is strategically ${<}\kappa$-closed  and $(\kappa,{<}\lambda)$-centred for $\lambda<\kappa$.
	\end{lmm}
	
	\begin{proof}
		For each $\alpha\in\kappa$, let $\gamma_\alpha$ be minimal such that $\gamma_\alpha\geq\card{\alpha}$ and  $\card{B_\xi}\geq \alpha$ for all $\xi\geq \gamma_\alpha$. For strategic closure, it suffices to note that $\set{p\in\LockB}{\gamma\geq\gamma_\lambda}$ is ${<}\lambda$-closed. For centredness, if $\phi\in\SkB$ and $s=\phi\restriction \gamma_\lambda$, then $\set{p\in\LockB}{\phi_p\restriction\gamma_\lambda=s}$ is ${<}\lambda$-centred.
	\end{proof}
	
	Let $\bf V$ be a model of $\sf{ZFC}$ containing $B\in\BPklim$ and let $G$ be $\LockB$-generic over $\bf V$, then we may define $\phi_\rm{gen}:=\Cup_{p\in G}(\phi_p\restriction\gamma_p)$ and note that $\phi_\rm{gen}\in\prod_{\alpha\in\kappa}\sqb[big]{{}^{B_\alpha}2}^{\leq\card\alpha}$. Moreover, by genericity it is easy to see that $(\kacs)^\bf V\subset\sqb{\phi_\rm{gen}}_*$. As such, we may iterate, with ${<}\kappa$-support, forcing notions of the form $\LockB$, using bookkeeping to ascertain that every $B\in\BPklim$ (old or new) is used as a parameter cofinally often. By the following lemma, such iterations do not collapse cardinals.
	
	\begin{lmm}
		Let $\bb P_\mu=\abset{\bb P_\alpha,\dot{\bb Q}_\alpha}{\alpha\in\mu}$ be a forcing iteration with ${<}\kappa$-support such that for each $\alpha\in\mu$ there is some $\bb P_\alpha$-name $\dot B^\alpha$ such that $\fc_{\bb P_\alpha}\ap{\dot B^\alpha\in\BPklim\text{ and }\dot{\bb Q}_\alpha=\Loc_\kappa(\dot B^\alpha)}$. Then $\bb P_\mu$ is strategically ${<}\kappa$-closed and ${<}\kappa^+$-c.c.
	\end{lmm}
	
	\begin{proof}
		That strategic ${<}\kappa$-closure is preserved under ${<}\kappa$-support iteration is well-known (see \cite[Proposition~A.1.6]{RoslanowskiShelah} for a proof). For ${<}\kappa^+$-c.c., note that there is a dense set $D$ of conditions $p\in\bb P_\mu$ such that for each $\alpha\in\supp(p)$ there exists $s_\alpha\in\bf V$ for which $p\restriction\alpha\fc\ap{\phi_{p(\alpha)}\restriction\gamma_{p(\alpha)}=s_\alpha}$. Any two conditions $p,p'\in\LockB$ with the same stem $\phi_p\restriction\gamma_p=\phi_{p'}\restriction\gamma_{p'}$ are compatible, hence we can apply the sunflower lemma (i.e., the $\Delta$-system lemma).
	\end{proof}
	
	\begin{thm}\label{thm:simple loc model}
		Let $\bf V\md\ap{\sf{GCH}}$ and $\mu>\kappa$ be regular, then there exists a cardinal-preserving ${<}\kappa$-support iteration $\bb P_\mu$ of length $\mu$, where each of the iterands is (forced to be) of the form $\LockB$, such that $\fc_{\bb P_\mu}\ap{\non(\NAk)=\mu}$.
	\end{thm}
	\begin{proof}
		In the extension by $\bb P_\mu$, let $X\subset \kacs$ with $\card X<\mu$ and let $B\in\BPklim$. Then $X$ has a $\bb P_\alpha$-name for some $\alpha<\mu$, and by bookkeeping there is some $\beta>\alpha$ such that $\bb P_\beta\fc\ap{\dot{\bb Q}_\beta=\LockB}$. Then the $\dot{\bb Q}_\beta$-generic adds $\phi\in\prod_{\alpha\in\kappa}\sqb[big]{{}^{B_\alpha}2}^{\leq\card\alpha}$ such that $X\subset(\kacs)^{\bf V^{\bb P_\beta}}\subset\sqb \phi_*$. Since $B$ was arbitrary, we see that $X\in\NAk$.
	\end{proof}
	
	As a corollary (using \cref{thm:NAk addEkI}), if $\kappa$ is strongly inaccessible, then the above model satisfies
	\[
	\forall B\in\BPklim\,(\add(\BEkA)=\add(\EkA)=\add(\Ek(B))=\mu),
	\]
	and for arbitrary regular uncountable $\kappa$, the model satisfies $\non(\Ek(B))=\mu$ for each $B\in\BPklim$.
	
	It is not hard to see that $\LockB$ adds a $\kappa$-Cohen generic; for instance, in $\bf V$, let $\prec_{\alpha}$ well-order ${}^{B_\alpha}2$ and define $c\in\kacs$ by  $c(\alpha)=0$ if and only if $\ot(\phi_\rm{gen}(\alpha),\prec_{\alpha})$ is even, then $c$ is $\kappa$-Cohen generic.
	
	\begin{lmm}
		If $\kappa$ is weakly compact and $B\in\BPklim$, then $\LockB$ is $(\kappa,\kappa)$-calibre. 
	\end{lmm}
	\begin{proof}
		Fix any condition $p^*\in\LockB$, let us write $s:=\phi_{p^*}\restriction\gamma_{p^*}$, then we claim that the set $C:=\set{p\in\bb P}{\phi_p\restriction\gamma_p=s}$ has $\kappa$-calibre. We assume for convenience that $\card{B_\alpha}\geq\alpha$, the case for arbitrary $B\in\BPklim$ is proved similarly. 
		
		Fix some sequence of distinct conditions $\abset{p_\xi}{\xi\in\kappa}$ in $C$ and note that $\gamma_{p_\xi}=\gamma_{p^*}$ for each $\xi$. Define a tree $T$ where $t\in T$ if and only if there are $\xi\neq \xi'$ and $\gamma\geq \gamma_{p^*}$ with $t=\phi_{p_\xi}\restriction \gamma= \phi_{p_{\xi'}}\restriction\gamma$. 
		Since $\kappa$ is inaccessible, we see that $\card[m]{\set[m]{\phi\restriction\gamma}{\phi\in\prod_{\alpha\in\kappa}\sqb[big]{{}^{B_\alpha}2}^{\leq\card\alpha}}}<\kappa$ for each $\gamma\in\kappa$, and thus it follows that $T$ is a ${<}\kappa$-branching tree of height $\kappa$. By weak compactness of $\kappa$, we then see that $T$ has a branch $b\subset T$, and therefore $\psi:=\Cup b\in\prod_{\alpha\in\kappa}\sqb[big]{{}^{B_\alpha}2}^{\leq\card\alpha}$. Now, fix a strictly increasing sequence $\abset[big]{\xi(\gamma)}{\gamma\in\iv{\gamma_{p^*}}\kappa}$ such that $\phi_{p_{\xi(\gamma)}}\restriction\gamma=\psi\restriction\gamma$ for each $\gamma$, then if $\phi'\colon\alpha\mapsto\Cup_{\gamma\in\iv{\gamma_{p^*}}\kappa}\phi_{p_{\xi(\gamma)}}(\alpha)$ we see that $\phi'\in\prod_{\alpha\in\kappa}\sqb[big]{{}^{B_\alpha}2}^{\leq\card\alpha}$ and $(\phi',\gamma_{p^*})\leq p_{\xi(\gamma)}$ for all $\gamma\in\iv{\gamma_{p^*}}\kappa$, as required.
	\end{proof}
	
	\begin{crl}
		If $\kappa$ is weakly compact and $B\in\BPklim$, then $\LockB$ does not add a dominating $\kappa$-real.\qedhere
	\end{crl}
	
	We do not know if the property of not adding dominating $\kappa$-reals is preserved in the iteration of \cref{thm:simple loc model}.

	\subsubsection{Adding a new slalom through a new partition} The second forcing notion $\Lockn$ is similar to $\LockB$, but it not only adds a new slalom $\phi$, but also a new partition $I\in\IPklim$ for which $\phi\in\SkI$.
	
	\begin{dfn}
		Every condition of $\Lockn$ is a triple $p=(s_p,I^p,F_p)$ such that
		\begin{enumerate}
			\item $s_p$ and $I^p$  are functions with $\dom(s_p)=\dom(I^p)\in\kappa$ and we write $I^p_\alpha:=I^p(\alpha)$,
			\item if $\alpha\in\dom(s_p)$, then $I^p_\alpha\subset\kappa$ is a nonempty interval with $\min(\kappa\setminus\Cup_{\xi<\alpha}I^p_\xi)=\min(I^p_\alpha)$ and $\sup(I^p_\alpha)<\kappa$,
			\item if $\alpha\in\dom(s_p)$, then $s_p(\alpha)\subset{}^{I^p_\alpha}2$ and $0<\card{s_p(\alpha)}<2^{\card{I^p_\alpha}}$,
			\item $F_p\subset\kacs$ is a set of functions such that $\card{F_p}<\kappa$.
		\end{enumerate}
		If $p,q\in\Loc_\kappa^\rm{new}$, then we let $q\leq p$ if and only if $s_p\subset s_q$ and $I^p\subset I^q$ and $F_p\subset F_q$, and for any $\alpha\in\dom(s_q\setminus s_p)$ and $f\in F_p$ we have $f\restriction I^q_\alpha\in s_q(\alpha)$.
	\end{dfn}
	
	If $\bf V$ is a model of $\sf{ZFC}$ and $G$ is $\Lockn$-generic over $\bf V$, then $I^\rm{gen}:=\Cup_{p\in G}I^p\in\IPklim$ and $\phi_\rm{gen}:=\Cup_{p\in G}s_p\in\Sigma_\kappa(I^\rm{gen})$. By genericity, $(\kacs)^\bf V\subset\sqb{\phi_\rm{gen}}_*$. Compared to $\LockB$, the forcing notion $\Lockn$ has stronger closure and linkedness properties:
	\begin{lmm}
		$\Loc_\kappa^\rm{new}$ is ${<}\kappa$-closed and $(\kappa,{<}\kappa)$-centred.\qed
	\end{lmm} 
	\begin{lmm}\label{lmm:centred preservation}
		Let $\bb P_\mu=\abset{\bb P_\alpha,\dot{\bb Q}_\alpha}{\alpha\in\mu}$ be a forcing iteration with ${<}\kappa$-support such that for each $\alpha\in\mu$ we have $\fc_{\bb P_\alpha}\ap{\dot{\bb Q}_\alpha=\Lockn}$. Then $\bb P_\mu$ is ${<}\kappa$-closed,  and if $\mu<(2^\kappa)^+$, then $\bb P_\mu$ is $(\kappa,{<}\kappa)$-centred.
	\end{lmm}
	\begin{proof}
		The preservation of $(\kappa,{<}\kappa)$-centredness follows as in \cite[Lemma~55]{BrendleBrookeTaylorFriedmanMontoya}. 
	\end{proof}
	\begin{thm}
		Let $\bf V\md\ap{\sf{GCH}}$ and let $\mu>\kappa$ be regular. Let $\bb C$ be the $\mu$-length ${<}\kappa$-support product of $\kappa$-Cohen forcing and let $\dot{\bb P}_\mu$ be a $\bb C$-name for the $\mu$-length ${<}\kappa$-support iteration such that for each $\alpha\in\mu$ we have $\fc_{\bb C\ast\dot{\bb P}_\alpha}\ap{\dot{\bb Q}_\alpha=\Lockn}$. Then 
		\[\fc_{\bb C\ast\dot{\bb P}_\mu}\ap{\forall B\in\BPklim(\kappa^+=\non(\Ek(B))<\cov(\Mk)=\non(\Ek)=\mu=2^\kappa)}.\]
	\end{thm}
	\begin{proof}
		Let $H$ be $\bb C$-generic and note that $\bf V[H]\md\ap{\kappa^+=\non(\Mk)<\cov(\Mk)=\mu=2^\kappa}$. It then follows from \cref{lmm:centred preservation} that $\bb P_\mu:=(\dot{\bb P}_\mu)^{[H]}$ is $(\kappa,{<}\kappa)$-centred in $\bf V[H]$. \cite[Definition 65]{vanderVlugt2025} defines a cardinal characteristic $\fr{sup}_\kappa({\neqi})$, and from \cite[Lemma 79]{vanderVlugt2025} we know that $(\kappa,{<}\kappa)$-centred forcing notions preserve \mbox{$\fr{sup}_\kappa({\neqi})=\kappa^+$}. Since $\bins{b,b}\leq \fr{sup}_\kappa({\neqi})\leq \non(\Mk)$, we may use \cref{cov Ek as a localisation number} to conclude that $\fc_{\bb C\ast\dot{\bb P}_\mu}\ap{\forall B\in\BPklim(\non(\Ek(B))=\kappa^+)}$.
		
		Now, working in $\bf V[H]$, suppose that $\dot X$ is $\bb P_\mu$-name for a subset of $\kacs$ of size ${<}\mu$, then there is a stage $\gamma<\mu$ such that $\dot X$ is a $\bb P_\gamma$-name. But, as we discussed above, in the next stage of the iteration a slalom $\phi_\rm{gen}$ is added such that $\dot X\subset\sqb{\phi_\rm{gen}}_*$, thus $\fc_{\bb P_\mu}\ap{\dot X\in\EkE}$. Consequently, $\fc_{\bb C\ast\dot{\bb P}_\mu}\ap{\non(\EkE)=\mu}$ and thus the theorem follows from $\non(\EkE)\leq\non(\Ek)$.
	\end{proof}

	It is again easy to see that $\Lockn$ adds a $\kappa$-Cohen generic. 
	Moreover, if we assume that $\kappa$ is weakly compact, then $\Loc_\kappa^\rm{new}$ does not add dominating $\kappa$-reals, but we do not know how to preserve this property under iteration, thus we will omit the details.

	\subsection{The Laver property}
	
	We make a remark about the Laver property. Classically, a forcing notion with the Laver property can be iterated with countable support to obtain a model in which $\cov(\Ek)$ is small, see for instance \cite[Lemma 7.2.3]{BartoszynskiJudah}. The same argument works in the higher context as well. One well-known higher forcing notion that has a higher Laver property, is the $\kappa$-Sacks forcing of \cite{Kanamori80} as well as its iteration under ${\leq}\kappa$-support, see also \cite[\S\,5.1]{BrendleBrookeTaylorFriedmanMontoya} and \cite[Proposition 1.86]{MontoyaThesis}. We may therefore conclude that if $\kappa$ is strongly inaccessible and we assume $\sf{GCH}$, then the $\kappa^{++}$-length iteration with ${\leq}\kappa$-support of the $\kappa$-Sacks forcing will result in a model where $\cov(\Ek)=\kappa^+<2^\kappa$, and moreover, as shown in \cite{BrendleBrookeTaylorFriedmanMontoya}, the resulting model additionally satisfies $\cof(\Mk)=\kappa^+$.

	\AtNextBibliography{\footnotesize}
	
	\color{black}
	\printbibliography

@misc{AgostiniBarreraDimonte,
	author = {Agostini, Claudio and Barrera, Fernando and Dimonte, Vincenzo},
	title = {On the problem of generalized measures: an impossibility result},
	year = {2026},
	eprint = {2601.13321},
	archiveprefix = {arXiv},
	note = {Preprint},
}

@article{BartoszynskiLocalisation,
	author = {Bartoszy{\'n}ski, Tomek},
	title = {Combinatorial aspects of measure and category},
	journal = {Fundam.\ Math.},
	volume = {127},
	number = {3},
	year = {1987},
	pages = {225-239},
	doi = {10.4064/fm-127-3-225-239},
}

@book{BartoszynskiJudah,
	author = {Bartoszy{\'n}ski, Tomek and Judah, Haim},
	title = {Set {T}heory: {O}n the {S}tructure of the {R}eal {L}ine},
	year = {1995},
	publisher = {A.K. Peters},
	address = {Wellesley, MA},
}

@article{BartoszynskiShelah,
	author = {Bartoszyński, Tomek and Shelah, Saharon},
	title = {Closed measure zero sets},
	volume = {58},
	number = {2},
	pages = {93--110},
	rights = {https://www.elsevier.com/tdm/userlicense/1.0/},
	url = {https://linkinghub.elsevier.com/retrieve/pii/016800729290001G},
	doi = {10.1016/0168-0072(92)90001-G},
	journaltitle = {Ann.\ Pure Appl.\ Log.},
	shortjournal = {Ann.\ Pure Appl.\ Log.},
	urldate = {2025-10-27},
	date = {1992-09},
}

@misc{SwitzerBenhamou,
	author = {Benhamou, Tom and Switzer, Corey Bacal},
	title = {Eventual Capture on a Measurable Cardinal},
	year = {2025},
	eprint = {2511.07411},
	archiveprefix = {arXiv},
	note = {Preprint},
}

@article{Blass1994,
	author = {Blass, A.},
	title = {Cardinal Characteristics and the Product of Countably Many Infinite Cyclic Groups},
	journal = {J.\ Algebra},
	volume = {169},
	number = {2},
	year = {1994},
	pages = {512–540},
	publisher = {Elsevier BV},
	url = {http://dx.doi.org/10.1006/jabr.1994.1295},
	doi = {10.1006/jabr.1994.1295},
}

@online{BlassHyttinenZhang,
	author = {Blass, Andreas and Hyttinen, Tapani and Zhang, Yi},
	title = {Mad families and their neighbors},
	note = {Preprint},
	url = {https://dept.math.lsa.umich.edu/~ablass/madfam.pdf},
}

@article{BrendleEvasionI,
	author = {Brendle, Jörg},
	title = {Evasion and prediction - the Specker phenomenon and Gross spaces},
	journal = {Forum Math.},
	volume = {7},
	year = {1995},
	pages = {513--542},
	url = {https://doi.org/10.1515/form.1995.7.513},
	doi = {doi:10.1515/form.1995.7.513},
	lastchecked = {2026-08-01},
}

@article{BrendleBrookeTaylorFriedmanMontoya,
	author = {Brendle, J{\"o}rg and Brooke-Taylor, Andrew and Friedman, Sy-David and Montoya, Diana Carolina},
	title = {Cicho{\'n}’s diagram for uncountable cardinals},
	journal = {Isr.\ J.\ Math.},
	volume = {225},
	number = {2},
	year = {2018},
	pages = {959--1010},
	publisher = {Springer},
	doi = {10.1007/s11856-018-1688-y},
}

@article{CardonaMartonSupina2025,
	author = {Cardona, Miguel A. and Marton, Adam and Supina, Jaroslav},
	title = {Cardinal characteristics associated with small subsets of reals},
	journal = {Fundam.\ Math.},
	volume = {274},
	number = {3},
	year = {2026},
	doi = {10.4064/fm250317-14-3},
}

@article{ChapmanSchurz2024,
	author = {Chapman, Nick Steven and Sch\"{u}rz, Johannes Philipp},
	title = {Strong measure zero sets on $2^\kappa$ for $\kappa$ inaccessible},
	journal = {J.\ Symb.\ Log.},
	volume = {90},
	number = {3},
	year = {2024},
	pages = {1277–1307},
	publisher = {Cambridge University Press (CUP)},
	url = {http://dx.doi.org/10.1017/jsl.2023.100},
	doi = {10.1017/jsl.2023.100},
}

@article{CohenShelah,
	author = {Cohen, Shani and Shelah, Saharon},
	title = {Generalizing random real forcing for inaccessible cardinals},
	journal = {Isr.\ J.\ Math.},
	volume = {234},
	year = {2019},
	pages = {547-580},
	doi = {10.1007/s11856-019-1925-z},
}

@article{CummingsShelah,
	author = {Cummings, James and Shelah, Saharon},
	title = {Cardinal invariants above the continuum},
	journal = {Ann.\ Pure Appl.\ Log.},
	volume = {75},
	year = {1995},
	pages = {251-268},
	doi = {10.1016/0168-0072(95)00003-Y},
}

@article{FriedmanLaguzzi,
	author = {Friedman, Sy-David and Laguzzi, Giorgio},
	title = {A null ideal for inaccessibles},
	journal = {Arch.\ Math.\ Log.},
	volume = {56},
	year = {2017},
	pages = {691-697},
	doi = {10.1007/s00153-017-0562-7},
}

@article{GalvinMycielskiSolovay1979,
	author = {Galvin, Fred and Mycielski, Jan and Solovay, Robert M.},
	title = {Strong measure zero sets},
	journal = {Not.\ Am.\ Math.\ Soc.},
	volume = {26},
	year = {1979},
	pages = {p.\ A-280},
}

@article{GalvinMycielskiSolovay2017,
	author = {Galvin, Fred and Mycielski, Jan and Solovay, Robert M.},
	title = {Strong measure zero and infinite games},
	journal = {Arch.\ Math.\ Log.},
	volume = {56},
	number = {7–8},
	year = {2017},
	pages = {725–732},
	publisher = {Springer Science and Business Media LLC},
	url = {http://dx.doi.org/10.1007/s00153-017-0541-z},
	doi = {10.1007/s00153-017-0541-z},
}

@article{GoldsternShelah,
	author = {Goldstern, Martin and Shelah, Saharon},
	title = {Many simple cardinal invariants},
	journal = {Arch.\ Math.\ Log.},
	volume = {32},
	number = {3},
	year = {1993},
	pages = {203--221},
	publisher = {Springer},
	doi = {10.1007/BF01375552},
}

@book{Halko1996,
	author = {Halko, Aapo},
	title = {Negligible subsets of the generalized Baire space $\omega_1^{\omega_1}$},
	journal = {Ann.\ Acad.\ Sci.\ Fenn.\ Math.},
	year = {1996},
	publisher = {The Finnish Mathematical Society},
	url = {http://dx.doi.org/10.5186/aasfmd.1996.107},
	doi = {10.5186/aasfmd.1996.107},
}

@article{HalkoShelah2001,
	author = {Halko, Aapo and Shelah, Saharon},
	title = {On strong measure zero subsets of ${}^\kappa2$},
	journal = {Fundam.\ Math.},
	volume = {170},
	number = {3},
	year = {2001},
	pages = {219–229},
	publisher = {Institute of Mathematics, Polish Academy of Sciences},
	url = {http://dx.doi.org/10.4064/fm170-3-1},
	doi = {10.4064/fm170-3-1},
}

@article{Kanamori80,
	author = {Kanamori, Akihiro},
	title = {Perfect-set forcing for uncountable cardinals},
	journal = {Ann.\ Math.\ Log.},
	volume = {19},
	number = {1-2},
	year = {1980},
	pages = {97–114},
	doi = {10.1016/0003-4843(80)90021-2},
}

@article{Kellner,
	author = {Kellner, Jakob},
	title = {Even More Simple Cardinal Invariants},
	journal = {Arch.\ Math.\ Log.},
	volume = {47},
	number = {5},
	year = {2008},
	pages = {503--515},
	publisher = {Springer},
	doi = {10.1007/s00153-008-0094-2},
}

@article{Miller,
	author = {Miller, Arnold W.},
	title = {Some Properties of Measure and Category},
	journal = {Trans.\ Am.\ Math.\ Soc.},
	volume = {266},
	number = {1},
	year = {1981},
	pages = {93--114},
	publisher = {American Mathematical Society},
	doi = {10.1090/S0002-9947-1981-0613787-2},
}

@thesis{MontoyaThesis,
	author = {Montoya Amaya, Diana Carolina},
	title = {Some cardinal invariants of the generalized Baire spaces},
	year = {2017},
	school = {Universität Wien},
	type = {PhD thesis},
}

@article{RaghavanShelah2017,
	author = {Raghavan, Dilip and Shelah, Saharon},
	title = {Two inequalities between cardinal invariants},
	journal = {Fundam.\ Math.},
	volume = {237},
	number = {2},
	year = {2017},
	pages = {187–200},
	publisher = {Institute of Mathematics, Polish Academy of Sciences},
	url = {http://dx.doi.org/10.4064/fm253-7-2016},
	doi = {10.4064/fm253-7-2016},
}

@misc{RoslanowskiLoc,
	author = {Roslanowski, Andrzej},
	title = {\.Zaba: a case for clubs},
	year = {2026},
	eprint = {2607.23923},
	archiveprefix = {arXiv},
	note = {Preprint},
}

@article{RoslanowskiShelah,
	author = {Ros{\l}anowski, Andrzej and Shelah, Saharon},
	title = {Sheva-sheva-sheva: Large creatures},
	journal = {Isr.\ J.\ Math.},
	volume = {159},
	number = {1},
	year = {2007},
	pages = {109–174},
	doi = {10.1007/s11856-007-0040-8},
}

@article{Schurz2023,
	author = {Sch\"{u}rz, Johannes Philipp},
	title = {The cofinality of the strong measure zero ideal for $\kappa$ inaccessible},
	journal = {Math.\ Log.\ Q.},
	volume = {69},
	number = {1},
	year = {2023},
	pages = {31–39},
	publisher = {Wiley},
	url = {http://dx.doi.org/10.1002/malq.202000093},
	doi = {10.1002/malq.202000093},
}

@article{Shelah1995,
	author = {Shelah, Saharon},
	title = {Every null-additive set is meager-additive},
	journal = {Isr.\ J.\ Math.},
	volume = {89},
	number = {1–3},
	year = {1995},
	pages = {357–376},
	publisher = {Springer Science and Business Media LLC},
	url = {http://dx.doi.org/10.1007/BF02808209},
	doi = {10.1007/bf02808209},
}

@article{ShelahNull,
	author = {Shelah, Saharon},
	title = {A parallel to the null ideal for inaccessible {$\lambda$}: {P}art {I}},
	journal = {Arch.\ Math.\ Log.},
	volume = {56},
	number = {3-4},
	year = {2017},
	pages = {319--383},
	doi = {10.1007/s00153-017-0524-0},
	url = {https://doi.org/10.1007/s00153-017-0524-0},
}

@article{Suzuki1998,
	author = {Suzuki, Toshio},
	title = {About splitting numbers},
	journal = {Proc.\ Jpn.\ Acad., Ser.\ A, Math.\ Sci.},
	volume = {74},
	number = {2},
	year = {1998},
	publisher = {Project Euclid},
	url = {http://dx.doi.org/10.3792/pjaa.74.33},
	doi = {10.3792/pjaa.74.33},
}

@article{vanderVlugt2025,
	author = {van der Vlugt, Tristan},
	title = {Cardinal characteristics on bounded generalised {Baire} spaces},
	journal = {Ann.\ Pure Appl.\ Log.},
	volume = {176},
	number = {7},
	year = {2025},
	pages = {103582},
	publisher = {Elsevier BV},
	url = {http://dx.doi.org/10.1016/j.apal.2025.103582},
	doi = {10.1016/j.apal.2025.103582},
}

@phdthesis{Wohofsky,
	author = {Wohofsky, Wolfgang},
	title = {Special sets of real numbers and variants of the Borel Conjecture},
	year = {2013},
	school = {TU Wien},
}

@article{Wontner,
	author = {Wontner, Ned},
	title = {Views from a peak: {G}eneralisations and descriptive set theory},
	journal = {Doctoral dissertation, University of Amsterdam},
	year = {2023},
}

@article{Zapletal,
	author = {Zapletal, Jindřich},
	title = {Splitting Number at Uncountable Cardinals},
	journal = {J.\ Symb.\ Log.},
	volume = {62},
	number = {1},
	year = {1997},
	pages = {35--42},
	publisher = {Association for Symbolic Logic},
	urldate = {2024-04-19},
	doi = {10.2307/2275730},
}

@article{Zindulka2019,
	author = {Zindulka, Ond\v{r}ej},
	title = {Strong measure zero and meager-additive sets through the prism of fractal measures},
	journal = {Commentationes Math.\ Univ.\ Carolinae},
	volume = {60},
	number = {1},
	year = {2019},
	pages = {131–155},
	publisher = {Charles University in Prague, Karolinum Press},
	url = {http://dx.doi.org/10.14712/1213-7243.2015.277},
	doi = {10.14712/1213-7243.2015.277},
}
		
	\end{document}